\documentclass[11pt,lqno]{amsart}
\pdfoutput=1
\usepackage{latexsym,amsthm,amssymb,amsmath,amscd,mathtools,hyperref,xpatch}

\mathtoolsset{showonlyrefs}
\usepackage{tikz-cd}
\tikzcdset{
  cells={font=\everymath\expandafter{\the\everymath\displaystyle}},
}
\usetikzlibrary{matrix,arrows,decorations.pathmorphing}
\usepackage{bm}
\usepackage{fancyhdr}
\usepackage{mathtools}

\usepackage{tcolorbox}

\usepackage[normalem]{ulem}

\newcommand{\nospacepunct}[1]{\makebox[0pt][l]{\,#1}}%https://tex.stackexchange.com/questions/187444/interpunction-and-tikz-diagrams to fix diagram spacing where periods appear. T

\usepackage[hmargin=2cm,vmargin=1.5cm]{geometry}

\usepackage{graphicx}

\usepackage[shortlabels]{enumitem}
\setlist[enumerate]{nosep}

\usepackage{xcolor}
\newcommand\redsout{\bgroup\markoverwith{\textcolor{red}{\rule[0.5ex]{2pt}{0.8pt}}}\ULon}

\newtheorem{theorem}{Theorem}[section]

\newtheorem*{theorem*}{Theorem}
\newtheorem{lemma}[theorem]{Lemma}
\newtheorem{corollary}[theorem]{Corollary}
\newtheorem{proposition}[theorem]{Proposition}

\newtheorem{remark}[theorem]{Remark}
\newtheorem{definition}[theorem]{Definition}

\newtheorem*{question*}{Question}

\newtheorem*{questions*}{Questions}

\newcommand{\cE}{\mathcal E}
\newcommand{\cF}{\mathcal F}
\newcommand{\cG}{\mathcal G}
\newcommand{\cH}{\mathcal H}

\newcommand{\cN}{\mathcal N}
\newcommand{\cO}{\mathcal O}
\newcommand{\cP}{\mathcal P}
\newcommand{\cQ}{\mathcal Q}
\newcommand{\cR}{\mathcal R}

\newcommand{\cT}{\mathcal T}
\newcommand{\cU}{\mathcal U}

\newcommand{\cX}{\mathcal X}
\newcommand{\cZ}{\mathcal Z}

\def\Az{\mathbb{A}}
\def\Bz{\mathbb{B}}
\def\Cz{\mathbb{C}}

\def\Kz{\mathbb{K}}

\def\Nz{\mathbb{N}}

\def\Iz{\mathbb{I}}
\def\Oz{\mathbb{O}}

\def\Qz{\mathbb{Q}}
\def\Rz{\mathbb{R}}
\def\Tz{\mathbb{T}}
\def\Zz{\mathbb{Z}}

\def\1z{\mathbb{1}}
\newcommand{\fA}{\mathfrak A}
\newcommand{\fB}{\mathfrak B}
\newcommand{\fC}{\mathfrak C}

\newcommand{\fR}{\mathfrak R}
\newcommand{\fS}{\mathfrak S}

\def\SEMI{\mbox{$\times\kern-2pt\vrule height5pt width.6pt \kern3pt $}}

\newcommand{\Hom}{\textup{Hom}}

\newcommand{\Aut}{\textup{Aut}}

\newcommand{\Coker}{{\rm Coker}\,}
\newcommand{\id}{{\rm id}}

\newcommand{\Ind}{\mathrm{ Ind}\,}

\newcommand{\ev}{\operatorname{ev}} % Auswertung

\newcommand{\fa}{\text{ for all }} % fuer alle
\newcommand{\p}{\mathfrak{p}}
\newcommand{\q}{\mathfrak{q}}

\newcommand{\Cl}{\textup{Cl}}
\newcommand{\Gal}{\textup{Gal}}
\newcommand{\Prim}{\textup{Prim}}

\newcommand{\acts}{\curvearrowright} 

\newcommand{\fg}[1]{[\![#1]\!]}

\newcommand{\spn}{\textup{span}}
\renewcommand{\phi}{\varphi}

\renewcommand{\O}{\mathcal{O}}

\newcommand{\ol}{\overline}
\newcommand{\tors}{\textup{tors}}

\newcommand{\Ext}{\textup{Ext}}

\newcommand{\lwedge}{{\textstyle\bigwedge}} %another bigwedge
\newcommand{\val}{\mathrm{val}}
\newcommand{\fval}{\mathfrak{val}}
\DeclareMathOperator{\rank}{rank}
\newcommand{\ho}[1]{\,\hat{\otimes}_{#1}}

\newcommand{\cExt}{\mathbf{Ext}}
\newcommand{\KK}{\textup{KK}}
\newcommand{\K}{\textup{K}}
\newcommand{\Arr}{\textup{Arr}}
\newcommand{\bs}[1]{\boldsymbol{#1}}
\DeclareMathOperator{\Interior}{Int}
\DeclareMathOperator{\Ker}{Ker}
\renewcommand{\Im}{\textup{Im}\,}
\DeclareMathOperator{\Homeo}{Homeo}

\begin{document}

\title[Constructing number field isomorphisms from C*-algebras]{Constructing number field isomorphisms from *-isomorphisms of certain crossed product C*-algebras}

\thispagestyle{fancy}

\author{Chris Bruce}\thanks{C. Bruce is supported by a Banting Fellowship administered by the Natural Sciences and Engineering Research Council of Canada (NSERC). This project has received funding from the European Research Council (ERC) under the European
Union’s Horizon 2020 research and innovation programme (grant agreement No. 817597).}

\address[Chris Bruce]{School of Mathematical Sciences, Queen Mary University of London, Mile End Road, E1 4NS London, United Kingdom, and
	School of Mathematics and Statistics, University of Glasgow, University Place, Glasgow G12 8QQ, United Kingdom}
\email[Bruce]{Chris.Bruce@glasgow.ac.uk}

\author{Takuya Takeishi}\thanks{T. Takeishi is supported by JSPS KAKENHI Grant Number JP19K14551.}

\address[Takuya Takeishi]{
	Faculty of Arts and Sciences\\
	Kyoto Institute of Technology\\
	Matsugasaki, Sakyo-ku, Kyoto\\
	Japan}
\email[Takeishi]{takeishi@kit.ac.jp}

\date{\today}


\maketitle

\begin{abstract}
We prove that the class of crossed product C*-algebras associated with the action of the multiplicative group of a number field on its ring of finite adeles is rigid in the following explicit sense: Given any *-isomorphism between two such C*-algebras, we construct an isomorphism between the underlying number fields.
As an application, we prove an analogue of the Neukirch--Uchida theorem using topological full groups, which gives a new class of discrete groups associated with number fields whose abstract isomorphism class completely characterises the number field.
\end{abstract}

\setlength{\parindent}{0cm} \setlength{\parskip}{0.5cm}

\section{Introduction} 

\subsection{Context} 
The study of C*-algebras of number-theoretic origin was initiated by Bost and Connes three decades ago \cite{BC}.
The Bost--Connes C*-algebra for $\Qz$ carries a canonical time evolution, and it was shown in \cite{BC} that the associated C*-dynamical system exhibits several remarkable properties related to the explicit class field theory of $\Qz$. With Hilbert's 12th problem in mind, the initial objective was to find an appropriate notion of Bost--Connes C*-algebra for a general number field. Such a C*-algebra should carry a canonical time evolution, and the associated C*-dynamical system should reflect the class field theory of the number field (see \cite[Problem~1.1]{CMR} for the precise formulation of this problem). Several C*-algebras were proposed early on \cite{ALR,Cohen,HL,LF}. However, the C*-dynamical systems from these constructions did not exhibit the desired phase-transition phenomena in general. A construction was given by Connes, Marcolli, and Ramachandran \cite{CMR} for imaginary quadratic fields that did satisfy all the desired criteria, and then a construction of a Bost--Connes type C*-algebra associated with a general number field was given by Ha and Paugam \cite{HaPaugam} (in greater generality) and by Laca, Larsen, and Neshveyev \cite{LLN}. The C*-algebras defined in \cite{LLN} are now considered the standard Bost--Connes C*-algebras. The KMS-structure was computed in \cite{LLN}, but the existence of an arithmetic subalgebra---the remaining property from \cite[Problem~1.1]{CMR}---was only proven later by Yalkinoglu \cite{Yalk} (since the construction in \cite{Yalk} relies on results from class field theory, it did not shed light on Hilbert's 12th problem). The study of Bost--Connes C*-algebras has led to several recent purely number-theoretic results \cite{CdSLMS,CLMS,Smit}.

Another avenue for constructing C*-algebras from number fields is via ring C*-algebras. They were first defined by Cuntz for the ring $\Zz$ \cite{Cun} and then for general rings of integers in number fields by Cuntz and Li \cite{CL:10} (see also \cite{Li:ring}). Inspired by \cite{Cun}, Laca and Raeburn considered a semigroup C*-algebra over $\Zz$ \cite{LR:10}, and a modification of this construction was introduced shortly after for general rings of integers by Cuntz, Deninger, and Laca \cite{CDL}: Given a ring of integers in a number field, one considers the semigroup C*-algebra of the $ax+b$-semigroup over the ring, which is a natural extension of the ring C*-algebra. Such semigroup C*-algebras have received a great deal of attention over the last decade: They provided a fundamental example class for the development of Li's theory of semigroup C*-algebras \cite{Li:semigp1,Li:semigp2}; the problem of computing their \K-theory was a driver for general results on \K-theory by Cuntz, Echterhoff, and Li \cite{CEL1,CEL2} and Li \cite{Li:CMP}; and the study of their KMS-structre led to general results by Neshveyev on KMS states \cite{Nesh}.

Initially, the focus in each of these avenues of investigation was on internal structure and KMS states of the C*-algebra in question for a fixed number field. The problem of rigidity---that is, of comparing the C*-algebraic data arising from two different number fields---was first
considered by Cornelissen and Marcolli who considered the Bost--Connes C*-dynamical systems \cite{CM2}.
Not long after, it was proven by Li and L\"{u}ck \cite{LiLuck}---building on earlier work by Cuntz and Li \cite{CL:11}---that the ring C*-algebras associated with rings of integers have no rigidity in the sense that the ring C*-algebra of a ring of integers does not depend, up to isomorphism, on the ring of integers.
Li then established, under some technical assumptions, that the C*-algebra of an $ax+b$-semigroup over a ring of integers remembers the Dedekind zeta function of the number field \cite{Li:Kt1}; Li also showed that if one keeps track of the canonical Cartan subalgebra, then the ideal class group of the number field can be recovered \cite{Li:Kt2}. The technical assumption in \cite{Li:Kt1} was removed later in work by the first-named author and Li \cite{BruceLi1}. As a consequence of these results, one obtains a rigidity theorem for number fields that are Galois over $\Qz$: If the semigroup C*-algebras of the $ax+b$-semigroups from two rings of integers are isomorphic, and one of the number fields is Galois, then the fields must be isomorphic. However, it is not clear how such an isomorphism of fields is related to the initial *-isomorphism of C*-algebras.
Inspired by Li's rigidity results for $ax+b$-semigroup C*-algebras, the  problem of rigidity phenomena for the Bost--Connes C*-algebras---without any additional data such as time evolutions---was considered by the second-named author \cite{Tak1,Tak2}, where it was proven that the Bost--Connes C*-algebra remembers both the Dedekind zeta function and the narrow class number of the underlying number field; as in the case of semigroup C*-algebras, this implies rigidity if one of the number fields is Galois. A complete solution to the rigidity problem was later obtained by the second-named author and Kubota \cite{KT}: Any *-isomorphism between the Bost--Connes C*-algebras associated with two number fields gives rise to a conjugacy between the Bost--Connes semigroup dynamical systems. Moreover, this conjugacy is constructed from \K-theoretic invariants of the C*-algebras. The Bost--Connes semigroup dynamical system from \cite{LLN} is constructed from a number field using the action of the ideal semigroup on the balanced product of the integral adeles and the Galois group of the maximal abelian extension of the field, and combining the aforementioned reconstruction theorem with the dynamical characterisation of number fields obtained by combining \cite[Thoerem~3.1]{CdSLMS} and \cite[Theorem~3.1]{CLMS}, one obtains a complete solution to the rigidity problem for Bost--Connes C*-algebras. However, this is an abstract rigidity result in the sense that a *-isomorphism of Bost--Connes C*-algebras implies the underlying number fields are isomorphic, but it is not clear how this isomorphism of fields is related to the initial *-isomorphism.

The construction of Bost--Connes C*-algebras in \cite{LLN} is ad hoc: It is designed in such a way that the associated C*-dynamical system has the desired KMS-structure and symmetry group. The Hecke C*-algebras constructed from totally positive $ax+b$-groups in \cite{LNT} are less ad hoc. These are full corners in the Bost--Connes C*-algebra by \cite{LLN3}, so the results of the second-named author and Kubota in \cite{KT} yield a rigidity theorem for such Hecke C*-algebras. However, from the perspective of rigidity for C*-algebras constructed from number fields, there are other C*-algebras that are more natural from a number-theoretic or dynamical view point.

\subsection{Main result} 
Let $K$ be a number field with ring of integers $\O_K$. We consider the crossed product C*-algebra $\fA_K:=C_0(\Az_{K,f})\rtimes K^*$, where the multiplicative group $K^*$ of $K$ acts on the ring  of finite adeles $\Az_{K,f}$ through the embedding $K\subseteq \Az_{K,f}$. The canonical unital full corner $1_{\ol{\O}_K}\fA_K1_{\ol{\O}_K}$, where $\ol{\O}_K$ is the compact ring of integral adeles, is isomorphic to the semigroup crossed product $C^*(K/\O_K)\rtimes \O_K^\times$ from \cite{ALR}, where $\O_K$ is the ring of integers in $K$ and $\O_K^\times:=\O_K\setminus\{0\}$ its multiplicative monoid. This construction of C*-algebras from number fields is one of the easiest and most natural.
Our main result says that the rigidity phenomenon from \cite[Corollary~1.2]{KT} holds even in this setting:

\begin{theorem}\label{thm:mainA}
Let $K$ and $L$ be number fields. Then, the following are equivalent: 
\begin{enumerate}[\upshape(i)]
    \item The fields $K$ and $L$ are isomorphic. 
    \item There exist full projections $p \in M(\fA_K)$ and $q \in M(\fA_L)$ such that $p\fA_K p$ and $q \fA_Lq$ are *-isomorphic. 
\end{enumerate}
In particular, $K\cong L$ if and only if $\fA_K\cong \fA_L$.
\end{theorem}
Theorem~\ref{thm:mainA} will be derived from Theorem~\ref{thm:main2}, which is a slightly stronger, but more technical result. 
Given any *-isomorphism $p\fA_Kp\cong q\fA_Lq$ as in (ii), our proof explicitly constructs a field isomorphism $K\cong L$ from the *-isomorphism $p\fA_Kp\cong q\fA_Lq$. 
Such an explicit rigidity result is a new phenomenon in the setting of C*-algebras from number theory; the main theorem of \cite{KT} reconstructs the Bost--Connes semigroup dynamical system, but does not provide an explicit isomorphism between number fields.
Condition (ii) is a priori stronger than having a Morita equivalence between $\fA_K$ and $\fA_L$. We leave it as an open problem to determine if a Morita equivalence (or more generally an ordered $\KK$-equivalence over the power set of primes) implies that $K$ and $L$ are isomorphic.

We now explain an application of our rigidity theorem to topological full groups of certain groupoids.
Topological full groups were introduced by Giordano, Putnam, and Skau for Cantor minimal systems \cite{GPS} and then for \'{e}tale groupoids over a Cantor set by Matui \cite{Matui:PLMS}: To each such groupoid $\cG$, one associates a discrete group $\fg{\cG}$ of certain homeomorphisms of the unit space $\cG^{(0)}$ of $\cG$. Such full groups often exhibit surprising rigidity phenomena in that the abstract isomorphism class of the group characterises the groupoid up to isomorphism, see \cite{Matui:Crelle} and \cite{Rub}. Recently, several of these rigidity results were extended by Nyland and Ortega \cite{NO} to the case of certain non-minimal \'{e}tale groupoids whose unit space is a locally compact Cantor space, which is pertinent to our work.  

The C*-algebra $\fA_K$ has a canonical groupoid model: $\fA_K\cong C_r^*(\cG_K)$, where $\cG_K:= K^*\ltimes \Az_{K,f}$ is the transformation groupoid associated with the action $K^*\acts\Az_{K,f}$. Theorem~\ref{thm:mainA} implies that the topological full group of the stabilisation of $\cG_K$ is a complete invariant of $K$:

\begin{corollary}\label{cor:fullgroup}
Let $K$ and $L$ be number fields, and let $\cR$ denote the full equivalence relation over a countably infinite set. Then, the following are equivalent:
\begin{enumerate}[\upshape(i)]
    \item The fields $K$ and $L$ are isomorphic.
    \item The topological full groups $\fg{\cG_K \times \cR}$ and $\fg{\cG_L \times \cR}$ are isomorphic as discrete groups.
\end{enumerate} 
\end{corollary}

We view Corollary~\ref{cor:fullgroup} as an analogue of the Neukirch--Uchida theorem \cite{Neu:69,Uchida} which says that the absolute Galois group of a number field characterises the number field: If $K$ and $L$ are number fields with absolute Galois groups $G_K$ and $G_L$, respectively, then $K\cong L$ if and only if $G_K\cong G_L$. We point out that $\fg{\cG_K\times\cR}$ is a countable discrete group, whereas the absolute Galois group $G_K$ is a profinite group. 

There are other classes of topological full groups that are complete invaraints of number fields. Indeed, following our approach in \S~\ref{sec:fullgroups} we can show that topological full group of the stabilised Bost--Connes groupoid is a complete invariant of the underlying number field (Remark~\ref{rmk:BC}).
In addition, in work of the first named-author and Li \cite{BruceLi2}, it is proven that the topological full group of the groupoid underlying the ring C*-algebra of a ring of integers is a complete invariant of the associated number field. The proof of this result is groupoid-theoretic, and thus quite different from the results in this paper.

The investigation of topological full groups has led to the resolution of several open problems in group theory, see \cite{JM,JNdelaS,Nek}, so it is a natural problem to study group-theoretic properties of $\fg{\cG_K\times\cR}$. This problem is left for future work and is not considered in this article.

Our results have several other interesting consequences, which are presented in \S~\ref{sec:explicit} and \S~\ref{sec:remark}. For instance, given a number field $K$, the splitting numbers of rational primes, the ideal class group of $K$, and the automorphism group $\Aut(K)$ are each given explicit C*-algebraic descriptions.

In number theory, it is more natural to consider the full adele ring $\Az_K$ rather than the finite adele ring $\Az_{K,f}$. The crossed product C*-algebra associated with the action $K^*\acts \Az_K$ is precisely Connes' noncommutative adele class space from \cite{C}.
In \cite{BT3}, we give a systematic study of such C*-algebras and use Theorem~\ref{thm:mainA} to deduce a rigidity result in that setting also. Here, we only point out that because $\Az_K$ has a large connected component, the techniques of the present paper do not apply directly when the full adele ring is considered.

\subsection{Outline of the proof}
Our proof of Theorem~\ref{thm:mainA} comprises two parts, the first C*-algebraic and second number-theoretic. Let us first explain the strategy of the C*-algebraic part and compare it to the proof of second-named author and Kubota in \cite{KT}. Note that in \cite{KT}, the number-theoretic results needed were established separately in \cite{CdSLMS} and \cite{CLMS}, whereas our paper includes both the C*-algebraic and the number-theoretic arguments.

We first study the primitive ideal space of $\fA_K$, and prove that $\fA_K$ is a C*-algebra over the power set of nonzero prime ideals of $\O_K$. This requires a computation of the quasi-orbit space for $K^*\acts\Az_{K,f}$, which extends the work \cite{LR:00} of Laca and Raeburn from the case of the rational numbers to general number fields.
Using general observations from \cite{KT} on C*-algebras over power sets, we then obtain subquotients of $\fA_K$---which are called composition factors---parametrised by the finite subsets of primes. We describe these subquotients explicitly as crossed product C*-algebras for certain semi-local dynamical systems---that is, systems from the action of $K^*$ on finite products of local fields. The composition factor at the zeroth level is the group C*-algebra $C^*(K^*)$. 
We follow the approach in \cite{KT} to analyse composition factors, using a valuation C*-algebra which plays a role similar to that in \cite{KT} (Definition~\ref{def:BKandBval}). 

The first technical problem in our situation is the presence of torsion in $K^*$: Abstractly, we have $K^*\cong \mu_K\times\Gamma_K$, where $\mu_K=\tors(K^*)$ is the finite cyclic group of roots of unity in $K$, and $\Gamma_K=K^*/\mu_K$. 
Essentially, for all levels except the zeroth, \K-theory forgets all information about torsion in the sense that the composition factor is Morita equivalent to the crossed product for the underlying semi-local dynamical system modulo roots of unity. Therefore, we can at best only recover the underlying dynamical system modulo roots of unity. The number-theoretic part of our proof is mostly devoted to overcoming this apparent loss of information.

We introduce the auxiliary C*-algebra $\fB_K:=C_0(\Az_{K,f}/\mu_K)\rtimes\Gamma_K$ built from the dynamical system $\Gamma_K\acts\Az_{K,f}/\mu_K$, see Definition~\ref{def:BKandBval}. We show that $\fB_K$ is also a C*-algebra over the power set of primes, and we then prove a reduction result from $\fA_K$ to $\fB_K$ (Proposition~\ref{prop:commutativeKL}).
This brings us closer to a situation where we can apply the general reconstruction result \cite[Corollary~3.18]{KT}. However, we encounter two technical problems at this step. The first is caused by the unit group $\cO_K^*$ of the ring of integers: We cannot recover $\Gamma_K$ from the primitive ideal space of $\fB_K$.  
In \cite{KT}, the isomorphism of ideal semigroups that induces the conjugacy of Bost--Connes semigroup dynamical systems is obtained from the homeomorphism of primitive ideal spaces. However, we cannot adopt that strategy here, since our group $\Gamma_K$ is not directly related to the primitive ideal space. Instead, we obtain the isomorphism $\gamma \colon \Gamma_K \to \Gamma_L$ from the unitary groups of the zeroth level composition factors (Lemma~\ref{lem:existgamma}). This strategy requires an extra argument to show that $\gamma$ is valuation-preserving (Proposition~\ref{prop:propertygamma}). This technology is one of the biggest difference between our strategy and \cite{KT}. In addition, it is the reason why the condition (2) of Theorem~\ref{thm:mainA} is difficult to weaken to Morita equivalence. 
The second problem is caused by the ideal class group $\Cl_K$ of $K$: We cannot find a canonical basis for $\Gamma_K$, so we cannot use basis-fixed arguments from \cite{KT} in their original form. Hence, we establish basis-free versions of these arguments (Lemma~\ref{lem:E_val} and Lemma~\ref{lem:calcboundary}). 
These technical results allow us to replace other basis-fixed arguments by arguments using only ``partial bases'' of $\Gamma_K$, that are taken depending on the situation (Lemma~\ref{lem:DFsquare}). Consequently, we have succeeded to make all necessary arguments basis-free.
We point out that it is common to encounter difficulties caused by $\mu_K$, $\cO_K^*$, and $\Cl_K$ when we construct C*-algebras from number fields using any method. The Bost--Connes case is exceptional---such difficulties disappear since the action of the ideal semigroup is considered. For the C*-algebras considered in this paper, we completely solve these technical difficulties.

With a bit more technical work, we can apply \cite[Corollary~3.18]{KT} to (the unital part of) our composition factors to reconstruct a family of conjugacies between the semi-local dynamical systems modulo roots of unity (Proposition~\ref{prop:semi-localdata}).
We then move to the number-theoretic part or our proof: We construct an isomorphism of number fields from the semi-local data (\S~\ref{sec:NT}). The role of this part is similar to the work of Cornelissen, de Smit, Li, Marcolli, and Smit \cite{CdSLMS} in the Bost--Connes case; however the semi-local number-theoretic data arising from our C*-algebra is completely different from the number-theoretic data arising from the Bost--Connes C*-algebras, so we have to establish some novel lifting technologies in order to be in a situation where ideas from \cite{CdSLMS} can be utilised. It is interesting that the part of the strategy from \cite{CdSLMS} that is useful for us is closer to the function field case in \cite{CdSLMS}, rather than the number field case. We rely on a reformulation of Hoshi's theorem from \cite{Hoshi} (Proposition~\ref{prop:CdSLMS}), which is essentially due to \cite{CdSLMS}. In our situation, the key technical difficulty is caused by $\mu_K$ as mentioned before: We have an isomorphism only of the multiplicative groups modulo roots of unity. We lift the isomorphism of the semi-local data modulo roots of unity at the first level (Proposition~\ref{prop:nfchar}). Interestingly, the second and third levels are also involved.

\subsection{Structure of the paper}
\S~\ref{sec:prelim} contains background on $\KK$-theory and number theory, and the definition of $\fA_K$.  
\S~\ref{sec:C*/finite} contains general results on crossed product C*-algebras for actions of certain abelian groups.
In \S~\ref{sec:subquotient/prims}, we first study the structure of general C*-algebras over a power set. Then, we introduce the auxiliary C*-algebras $\fB_K$ and $\fB_\val$, and apply the general results to $\fA_K$ and these C*-algebras, giving a description of their composition factors.
Our main results are contained in \S~\ref{sec:recon}: We state the technical version of our main theorem, and prove it except for the number-theoretic step. 
The number-theoretic part of our proof is contained in \S~\ref{sec:NT}.
Our application to topological full groups is contained in \S~\ref{sec:fullgroups}. In \S~\ref{sec:explicit}, we prove that several classical number-theoretic invariants can be given explicitly in terms of C*-algebraic data (\S~\ref{sec:explicit}). Finally, in \S~\ref{sec:remark}, we make several remarks and discuss other applications.

\subsection*{Acknowledgements}Part of this research was carried out during visits of C. Bruce to the Kyoto Institute of Technology and the University of Victoria; he would like to thank the respective mathematics departments for their hospitality and support.
We thank Yosuke Kubota and Xin Li for helpful comments.

\section{Preliminaries}
\label{sec:prelim}
\subsection{Notation and terminology}
\label{sec:notation}
We let $\Kz$ denote the C*-algebra of compact operators on a separable infinite dimensional Hilbert space. Given a C*-algebra $A$, we let $M(A)$ denote the multiplier algebra of $A$. 
For C*-algebras $A$ and $B$, a *-homomorphism $\alpha \colon A \to B$ is said to be non-degenerate if (the closure of) the subspace $\alpha(A)B$ is equal to $B$. A non-degenerate *-homomorphism $\alpha \colon A \to B$ lifts to a *-homomorphism $M(A) \to M(B)$, which is still denoted by $\alpha$.
We let $\Prim(A)$ denote the primitive ideal space of $A$. For background on multiplier algebras and primitive ideal spaces, see \cite{RW:book}.

For an abelian group $G$, its dual group is denoted by $\widehat{G}$. 
If an abelian group $G$ acts on a locally compact Hausdorff space $X$ by homeomorphisms, then we let $C_0(X)\rtimes G$ denote the associated crossed product C*-algebra. 
In addition, for $g \in G$, let $u_g \in M(C_0(X) \rtimes G)$ denote the unitary corresponding to $g$. 
We refer the reader to \cite{Will:book} for background on crossed products. Since $G$ is amenable, there is no distinction between the full and reduced crossed product here.

Given a subset $Y$ of a topological space $X$, we let $\ol{Y}$ and $\Interior(Y)$ denote the closure and interior of $Y$, respectively.

Let $\tilde{\Zz}:=\Zz\cup\{\infty\}$ and $\tilde{\Nz}:=\Nz\cup\{\infty\}$, both equipped with their usual topologies.

\subsection{Ext-groups and KK-theory}
\label{sec:Ext/KK}
We collect some basic results on Ext-groups and \KK-theory for C*-algebras; most of this material originated in \cite{Kas}. We refer the reader also to \cite{B} and \cite{D} for background.
Let $A$ be a separable nuclear C*-algebra, and let B be a $\sigma$-unital stable C*-algebra. 
Let $\cExt(A,B)$ be the semigroup of (strong) unitary equivalence classes of extensions of $A$ by $B$, that is, exact sequences of the form 
\[
    \cE \colon 0 \to B \to E \to A \to 0.
\]
For an extension $\cE$, its unitary equivalence class is also denoted by $\cE$. We always consider extensions up to unitary equivalence. 
Let $\Ext(A,B)$ be the quotient of $\cExt(A,B)$ by subsemigroup of trivial extensions. Then, $\Ext(A,B)$ is a group. For $\cE \in \cExt(A,B)$, its class in $\Ext(A,B)$ is denoted by $[\cE]_{\Ext}$. 
For an extension $\cE \in \cExt(A,B)$, the Busby invariant of $\cE$ is the *-homomorphism 
\[ \tau_\cE \colon A \cong E/B \to M(B)/B  =: Q(B). \] 
As usual, we define $\K_*(A)=\K_0(A) \oplus \K_1(A)$ and consider it as a $\Zz/2\Zz$-graded abelian group. 
There is a canonical homomorphism
\[ \Ext(A,B) \to \Hom(\K_*(A),\K_{*+1}(B)),\ [\cE]_{\Ext} \mapsto \partial_\cE := \partial_0 \oplus \partial_1, \]
where $\partial_0 \colon \K_0(A) \to \K_1(B)$ and $\partial_1 \colon \K_1(A) \to \K_0(B)$ are the boundary maps associated with $\cE$. 
In addition, there is a natural isomorphism $\Ext(A,B) \to \KK(A,Q(B))$ defined by sending $[\cE]_{\Ext}$ to $[\tau_\cE]_{\KK}$. Let $[\cE]_{\KK} \in \KK^1(A,B)$ denote the element corresponding to $[\tau_\cE]_{\KK}$ under the identification $\KK^1(A,B) \cong \KK(A,Q(B))$. For $\cE \in \cExt(A,B)$, $i=0,1$, and $x \in \K_i(A)$, we have $\partial_i(x)= x \ho{A} [\cE] \in \K_{i+1}(B)$, where the symbol $\ho{}$ denotes the Kasparov product. See \cite[\S~4]{D} and \cite{B} for details.

Let $A,A',A''$ be separable nuclear C*-algebras, and let $B, B', B''$ be $\sigma$-unital stable C*-algebras. Let $\cE \in \cExt(A,B)$, $\cE' \in \cExt(A',B')$, $\cE'' \in \cExt(A'',B'')$ be extensions. 
A homomorphism $\varphi \colon \cE \to \cE'$ is a triplet $\varphi = (\alpha,\eta, \beta)$, where $\alpha \colon A \to A'$, $\beta \colon B \to B'$, and  $\eta \colon E \to E'$ are *-homomorphisms such that the following diagram commutes (see \cite[\S~1]{Ror}, for instance):
    \begin{equation} \label{eqn:defhom} \begin{tikzcd}
	\cE:0\arrow{r}&	B \arrow[d, "\beta"] \arrow[r] & E \arrow[d, "\eta"] \arrow[r] & A \arrow[d, "\alpha"] \arrow[r] & 0\\
	\cE':0\arrow{r}&	B' \arrow[r]  & E' \arrow[r] & A' \arrow[r] & 0\nospacepunct{.} 
    \end{tikzcd}\end{equation}

Let $\KK$ denote the Kasparov category: The objects of $\KK$ are the separable C*-algebras, and the set of morphisms between two such C*-algebras $A$ and $B$ is the \KK-group $\KK(A,B)$. For the background on $\KK$, see \cite[\S~22.1]{B}, \cite{MeyerNest:06}, and the references therein. We regard $\KK^1(A,B)$ as a set of morphisms of $\KK$ via the natural isomorphism $\KK^1(A,B) \cong \KK(A, B \otimes C_0(\Rz))$. Let $\Arr(\KK)$ denote the arrow category of $\KK$: The objects of $\Arr(\KK)$ are the morphisms of the Kasparov category, that is, elements of $\KK(A,B)$, where $A$ and $B$ are separable C*-algebras, and a morphism in $\Arr(\KK)$ from $[\cE]_{\KK}$ to $[\cE']_{\KK}$ is a pair $\bs{\varphi}=(\bs{x},\bs{y})$, where $\bs{x} \in \KK(A,A')$ and $\bs{y} \in \KK(B,B')$ are such that $\bs{x} \ho{A'} [\cE']_{\KK}= [\cE]_{\KK} \ho{B} \bs{y}$. General morphisms in $\Arr(\KK)$ are defined similarly, but we shall only consider morphisms between \KK-classes of extensions in this article (see \cite[Chapter~II, \S~4]{MacLane} for background on arrow categories). The category $\Arr(\KK)$ provides convenient notation for several of our proofs. It has implicitly appeared in the classification of extensions, see, for instance, \cite[Theorem~3.2]{Ror} and \cite[Theorem~2.3]{ERR}.

If $\varphi = (\alpha,\eta,\beta) \colon \cE \to \cE'$ is a homomorphism, then the pair $([\alpha]_{\KK}, [\beta]_{\KK})$ is a morphism in $\Arr(\KK)$, which is denoted by $[\varphi]_{\KK}$. In particular, for the identity homomorphism $\id_{\cE}:=(\id,\id,\id) \colon \cE \to \cE$, we let $\bs{\id}_{[\cE]_{\KK}} = [\id_\cE]_{\KK}$.
When $\bs{\varphi} = (\bs{x}, \bs{y}) \colon [\cE]_{\KK} \to [\cE']_{\KK}$ and $\bs{\varphi}' = (\bs{x}', \bs{y}') \colon [\cE']_{\KK} \to [\cE'']_{\KK}$ are morphisms in $\Arr(\KK)$, the composition $\bs{\varphi}' \circ \bs{\varphi}$ is defined to be the pair $(\bs{x} \ho{A'} \bs{x}',\ \bs{y} \ho{B'} \bs{y}')$.
A morphism $\bs{\varphi}=(\bs{x}, \bs{y}) \colon [\cE]_{\KK} \to [\cE']_{\KK}$ is said to be an isomorphism in $\Arr(\KK)$ if both $\boldsymbol{x}$ and $\boldsymbol{y}$ are \KK-equivalences. In this case, we denote the morphism $(\bs{x}^{-1}, \bs{y}^{-1})$ in $\Arr(\KK)$ by $\boldsymbol{\varphi}^{-1}$. Note that if $\boldsymbol{\varphi} \colon [\cE]_{\KK} \to [\cE']_{\KK}$
is an isomorphism in $\Arr(\KK)$, then $\boldsymbol{\varphi}^{-1} \circ \boldsymbol{\varphi} = \boldsymbol{\id}_{[\cE]_{\KK}}$ and $\boldsymbol{\varphi} \circ \boldsymbol{\varphi}^{-1} = \boldsymbol{\id}_{[\cE']_{\KK}}$. 
In addition, if $\boldsymbol{\varphi} =(\bs{x}, \bs{y}) \colon [\cE]_{\KK} \to [\cE']_{\KK}$ is a morphism in $\Arr(\KK)$, then the diagram 
    \begin{equation} \begin{tikzcd}
	\K_*(A) \arrow[d, "-\ho{A} \boldsymbol{x}"] \arrow[r, "\partial_\cE"] & \K_{*+1}(B) \arrow[d, "-\ho{B} \boldsymbol{y}"]\\
	\K_*(A') \arrow[r, "\partial_{\cE'}"] & \K_{*+1}(B') 
    \end{tikzcd}\end{equation}
commutes. 

For a separable nuclear C*-algebra $D$ belonging to the bootstrap class from \cite[Definition 22.3.4]{B}, and for an extension 
\[ \cE \colon 0 \to B \to E \to A \to 0, \]
let $D \otimes \cE \in \cExt(D \otimes A, D \otimes B)$ denote the extension 
\[ D \otimes \cE \colon 0 \to D \otimes B \to D \otimes E \to D \otimes A \to 0. \]
Then, we have $[D \otimes \cE]_{\KK} = 1_D \ho{} [\cE]_{\KK}$ in $\KK^1(D \otimes A, D \otimes B)$. In particular, we have
\begin{equation} \label{eqn:Dtensor}
\partial_{D \otimes \cE}(x \otimes y) = x \otimes \partial_{\cE}(y) \in \K_*(D \otimes B)
\end{equation}
for $x \in \K_*(D)$ and $y \in \K_*(A)$ by associativity of the Kasparov product. 
Note that $\K_*(D) \otimes \K_*(A)$ and $\K_*(D) \otimes \K_*(B)$ are identified with subgroups of $\K_*(D \otimes A)$ and $\K_*(D \otimes B)$, respectively, by K\"unneth theorem for tensor products \cite[Definition 23.1.3]{B}. When we apply Equation \eqref{eqn:Dtensor} in this article, $\K_*(D)$ is always torsion-free, so that we always have $\K_*(D) \otimes \K_*(A) \cong \K_*(D \otimes A)$ and $\K_*(D) \otimes \K_*(B) \cong \K_*(D \otimes B)$. 

A variant of the Toeplitz extension plays a crucial role in this article. 
\begin{definition} \label{def:Toeplitz}
The \emph{dilated Toeplitz extension} $\cT \in \cExt(C^*(\Zz),\Kz)$ is the extension 
\[ 0 \to C_0(\Zz) \rtimes \Zz \to C_0(\tilde{\Zz})\rtimes \Zz \to C^*(\Zz) \to 0, \]
where $\Zz$ acts on $\tilde{\Zz}=\Zz\cup\{\infty\}$ by translation, and $C_0(\tilde{\Zz})\rtimes \Zz \to C^*(\Zz)$ is the *-homomorphism induced by evaluation at $\infty$. 
\end{definition}
\vspace{-3mm}

Note that $C_0(\Zz) \rtimes \Zz$ is *-isomorphic to $\Kz$, and the unitary equivalence class $\cT$ does not depend on the choice of the *-isomorphism, since any  *-automorphism of $\Kz$ is inner.

\begin{lemma}\label{lem:boundarytoeplitz}
We have
$\partial_\cT ([1_{C^*(\Zz)}]_0) = 0$ and $\ \partial_\cT([u]_1) = -1$,
where $u \in C^*(\Zz)$ is the generating unitary corresponding to $1 \in \Zz$. 
\end{lemma}
\begin{proof}
Let $\cT_0 \in \cExt(C^*(\Zz),\Kz)$ be the (usual) Toeplitz extension
\[ 0 \to p(C_0(\Zz) \rtimes \Zz)p \to p(C_0(\tilde{\Zz})\rtimes \Zz)p \to C^*(\Zz) \to 0, \] 
where $p$ is the characteristic function of $\tilde{\Nz}=\Nz \cup \{\infty\} \subseteq \tilde{\Zz}$. 
We identify $C_0(\Zz) \rtimes \Zz$ with $\Kz(\ell^2(\Zz))$. Then, the decomposition $\ell^2(\Zz) \cong p\ell^2(\Zz) \oplus (1-p)\ell^2(\Zz)$ induces a *-isomorphism $C_0(\Zz) \rtimes \Zz \cong M_2(p(C_0(\Zz) \rtimes \Zz)p )$. 
Under this identification, we can see that $\tau_{\cT}$ is equal to $\tau_{\cT_0} \oplus 0$, which implies that $[\cT]_{\Ext}=[\cT_0]_{\Ext}$. Now the claim follows by the fact $\partial_{\cT_0}([u]_1)=-1 \in \K_0(\Cz)$, since $u \in C^*(\Zz)$ lifts to a generating isometry of the Toeplitz algebra $p(C_0(\tilde{\Zz})\rtimes \Zz)p$ (see \cite[p.168]{RLL}).
\end{proof}

Next, we fix terminology related to positive cones of $\K_0$-groups. 
Recall that for every C*-algebra $A$, the pair $(\K_0(A),\K_0(A)_+)$ is a preordered abelian group, where $\K_0(A)_+:=\{[p]_0 : p\text{ is a projection in }A\otimes\Kz\}$. For C*-algebras $A$ and $B$, we say that a group isomorphism $\varphi\colon \K_0(A)\to \K_0(B)$ is an order isomorphism if $\varphi(\K_0(A)_+)=\K(B)_+$. We shall say that a homomorphism of $\Zz/2\Zz$-graded abelian groups $\varphi\colon \K_*(A)\to \K_*(B)$ is an order isomorphism if $\varphi\colon \K_0(A)\to \K_0(B)$ is an order isomorphism. 
A \KK-equivalence $\boldsymbol{x} \in \KK(A,A')$ is said to be an ordered \KK-equivalence if the isomorphism $-\ho{A} \boldsymbol{x} \colon \K_*(A) \to \K_*(A')$ carries $\K_0(A)_+$ onto $\K_0(A')_+$. Similarly, an isomorphism $\bs{\varphi}=(\bs{x}, \bs{y}) \colon [\cE]_{\KK} \to [\cE']_{\KK}$ in $\Arr(\KK)$ is said to be an order isomorphism in \Arr(\KK) if both $\boldsymbol{x}$ and $\boldsymbol{y}$ are ordered \KK-equivalences. 

The next proposition is well-known, but we include the proof for the reader's convenience. We frequently use it without reference in this article. 
\begin{proposition}
Let $A$ and $B$ be separable C*-algebras, and suppose $\cX$ is an $A$--$B$-imprimitivity bimodule. Let $\bs{x} \in \KK(A,B)$ be the element corresponding to $\cX$. Then, the isomorphism
\[ - \ho{A} \bs{x} \colon \K_*(A) \to \K_*(B) \]
carries $\K_0(A)_+$ onto $\K_0(B)_+$. In particular, if $p \in M(A)$ is a full projection, then the inclusion map $pAp \to A$ induces an ordered \KK-equivalence. 
\end{proposition}
\begin{proof}
Let $p \in A \otimes \Kz$ be a projection. It suffices to show that $[p]_0 \ho{A} \bs{x} \in \K_0(B)_+$. 
Under the standard isomorphism $\K_0(A) \cong \KK(\Cz,A)$, $[p]_0$ corresponds to the Kasparov bimodule $[A \otimes \cH, \Phi_p, 0] \in \KK(\Cz,A)$, where $\cH$ is the infinite dimensional separable Hilbert space and $\Phi_p \colon \Cz \to \Kz(A \otimes \cH) \cong A \otimes\Kz$ sends $1 \in \Cz$ to $p$. On the other hand, we have $\bs{x} = [\cX,\varphi,0]$, where $\varphi \colon A \to \Kz(\cX)$ is the canonical isomorphism. Then, the Kasparov product $[p]_0 \ho{A} \bs{x}$ is given by 
\[ [p]_0 \ho{A} \bs{x} = [(A \otimes \cH) \otimes_A \cX, \Phi_p \otimes_A 1, 0].\]
Since $\cX$ is a full right Hilbert $B$-module, Kasparov's stabilisation theorem \cite[Theorem 1.9]{MingoPhillips} gives
\[ (A \otimes \cH) \otimes_A \cX \cong \cX \otimes \cH \cong B \otimes \cH\]
as right Hilbert $B$-modules.
Hence, with $q=(\Phi_p \otimes_A 1)(p) \in B \otimes \Kz$, we have $[p]_0 \ho{A} \bs{x}=[q]_0 \in \K_0(B)_+$. 
\end{proof}

\subsection{Number-theoretic background}
\label{sec:NTbackground}
Let $K$ be a number field with ring of integers $\cO_K$, and denote by $\mu_K$ the (finite, cyclic) group of roots of unity in $K$. Let $\cP_K$ be the set of nonzero prime ideals of $\cO_K$, and for a nonzero ideal $I\unlhd\O_K$, let $N(I):=[\O_K:I]$ be the norm of $I$. For $\p\in\cP_K$, we let $v_\p$ denote the corresponding valuation, where our normalizing conventions follow \cite[p.67]{Neu}. 
The absolute value associated with $\p\in\cP_K$ will be denoted by $|\cdot|_\p$, and we shall use $K_\p$ to denote the locally compact completion of $K$ with respect to $|\cdot|_\p$ and let $\cO_{K,\p}\subseteq K_\p$ be the associated discrete valuation ring. Let $\Az_{K,f}:=\prod_{\p\in\cP_K}'(K_\p,\cO_{K,\p})$ be the ring of finite adeles over $K$. We let $\Az_{K,f}^*:= \prod_{\p\in\cP_K}'(K_\p^*,\cO_{K,\p}^*)$ be the group of finite ideles equipped with the restricted product topology, and let $\ol{\O}_K:=\prod_{\p\in\cP_K}\O_{K,\p}$ be the compact subring of integral adeles. The multiplicative group $K^*:=K\setminus\{0\}$ embeds diagonally into $\Az_{K,f}^*$. 

We refer the reader to \cite{Neu} and \cite{Hasse} for background on number theory. We shall make frequent use of the following approximation result, which is a consequence of Strong Approximation, as stated in, for instance, \cite[Chapter~20]{Hasse}.
\begin{lemma}
\label{lem:approx}
Given $n_1,...,n_k\in\Zz$ and $\p_1,...,\p_k\in\cP_K$, there exists $x\in K^*$ such that $v_{\p_i}(x)=n_i$ for all $1\leq i\leq k$ and $v_\p(x)\geq 0$ for all $\p\in\cP_K\setminus\{\p_1,...,\p_k\}$. In particular, the image of $K^*$ in $\Az_{K,f}$ is dense.
\end{lemma}

For $\p\in\cP_K$, let $\O_{K,\p}^{(1)}:=1+\p\O_{K,\p}$. We will need the following result on the multiplicative group of $K_\p$.
\begin{lemma}[{\cite[Proposition~II.5.3]{Neu}} and {\cite[p.224]{Hasse}}]
\label{lem:localstuff}
For $\p\in\cP_K$, choose $\pi_\p\in K^*$ with $v_\p(\pi_\p)=1$, and let $p$ be the rational prime lying under $\p$. Then, there are isomorphisms of topological groups
\begin{enumerate}[\upshape(i)]
    \item $K_\p^*\cong \pi_\p^\Zz\times \O_{K,\p}^*$;
    \item $\O_{K,\p}^*\cong(\Zz/(N(\p)-1)\Zz)\times\O_{K,\p}^{(1)}$;
    \item $\O_{K,\p}^{(1)}\cong (\Zz/p^a\Zz)\times\Zz_p^{[K_\p:\Qz_p]}$, where $a\geq 0$ with $a=0$ if $p$ is odd and unramified in $K$.
\end{enumerate}
\end{lemma}
Note that the element $\pi_\p$ in Lemma~\ref{lem:localstuff} exists by Lemma~\ref{lem:approx}.

We let $\Gamma_K:=K^*/\mu_K$, and for each finite subset $F\subseteq\cP_K$, let $\Gamma_K^F:=\{a\in \Gamma_K : v_\p(a)=0\text{ for all }\p\in F\}$.
\begin{lemma}
\label{lem:GammaKdecomp}
Let $F=\{\p_1,\dots,\p_l\}$ be a finite set of primes. Let $\{\pi_{\p_j}\}_{j=1}^l \subseteq K^*$ be a family satisfying  $v_{\p_j}(\pi_{\p_k})=\delta_{j,k}$ for all $j,k=1,\dots,l$.  
Then, we have
\[
   \Gamma_K = (\textstyle{\prod}_{j=1}^l \pi_{\p_j}^\Zz) \times \Gamma_K^F.
 \] 
In particular, $\Gamma_K^F$ is a summand of $\Gamma_K$. 
\end{lemma}
\begin{proof}
Note that such a family $\{\pi_{\p_j}\}_{j=1}^l$ exists by Lemma~\ref{lem:approx}. We have a surjective homomorphism $\Gamma_K\to \prod_{j=1}^l\Zz=\Zz^F$, $x\mapsto (v_{\p_j}(x))_{j=1}^l$, 
with kernel $\Gamma_K^F$.
Since $v_{\p_j}(\pi_{\p_k})=\delta_{j,k}$ for all $j,k=1,\dots,l$, the elements $\bm{e}_j:=(v_{\p_j}(\pi_{\p_k}))_{j=1}^l$ ($1\leq j\leq l$) are precisely the standard $\Zz$-basis for $\Zz^F$. Hence, the map $\Zz^F\to\Gamma_K$ determined by $\bm{e}_j\mapsto \pi_{\p_j}$ is a section for $\Gamma_K\to\Zz^F$, which gives a splitting for the short exact sequence 
\[
0\to \Gamma_K^F\to\Gamma_K\to\Zz^F\to 0.
\]
Thus, $\Gamma_K = (\prod_{j=1}^l \pi_{\p_j}^\Zz) \times \Gamma_K^F$.
\end{proof}

\subsection{The C*-algebra associated with the action on the ring of finite adeles}
\label{sec:fAK}
The group $K^*$ acts on $\Az_{K,f}$ by homeomorphisms through the diagonal embedding $K^*\to \Az_{K,f}$.
\begin{definition}
For a number field $K$, we let
\[
\fA_K:=C_0(\Az_{K,f})\rtimes K^*.
\]
\end{definition}
\vspace{-3mm}

Similarly to \cite{KT}, we will decompose $\fA_K$ into extensions of composition factors. Their unital parts are the following crossed product C*-algebras arising from semi-local number-theoretic data modulo roots of unity.
\begin{definition}
\label{def:BKF}
For a number field $K$ and a finite subset $F\subseteq\cP_K$, we let 
\[
B_K^F:=C\Bigl(\bigl(\textstyle{\prod}_{\p\in F}\O_{K,\p}^*\bigr)/\mu_K\Bigr)\rtimes\Gamma_K^F,
\] where $\Gamma_K^F$ acts on $(\textstyle{\prod}_F\O_{K,\p}^*)/\mu_K$ through the canonical embedding $\Gamma_K^F\to (\textstyle{\prod}_F\O_{K,\p}^*)/\mu_K$.
\end{definition}
\vspace{-3mm}
By definition, $B_K^\emptyset=C^*(\Gamma_K)$. If $F$ is nonempty, then $\Gamma_K\subseteq (\textstyle{\prod}_{\p\in F}\O_{K,\p}^*)/\mu_K$ is a dense subgroup, so that $B_K^F$ is simple and has a unique tracial state. The \K-groups of the C*-algebras $B_K^F$ and the boundary maps between them will be the main tools of our analysis. Note that the C*-algebra $B_K^F$ is classifiable, see \S~\ref{sec:classification}.

The original Bost--Connes C*-algebra from \cite{BC} is a full corner in the crossed product $C_0(\Az_{\Qz,f})\rtimes \Qz_+^*$, and $\fA_\Qz$ is the crossed product of $C_0(\Az_{\Qz,f})\rtimes \Qz_+^*$ by $\{\pm 1\}$. In general, the full corner $1_{\ol{\O}_K}\fA_K1_{\ol{\O}_K}$ of $\fA_K$ is, by \cite[Proposition~2.6]{LL}, isomorphic to the semigroup crossed product $C^*(K/\O_K)\rtimes_\delta\O_K^\times$, where the action $\delta$ of multiplicative monoid $\cO_K^\times =\O_K\setminus\{0\}$ on the group C*-algebra $C^*(K/\O_K)$ is given on the generating unitaries by 
\[
\delta_a(u_y)=\frac{1}{N(a\O_K)}\sum_{\substack{x\in K/\O_K,\\ ax=y}}u_x \quad\text{ for } a\in \O_K^\times\text{ and }  y\in K/\O_K,
\] 
cf. \cite[Proposition~1.2]{ALR}. The C*-algebra $C^*(K/\O_K)\rtimes_\delta\O_K^\times$ has the advantage that it is defined without reference to local fields or adele rings. Moreover, by \cite[Proposition~2.1]{ALR}, $C^*(K/\O_K)\rtimes_\delta\O_K^\times$ has a natural presentation that generalises that of the original Bost--Connes C*-algebra in \cite[Proposition~18]{BC}.
There is also a model of $C^*(K/\O_K)\rtimes_\delta\O^\times$ as a Hecke C*-algebra, see \cite[Corollary~2.5]{ALR}. 

All C*-algebras in this article can be easily checked to be separable and nuclear (except for multiplier algebras). Those properties are tacitly used when we apply \KK-theory.

\section{Crossed products modulo a finite group}
\label{sec:C*/finite}

\subsection{General observations}
\label{sec:genob}
Let $X$ be a locally compact second-countable Hausdorff space, and let $G$ be a countable abelian group such that $\mu := \tors(G)$ is finite and $G/\mu$ is free abelian. Let $\alpha \colon G \acts X$ be an action, and let $\pi \colon X \to X/\mu$ be the quotient map. Let $\ol{\alpha} \colon G/\mu \acts X/\mu$ be the action induced from $\alpha$. 
In this setting, we make several observations on the relation between $C_0(X/\mu) \rtimes_{\ol{\alpha}} (G/\mu)$ and $C_0(X) \rtimes_\alpha G$.

At first, we fix notation that will be used throughout this section. 
For $g \in G$, let $\ol{g}$ denote the image of $g$ in $G/\mu$. 
For a character $\chi \in \widehat{\mu}$, let $p_\chi \in C^*(\mu) \subseteq M(C_0(X) \rtimes_\alpha G)$ be the projection corresponding to $\chi$. Namely, 
$p_\chi = \frac{1}{|\mu|} \sum_{g \in \mu} \chi(g) u_g$.
For $\chi \in \widehat{G}$, the projection corresponding to the restriction of $\chi$ to $\mu$ is also denoted by $p_\chi$. 

\begin{proposition} \label{prop:crossedproduct1}
 Let $\chi \in \widehat{G}$. 
 Then, there exists a *-homomorphism 
 \[ \Phi_\chi \colon C_0(X/\mu) \rtimes_{\ol{\alpha}} (G/\mu) \to C_0(X) \rtimes_\alpha G \]
 such that $\Phi_\chi(fu_{\ol{g}}) = p_\chi (f \circ \pi)\chi(g) u_g$ for $f \in C_0(X/\mu)$ and $g \in G$. 
\end{proposition}

\begin{proof}
 First, the projection $p_\chi$ commutes with $u_g$ for any $g \in G$ since $G$ is commutative. In addition, $p_\chi$ commutes with $f \circ \pi$ for any $f \in C_0(X/\mu)$ since $f \circ \pi$ is a $\mu$-invariant function. 
 For $g \in G$, let $U_g := p_\chi \chi(g) u_g$. 
 Then, $U$ is a unitary representation of $G$ in $M(p_\chi(C_0(X) \rtimes_\alpha G)p_\chi)$. For $g \in \mu$, we have 
\[
    p_\chi u_g = \frac{1}{|\mu|} \sum_{h \in \mu} \chi(h) u_{gh} 
               = \frac{1}{|\mu|} \sum_{h \in \mu} \chi(g^{-1}h) u_{h} 
               = \chi(g^{-1}) p_\chi,
\]
so that $p_\chi \chi(g) u_g = p_\chi$. Thus, $\Ker U$ contains $\mu$, so that $U$ factors through $G/\mu$. 
 Hence, 
 \[ \ol{U} \colon G/\mu \to M(p_\chi(C_0(X) \rtimes_\alpha G)p_\chi),\ 
 \ol{g} \mapsto U_g\]
 is a well-defined unitary representation of $G/\mu$. 
 Let $\rho \colon C_0(X/\mu) \to p_\chi (C_0(X) \rtimes_\alpha G) p_\chi$ be the *-homomorphism defined by $\rho(f) = p_\chi(f \circ \pi)$ for $f \in C_0(X/\mu)$. 
 We can see 
 $\ol{U}_{\ol{g}} \rho(f) \ol{U}_{\ol{g}}^* = \rho(\ol{\alpha}_{\ol{g}}(f))$, which completes the proof by the universality of crossed products. 
\end{proof}

\begin{proposition} \label{prop:factorKK}
The map 
\[ \widehat{G} \to \KK(C_0(X/\mu)\rtimes_{\ol{\alpha}} (G/\mu),\, C_0(X)\rtimes_\alpha G),\ \chi \mapsto [\Phi_\chi]_{\KK}\]
factors through $\widehat{G} \to \widehat{\mu}$. 
\end{proposition}
\begin{proof}
Let $\chi,\chi' \in \widehat{G}$ and suppose that $\chi\vert_\mu=\chi'\vert_\mu$. 
Since $\widehat{G}$ is isomorphic to $\widehat{\mu} \times \Tz^N$, where $N = \rank G/\mu$ (possibly infinite), $\chi$ and $\chi'$ live in the same path-connected component of $\widehat{G}$. Hence, $\Phi_\chi$ and $\Phi_{\chi'}$ are homotopic, which implies that  $[\Phi_\chi]_{\KK}=[\Phi_{\chi'}]_{\KK}$.
\end{proof}

Let $\tau \colon \widehat{G} \acts C_0(X) \rtimes_{\alpha} G$ be the dual action. By definition, we have for every $\chi \in \widehat{G}$
\[ \tau_\chi(fu_g) = \chi(g)fu_g\quad \text{ for all } f \in C_0(X)\text{ and } g \in G. \]
We extend the dual action to $\tau \colon \widehat{G} \acts M(C_0(X) \rtimes_{\alpha} G)$. For $\chi,\chi' \in \widehat{G}$, we have $\tau_{\chi'}(p_\chi)=p_{\chi'\chi}$.
The proof of the next proposition is exactly the same as that of Proposition~\ref{prop:factorKK}. 
\begin{proposition} \label{prop:factorKK2}
The homomorphism 
\[ \widehat{G} \xrightarrow{\tau} \Aut(C_0(X) \rtimes_{\alpha} G) \to \Aut(C_0(X) \rtimes_{\alpha} G)/\sim\] 
factors through $\widehat{G} \to \widehat{\mu}$, where $\sim$ denotes the homotopy equivalence relation. 
\end{proposition}

\begin{lemma} \label{lem:generalcondexp}
 Let $E^\mu \colon C_0(X) \to C_0(X/\mu)$ be the canonical faithful conditional expectation. Then, we have 
 $p_\chi fp_\chi = E^\mu(f)p_\chi=p_\chi E^\mu(f)$ for every $f \in C_0(X)$ and $\chi \in \widehat{G}$. 
\end{lemma}
\begin{proof}
 The conditional expectation $E^\mu$ is given by 
 \[ E^\mu(f) = \frac{1}{|\mu|}\sum_{g \in \mu} \alpha_g(f). \]
The following calculation completes the proof:
 \begin{align*}
     p_\chi fp_\chi &= \frac{1}{|\mu|^2}\sum_{g,h \in \mu} (\chi(g)u_g) f (\chi(h)u_h)
     = \frac{1}{|\mu|^2}\sum_{g,h \in \mu} \alpha_g(f) \chi(gh) u_{gh} \\
     &= \frac{1}{|\mu|^2}\sum_{g,h \in \mu} \alpha_g(f) \chi(h) u_{h} 
     = \left( \frac{1}{|\mu|}\sum_{g \in \mu} \alpha_g(f) \right) 
     \left( \frac{1}{|\mu|}\sum_{h \in \mu} \chi(h)u_h \right)
     = E^\mu(f)p_\chi.
     \qedhere 
 \end{align*}
\end{proof}

\begin{lemma} \label{lem:corner}
 Let $\chi \in \widehat{G}$, and let $\Phi_\chi$ be the *-homomorphism from Proposition~\ref{prop:crossedproduct1}. Then, $\Phi_\chi$ is injective and $\Im \Phi_\chi = p_\chi(C_0(X) \rtimes_\alpha G)p_\chi$. 
\end{lemma}
\begin{proof}
 First, we determine the image of $\Phi_\chi$. Let $(\ol{U},\rho)$ be the covariant representation of $(G/\mu,C_0(X/\mu),\alpha)$ in $M(C_0(X) \rtimes_\alpha G)$ from the proof of Proposition~\ref{prop:crossedproduct1} that defines $\Phi_\chi$. The image of $\Phi_\chi$ is clearly contained in the corner by $p_\chi$. Since $p_\chi$ commutes with every element in $C^*(G)$, in order to show that $\Im\Phi_\chi$ contains the corner by $p_\chi$, it suffices to prove that $\Im\rho=p_\chi C_0(X)p_\chi$ and that $\Im\ol{U}$ generates $p_\chi C^*(G)$.
Lemma \ref{lem:generalcondexp} implies that $p_\chi C_0(X)p_\chi = \rho(C_0(X/\mu))$. 
  For each $g \in G$, let $U_g = p_\chi \chi(g) u_g$. Then, each $U_g$ is a unitary of $p_\chi C^*(G)$, and the C*-algebra $p_\chi C^*(G)$ is generated by $\{U_g \colon g \in G\}$. By definition, we have $\ol{U}_{\ol{g}}=U_g$, so $\Im\ol{U}$ generates $p_\chi C^*(G)$.

We show  injectivity of $\Phi_\chi$. 
Let $E \colon C_0(X) \rtimes_\alpha G \to C_0(X)$ and $\ol{E} \colon C_0(X/\mu) \rtimes_{\ol{\alpha}} (G/\mu) \to C_0(X/\mu)$ be the canonical faithful conditional expectations.
Let 
\[ \Theta = |\mu| p_\chi E^\mu \circ E \colon C_0(X) \rtimes_\alpha G \to p_\chi C_0(X/\mu). \]
Then, $\Theta$ is a continuous linear map. We can see that the diagram 
 \[\begin{tikzcd}
 C_0(X/\mu) \rtimes_{\ol{\alpha}} (G/\mu)\arrow[d,"\ol{E}"] \arrow[r, "\Phi_\chi"] & C_0(X)  \rtimes_\alpha G \arrow[d,"\Theta"] \\
  C_0(X/\mu) \arrow[r, "\Phi_\chi\vert_{C_0(X/\mu)}"] & p_\chi C_0(X/\mu)
 \end{tikzcd}\]
 commutes, since for any $\ol{g} \in G/\mu$ and $f \in C_0(X/\mu)$
 \begin{align*}
     E\circ\Phi_\chi(fu_{\ol{g}})=\begin{cases}
    \frac{1}{|\mu|}(f\circ\pi) & \text{ if } \ol{g}=1,\\
     0 &\text{ if } \ol{g}\neq 1.
    \end{cases}
 \end{align*}
We claim that $\Phi_\chi$ is injective on $C_0(X/\mu)$. If $\Phi_\chi(f)=p_\chi(f\circ\pi)=0$ for some $f\in C_0(X/\mu)$, then for all $\chi' \in\widehat{\mu}$, we have 
 \[
 0=\tau_{\chi'\chi^{-1}} (p_\chi(f\circ\pi)) = p_{\chi'} (f\circ\pi).
 \]
Thus, $f\circ\pi=(\sum_{\chi'}p_{\chi'})(f\circ\pi)=0$, so $f\circ\pi=0$. Because $\pi$ is surjective, if follows that $f=0$.

Suppose $\Phi_\chi(a)=0$ for some $a\in  C_0(X/\mu) \rtimes_{\ol{\alpha}} (G/\mu)$. Then $\Phi_\chi(a^*a)=0$, so commutativity of the above diagram gives us $0=\Theta\circ\Phi_\chi(a^*a)=\Phi_\chi\circ\ol{E}(a^*a)$. Since $\Phi_\chi$ is injective on $C_0(X/\mu)$, this implies $\ol{E}(a^*a)=0$, so that faithfulness of $\ol{E}$ gives $a=0$.
\end{proof}

\subsection{Translation actions}
\label{sec:translation}
In this subsection, in addition to the assumptions in \S~\ref{sec:genob}, we assume that $X$ is a locally compact abelian group and that $G$ is a dense subgroup of $X$. Furthermore, we assume that $\alpha \colon G \acts X$ is the translation action. Let $\Phi = \Phi_1$ be the map from Proposition~\ref{prop:crossedproduct1} associated with the trivial character $1$ of $G$, and let $p=p_1$.

\begin{lemma} \label{lem:full}
The corner $p(C_0(X) \rtimes_\alpha G)p$ is full in $C_0(X) \rtimes_\alpha G$.
\end{lemma}
\begin{proof}
 Let $\chi \in \widehat{X}$ and  $g \in \mu$. Let $v_\chi \in C_b(X)$ be the unitary corresponding to $\chi$. First, we observe that 
 $u_g v_\chi u_g^* = \chi(g)^{-1} v_\chi$ for every $g \in G$. 
 Then, we have 
 $v_\chi u_g v_\chi^* = \chi(g) u_g$ 
 for every $\chi \in \widehat{X}$ and $g \in \mu$, because 
 \[ v_\chi u_g v_{\chi^{-1}} = v_\chi ((\chi^{-1})(g)^{-1} v_{\chi^{-1}})u_g
   = \chi(g) u_g. \]
 Let $\chi \in \widehat{\mu}$ be a character and extend $\chi$ to a character on $X$. Then, 
 \[ v_\chi p v_\chi^* = \frac{1}{|\mu|} \sum_{g \in \mu} v_\chi u_g v_\chi^*
    = \frac{1}{|\mu|} \sum_{g \in \mu} \chi(g) u_g = p_\chi. \]
 Let $A=C_0(X) \rtimes_\alpha G$. For every $a \in A_+$, we have $a^{1/2}p_\chi a^{1/2} \in ApA$. 
 Since $1 = \sum_\chi p_\chi$, we conclude that $A=ApA$, which completes the proof.
 \end{proof}
 
\begin{proposition}\label{prop:projmu}
 There exists a canonical *-homomorphism 
 \[ \Phi \colon C_0(X/\mu) \rtimes_{\ol{\alpha}} (G/\mu) \to C_0(X) \rtimes_\alpha G \]
 such that $\Phi(fu_{\ol{g}}) = p(f \circ \pi)u_g$ for $f \in C_0(X/\mu)$ and $g \in G$. Moreover, $\Phi$ induces an ordered \KK-equivalence. 
\end{proposition}

\begin{proof}
This follows from Proposition~\ref{prop:crossedproduct1}, Lemma~\ref{lem:corner}, and Lemma~\ref{lem:full}. 
\end{proof}

\subsection{Trivial actions}
\label{sec:trivial}
In this subsection, we consider the case $X={\rm pt}$. Then, $C_0(X) \rtimes_{\alpha} G = C^*(G)$.  

\begin{proposition} \label{prop:main6.3}
For each $\chi \in \widehat{\mu}$, choose an extension $\chi \in \widehat{G}$, and let $\Phi_\chi \colon C^*(G/\mu) \to C^*(G)$ be the *-homomorphism from Proposition~\ref{prop:crossedproduct1}. Then, 
\[ \Psi := \sum_{\chi \in \widehat{\mu}} \Phi_\chi \colon C^*(G / \mu)^{\oplus |\mu|} \to C^*(G)\]
is a *-isomorphism. Moreover, $[\Psi]_{\KK}$ does not depend on which extensions we choose. 
\end{proposition}
\begin{proof}
Since $\{p_\chi \colon \chi \in \widehat{\mu}\}$ is a family of orthogonal projections with $\sum_\chi p_\chi = 1$, we have  
\[ C^*(G) = \bigoplus_{\chi \in \widehat{\mu}} p_\chi C^*(G). \]
By Lemma~\ref{lem:corner}, each $\Phi_\chi$ induces a *-isomorphism $C^*(G/\mu) \to p_\chi C^*(G)$. Hence, $\Psi$ is a *-isomorphism. The latter claim follows from Proposition \ref{prop:factorKK}. 
\end{proof}

\begin{lemma} \label{lem:minproj}
 For each $\chi \in \widehat{\mu}$, the projection $p_\chi$ is a minimal projection in $C^*(G)$. In addition, if $p$ is a minimal projection in $C^*(G)$, then $p=p_\chi$ for some $\chi \in \widehat{\mu}$.
\end{lemma} 

\begin{proof}
 Since $G/\mu$ is free abelian, $C^*(G/\mu)$ does not contain nontrivial projections. Hence the claim follows from Proposition~\ref{prop:main6.3}. 
\end{proof}
We can also obtain Lemma~\ref{lem:minproj} directly by using the isomorphism $C^*(G) \cong C(\widehat{\mu} \times \Tz^N)$, where $N= \rank G/\mu$.

\subsection{Tensor product decompositions}
\label{sec:tensordecomp} 
In this section, we work on a slightly different setting than in the previous subsections. 

\begin{proposition} \label{prop:tensor}
    Let $X$ be a locally compact second-countable (Hausdorff) abelian group, let $\Gamma$ be a countable free abelian group, and let $\lambda \colon \Gamma \to X$ be a homomorphism. Let $\Gamma_0 \subseteq \Gamma$ be a summand. Suppose that $X_0 = \ol{\lambda(\Gamma_0)}$ is a compact open subgroup, and the induced homomorphism $\Gamma/\Gamma_0 \to X/X_0$ is an isomorphism. Then, the inclusion map $C(X_0) \rtimes_\lambda \Gamma_0 \subseteq C_0(X) \rtimes_\lambda \Gamma$ induces an ordered \KK-equivalence. 
\end{proposition}
\begin{proof}
    First, since $X_0$ is compact open in $X$, we can identify $C(X_0) \rtimes_\lambda \Gamma_0$ with a C*-subalgebra of $C_0(X) \rtimes_\lambda \Gamma$.

    Since $\Gamma/\Gamma_0$ is free abelian, we can choose a section $s\colon \Gamma/\Gamma_0 \to \Gamma$ of the quotient map. Let $\Gamma_1 = s(\Gamma/\Gamma_0)$ and $X_1=\lambda(\Gamma_1)$. 
    We have the following commutative diagram
    \[\begin{tikzcd}
    \Gamma_1 \arrow[d,"s^{-1}"]\arrow[r,"\lambda\vert_{\Gamma_1}"]& X_1 \arrow[d,"\rm{quot}"]\\
    \Gamma/\Gamma_0 \arrow[r,"\ol{\lambda}"]& X/X_0\nospacepunct{,}
    \end{tikzcd}\]
    where $\ol{\lambda}$ is the isomorphism induced by $\Gamma\to X/X_0$, and the right vertical arrow is the restriction of the quotient map $X\to X/X_0$ to $X_1$. Since $s^{-1}$ and $\ol{\lambda}$ are isomorphisms, the restriction of $\lambda$ to $\Gamma_1$ is an isomorphism onto $X_1$, and the quotient map $X\to X/X_0$ restricts to an isomorphism from $X_1$ onto $X/X_0$. The inverse of the isomorphism $X_1\cong X/X_0$ gives a decomposition $X=X_1\times X_0$ compatible with the decomposition $\Gamma=\Gamma_1\times\Gamma_0$, so the action $\Gamma\acts X$ is the product of the actions $\Gamma_1\acts X_1$ and $\Gamma_0\acts X_0$.
    We have 
    \[ C_0(X) \rtimes \Gamma 
    \cong (C_0(X_1) \rtimes \Gamma_1) \otimes (C(X_0) \rtimes \Gamma_0)
    \cong \Kz \otimes (C(X_0) \rtimes \Gamma_0). \]
    Let $e \in C_0(X_1)$ be the characteristic function of $\{1\} \subseteq X_1$. Then, under the above decomposition, the inclusion map $C(X_0) \rtimes \Gamma_0 \to C_0(X) \rtimes \Gamma$ is equal to $x \mapsto e \otimes x$. Hence, it induces an ordered \KK-equivalence.
\end{proof}

Contrary to arguments in \cite{KT}, we cannot start from a basis of $\Gamma$. One of the main tasks in this article is to replace arguments of \cite{KT} with basis-free arguments. In the proof of Proposition \ref{prop:tensor}, we have obtained the tensor product decomposition $C_0(X) \rtimes \Gamma \cong \Kz \otimes (C(X_0) \rtimes \Gamma_0)$, which is similar to the decomposition used in \cite{KT}. However, this decomposition depends on the choice of section $s$. 
The point is that the isomorphism in \K-theory from Proposition~\ref{prop:tensor} does not depend on the choice of section $s \colon \Gamma/\Gamma_0 \to \Gamma$.

\section{Subquotients, primitive ideals, and auxiliary C*-algebras}
\label{sec:subquotient/prims}

\subsection{C*-algebras over topological spaces}
\label{sec:C*overX}

We recall some basics on C*-algebras over topological spaces from \cite{MeyerNest:09} and \cite{Kirchberg}. 
Let $X$ be any topological space. We let $\Oz(X)$ denote the lattice of open subsets of $X$, ordered by inclusion. 
If $A$ is a C*-algebra, then we let $\Iz(A)$ be the lattice of (closed, two-sided) ideals of $A$, also ordered by inclusion. It is well-known that there is a lattice isomorphism
\[
\Oz(\Prim(A))\cong \Iz(A), \quad U\mapsto \textstyle{\bigcap}_{P\in U^c}P.
\]
A C*-algebra over $X$ is a C*-algebra $A$ together with a continuous map $\psi\colon \Prim(A)\to X$.
Let $(A,\psi)$ be a C*-algebra over $X$. Then, we get a map $\Oz(X)\to \Oz(\Prim(A))$ given by $U\mapsto \psi^{-1}(U)$. We let $A(U)$ denote the corresponding ideal of $A$ under $\Oz(\Prim(A))\cong \Iz(A)$.
If $A$ and $B$ are C*-algebras over $X$, then we say that a *-homomorphism $\varphi\colon A\to B$ is $X$-equivariant if $\varphi(A(U))\subseteq B(U)$ for every open set $U\subseteq X$.

A subset $Z\subseteq X$ is locally closed if there exists $U,V\in\Oz(X)$ with $V\subseteq U$ such that $Z=U\setminus V$. Given a locally closed set $Z=U\setminus V$, one obtains a subquotient $A(Z):=A(U)/A(V)$; by \cite[Lemma~2.15]{MeyerNest:09}, $A(Z)$ does not depend on the choice of open sets $U$ and $V$.
If $Z\subseteq X$ is locally closed, and $W\subseteq X$ is open, then we get an extension
\begin{equation}
\label{eqn:WinZ}
   \cE^Z_W(A)\colon 0\to A(Z\cap W)\to A(Z)\to A(Z\setminus W)\to 0.
\end{equation}
If $\varphi\colon A\to B$ is an $X$-equivariant *-homomorphism, then for every locally closed subset $Z\subseteq X$, there exists a *-homomorphism $\varphi_Z\colon A(Z)\to B(Z)$ induced from $\varphi$. 
If additionally $W\subseteq X$ is an open subset, then we get a homomorphism of extensions $\cE^Z_W(A)\to \cE^Z_W(B)$.

\subsection{C*-algebras over a power set}
\label{sec:powersets}

The C*-algebra $\fA_K$ and the auxiliary C*-algebras which will be introduced later are C*-algebras over the power set of the set of primes, equipped with the power cofinite topology. We begin with some generalities. 

Let $\cP$ be a nonempty set (later, $\cP$ will be the set of nonzero prime ideals in a ring of integers). Denote by $2^\cP=\{0,1\}^\cP$ the power set of $\cP$. The power-cofinite topology on $2^\cP$ is the product topology on $\{0,1\}^\cP$ with respect to the topology on $\{0,1\}$ given by $\Oz(\{0,1\})=\{\emptyset,\{0\},\{0,1\}\}$; the basic open sets for the power-cofinite topology are given by
\[
U_F:=2^{\cP\setminus F}=\{T\subseteq \cP : T\cap F=\emptyset\},
\]
where $F$ ranges over the finite subsets of $\cP$. Note that every $U_F$ is compact, but need not be closed. For any $S\subseteq \cP$, we have $\ol{\{S\}}=\{T\subseteq \cP : S\subseteq T\}$. Thus, for any two subsets $S,T \subseteq \cP$, we have $S \subseteq T$ if and only if $\ol{\{S\}} \supseteq \ol{\{T\}}$. Note that $\cP$ is the unique closed point of $2^{\cP}$.

For a finite subset $F\subseteq \cP$ and $\p\in \cP\setminus F$, we let $F_\p:=F\cup\{\p\}$. 
Then $\{F^c\}$ and $\{F^c,F_\p^c\}$ are locally closed subsets of $2^\cP$ (see \cite[\S~2]{KT}).
Composition factors and associated extensions from finite sets of primes were introduced in \cite[\S~2]{KT}:
\begin{definition}
Suppose $A$ is a C*-algebra over $2^{\cP}$ and $F\subseteq \cP$ is a finite subset. We let $A^F:=A(\{F^c\})$ be the subquotient of $A$ corresponding to $\{F^c\}$. For $\p\in F^c$, let $\cE_{A}^{F,\p}$ denote the extension 
\begin{equation}
\label{eqn:Efp}
\cE_A^{F,\p}\colon 0\to A(\{F_\p^c\})\to A(\{F^c,F_\p^c\})\to A(\{F^c\})\to 0
\end{equation}
from \S~\ref{sec:C*overX}.
\end{definition}
\vspace{-3mm}
Note that the extension $\cE_A^{F,\p}$ is obtained by taking $Z=\{F^c,F_\p^c\}$ and $W=\ol{\{F^c\}}^c$ with the notation from \S~\ref{sec:C*overX}. 
The extension $\cE_K^{F,\p}$ is determined from the structure of C*-algebras over $2^{\cP}$. Hence, a $2^{\cP}$-equivariant *-homomorphism induces homomorphisms between extensions.

\begin{definition}
Suppose $A$ and $B$ are C*-algebras over $2^{\cP}$ and $\alpha\colon A\to B$ is a $2^{\cP}$-equivariant *-homomorphism. For every $F \subseteq \cP_K$, let $\alpha^F\colon A^F\to B^F$ denote the *-homomorphism induced by $\alpha$. In addition, for every finite subset $F\subseteq \cP$ and $\p\in \cP \setminus F$, let 
\[
\alpha^{F,\p}:=(\alpha^F,*,\alpha^{F_\p})\colon \cE_A^{F,\p}\to \cE_B^{F,\p}
\]
denote the homomorphism of extensions induced by $\alpha$, where $*$ denotes the *-homomorphism $A(\{F^c\}, \{F_\p^c\})\\\to B(\{F^c\}, \{F_\p^c\})$. 
\end{definition}
\vspace{-3mm}
We do not label the middle *-homomorphism of $\alpha^{F,\p}$. 
Strictly speaking, the definition of $\alpha^{F,\p}$ does not fit to the definition in Section~\ref{sec:Ext/KK}, since composition factors $A^F$ may not be stable. However, in our concrete situation, all of the composition factors $A^F$ are  indeed stable ($A^\emptyset$ is exceptional, but it only appears in the last term of exact sequences).

For two sets $\cP$ and $\cQ$ and a bijection $\Psi \colon \cP \to \cQ$, let $\tilde{\Psi} \colon 2^\cP \to 2^\cQ$ denote the homeomorphism defined by $\tilde{\Psi}(S)=\Psi(S)$ for subsets $S \in 2^\cP$. The topology of $2^\cP$ is closely related to its order structure:

\begin{lemma}
\label{lem:2Sto2T}
 Let $\cP$ and $\cQ$ be nonempty sets, and let $\varphi \colon 2^\cP \to 2^\cQ$ be a bijection. Then, the following are equivalent: 
 \begin{enumerate}[\upshape(1)]
     \item There exists a bijection $\Psi \colon \cP \to \cQ$ such that $\varphi = \tilde{\Psi}$.
     \item The bijection $\varphi$ is a homeomorphism with respect to the power-cofinite topologies. 
     \item The bijection $\varphi$  is order-preserving. 
 \end{enumerate} 
\end{lemma}
\begin{proof}
 The claim that (1) implies (2) is clear. We show (2) implies (3). Suppose $\varphi$ is a homeomorphism. Let $S,T \subseteq \cP$ be such that $S \subseteq T$. Then, we have
 \[ \ol{\{ \varphi(S)\}} = \varphi(\ol{\{S\}}) 
 \supseteq \varphi(\ol{\{T\}}) = \ol{\{ \varphi(T)\}} \]
 by assumption, which implies that $\varphi(S) \subseteq \varphi(T)$. Hence, $\varphi$ is order-preserving. 
 
 We show that (3) implies (1). Suppose that $\varphi$ is order-preserving. First, observe that we have $\varphi(S \cap T) = \varphi(S) \cap \varphi(T)$ for any $S,T \subseteq \cP$, since $S \cap T$ is the largest subset of $\cP$ contained in both $S$ and $T$. Next, the sets $\{\p\}^c$ for $\p \in \cP$ are precisely the second-maximal subsets of $\cP$ in the sense that they are the maximal elements of $2^\cP \setminus \{\cP\}$. Since $\varphi$ is order-preserving, $\varphi$ induces a bijection from the family of second-maximal subsets of $\cP$ onto the family of second-maximal subsets of $\cQ$. Let $\Psi \colon \cP \to \cQ$ be the bijection characterised by $\varphi({\{\p\}^c})=\{\Psi(\p)\}^c$ for all $\p \in \cP$. 
 Then, for any $\p \in \cP$, we have $\varphi(\{\p\}^c)=\Psi(\{\p\}^c)$ (the right-hand side denotes the image of a subset by a map, whereas the left-hand side denotes the image of a point by a map). For $S \subseteq \cP$, we have 
 \[ \tilde{\Psi}(S) = \Psi \left(\bigcap_{\p \in \cP \setminus S} \{\p\}^c \right)
 = \bigcap_{\p \in \cP \setminus S} \Psi \left( \{\p\}^c \right) 
 = \bigcap_{\p \in \cP \setminus S} \varphi \left( \{\p\}^c \right) 
 = \varphi \left(\bigcap_{\p \in \cP \setminus S} \{\p\}^c \right)
 = \varphi(S), 
 \]
 which implies that $\varphi=\tilde{\Psi}$. 
\end{proof}

\subsection{A theorem of Williams}
\label{sec:Williams}
Let $G$ be a countable abelian group acting on a locally compact second-countable Hausdorff space $X$ by homeomorphism. 
The quasi-orbit of a point $x\in X$ is the orbit closure $[x]:=\overline{Gx}$. The quasi-orbit space $\cQ(X/G)$ is the quotient of $X$ by the equivalence relation 
\[
x\sim y \quad\text{if}\quad \overline{Gx}=\overline{Gy}.
\]

For $x\in X$, let $G_x:=\{g\in X : gx=x\}$ be the isotropy group of $x$. Because $G$ is abelian, isotropy groups are constant on quasi-orbits, that is, if $\ol{Gx}=\ol{Gy}$, then $G_x=G_y$. 
For each $x\in X$, let $\ev_x$ be the character (that is, nonzero one-dimensional representation) of $C_0(X)$ given by $\ev_x(f)=f(x)$. Then, for every $\chi\in\widehat{G_x}$, the pair $(\ev_x,\chi)$ is a covariant representation of $(C_0(X),G_x)$. Let $\ev_x\rtimes\chi$ denote the corresponding character of $C_0(X)\rtimes G_x$. By \cite[Proposition~8.27]{Will:book}, the induced representation $\Ind_{G_x}^G(\ev_x\rtimes\chi)$ of $C_0(X)\rtimes G$ is irreducible.
We need a slight reformulation of a theorem of Williams from  \cite[Theorem~8.39]{Will:book}, as given in \cite[Theorem~1.1]{LR:00}. Define an equivalence relation on $\cQ(X/G)\times\widehat{G}$ by 
\[
([x],\gamma)\sim ([y],\chi) \quad\text{if}\quad [x]=[y]\text{ and } \gamma\vert_{G_x}=\chi\vert_{G_x}.
\]

\begin{theorem}[{\cite[Theorem~8.39]{Will:book}}]
	\label{thm:Williams}
The map $q\colon ([x],\gamma)\mapsto \Ker \Ind_{G_x}^G(\ev_x\rtimes\gamma)$
is an open surjection that descends to a homeomorphism $(\cQ(X/G)\times\widehat{G})/\sim\xrightarrow{\simeq}\Prim(C_0(X)\rtimes G)$.
\end{theorem}
\begin{proof}
For openness of $q$, see \cite[Remark~8.40]{Will:book}. The rest is \cite[Theorem~1.1]{LR:00}.
\end{proof}

We now make several observations in the setting of Theorem~\ref{thm:Williams}.

\begin{remark}[cf. {\cite[Remark~3.5]{Tak2}}]
\label{rmk:primideals}
Let $x \in X$ and $\chi \in \widehat{G_x}$. 
The primitive ideal $\Ker \Ind_{G_x}^G \ev_x\rtimes\chi$ can be described concretely as follows: Extend $\chi$ to a character of $G$, and define a representation 
\[
\pi_{x,\chi}\colon C_0(X)\rtimes G\to\Bz(\ell^2(G/G_x)),\quad \pi_{x,\chi}(fu_g)\delta_{\ol{h}}=f(ghx)\chi(g)\delta_{\ol{gh}},
\]
where $\{\delta_{\ol{h}}: \ol{h}\in G/G_x\}$ is the standard orthonormal basis for $\ell^2(G/G_x)$. Up to unitary equivalence, $\pi_{x,\chi}$ does not depend on the choice of extension, and $\Ker\Ind_{G_x}^G(\ev_x\rtimes\chi)=\Ker\pi_{x,\chi}$.
\end{remark}

Next, we observe that $C_0(X)\rtimes G$ has a canonical structure as a C*-algebra over $\cQ(X/G)$.
\begin{lemma}
\label{lem:Williams}
There is a continuous, open surjective map $\psi\colon \Prim(C_0(X)\rtimes G)\to \cQ(X/G)$ such that the following diagram commutes: 
\begin{equation}
\label{eqn:overQ(X/G)}
    \begin{tikzcd}
    \cQ(X/G)\times\widehat{G}\arrow[d,"q"] \arrow[rr,"{([x],\chi)\mapsto [x]}"] & & \cQ(X/G)\\
    \Prim(C_0(X)\rtimes G)\nospacepunct{.} \arrow[rru,"\psi"]& &
    \end{tikzcd}
\end{equation}
\end{lemma}
\begin{proof}
Existence of such a continuous, surjective map $\psi$ follows from Theorem~\ref{thm:Williams}. The map $q$ is open by Theorem~\ref{thm:Williams}, so $\psi$ is open.
\end{proof}

We generalise \cite[Lemma~3.6]{Tak2} and \cite[Proposition~3.8]{Tak1}. For $[x]=\ol{Gx}\in\cQ(X/G)$, let
\begin{equation}
\label{eqn:P[x]}
P_{[x]}:=\bigcap_{\chi\in\widehat{G_x}}\Ker\pi_{x,\chi}
\end{equation}
be the intersection over the inverse image of $[x]$ under $\psi \colon \Prim(C_0(X)\rtimes G)\to \cQ(X/G)$.
The proof of \cite[Lemma~3.6]{Tak2} can be adapted to our general situation giving the following result. The details of the proof are left to the reader.
\begin{proposition}
\label{prop:P[x]}
For $[x]\in\cQ(X/G)$, we have $P_{[x]}=C_0(X\setminus [x])\rtimes G$.
\end{proposition}

We regard $A=C_0(X) \rtimes G$ as a C*-algebra over $\cQ(X/G)$ via $\psi$. 
Let $\xi \colon X \to \cQ(X/G)$ be the quotient map. 
\begin{proposition}
\label{prop:A(Z)}
Let $Z \subseteq \cQ(X/G)$ be a locally closed set. Then, we have 
\[ A(Z) = C_0(\xi^{-1}(Z)) \rtimes G. \]
\end{proposition}
\begin{proof}
First, we consider the case that $Z$ is an open set. 
Then, we have 
\[ A(Z) = \bigcap_{ P \in \psi^{-1}(Z^c)} P = \bigcap_{[x] \in Z^c} \bigcap_{\psi(P) = [x]} P = \bigcap_{[x] \in Z^c} P_{[x]}
 = C_0\left(\Interior\Biggl( \bigcap_{[x] \in Z^c} (X \setminus [x]) \Biggr) \right) \rtimes G, \]
where the last equality uses Proposition~\ref{prop:P[x]} and \cite[Proposition~1.7]{Sie}. In addition, we have, 
\[ \bigcap_{[x] \in Z^c} (X \setminus [x])  = X \setminus \bigcup_{[x] \in Z^c} [x] = X \setminus \xi^{-1}(Z^c) = \xi^{-1}(Z). \]
Since $\xi^{-1}(Z)$ is open, the claim holds in this case. For a general locally closed set $Z$, let $U,V \subseteq \cQ(X/G)$ be open sets with $V \subseteq U$ and $Z= U \setminus V$. Then, 
\[ A(Z) = A(U)/A(V) = C_0(\xi^{-1}(U) \setminus \xi^{-1}(V)) \rtimes G = C_0(\xi^{-1}(Z)) \rtimes G.\qedhere \]
\end{proof}
Note that $\xi^{-1}(Z)$ in Proposition~\ref{prop:A(Z)} is always locally compact and Hausdorff. 

\subsection{Auxiliary C*-algebras and primitive ideals}
\label{sec:aux}

There is a canonical action $\Gamma_K=K^*/\mu_K\acts \Az_{K,f}/\mu_K$, where we identify $\mu_K$ with a subgroup of $\Az_{K,f}$ via the diagonal embedding. Similarly, we have a canonical action $\Gamma_K\acts \Az_{K,f}/\ol{\O}_K^*$. The following C*-algebras play an important role in this article.

\begin{definition}
\label{def:BKandBval}
We let
\[
\fB_K:=C_0(\Az_{K,f}/\mu_K)\rtimes\Gamma_K
\]
be the \emph{C*-algebra modulo roots of unity}, and we let
\[
\fB_{\val}:=C_0(\Az_{K,f}/\ol{\O}_K^*)\rtimes\Gamma_K
\]
be the \emph{valuation C*-algebra}.
\end{definition}

We now describe the quasi-orbit spaces for the actions $K^*\acts\Az_{K,f}$, $\Gamma_K\acts \Az_{K,f}/\mu_K$, and $\Gamma_K\acts \Az_{K,f}/\ol{\O}_K^*$. When it is necessary to make a distinction, given $a\in \Az_{K,f}$, we shall let $\dot{a}$ and $\ol{a}$ denote the images of $a$ in $\Az_{K,f}/\mu_K$ and $\Az_{K,f}/\ol{\O}_K^*$, respectively.
For $a=(a_\p)_\p\in\Az_{K,f}$, let $\cZ(a):=\{\p\in\cP_ K: a_\p=0\}$. Note that $\cZ(a)$ only depends on the $\ol{\O}_K^*$-orbit of $a$, so it makes sense to define $\cZ(\ol{a}):=\cZ(a)$ and $\cZ(\dot{a}):=\cZ(a)$.
Part (i) from the following result is known for the case $K=\Qz$, see \cite[Lemma~3.2]{LR:00}.

\begin{lemma}
\label{lem:quasi-orbits}
Let $a=(a_\p)_\p\in\Az_{K,f}$. Then, 
\begin{enumerate}[\upshape(i)]
    \item $\overline{K^*a}=\{b\in\Az_{K,f} : \cZ(a)\subseteq \cZ(b) \}$;
    \item $\overline{\Gamma_K\dot{a}}=\{\dot{b}\in\Az_{K,f}/\mu_K : \cZ(a)\subseteq \cZ(b) \}$;
    \item $\overline{\Gamma_K\ol{a}}=\{\ol{b}\in\Az_{K,f}/\ol{\O}_K^* : \cZ(a)\subseteq \cZ(b) \}$.
\end{enumerate}
\end{lemma}
\begin{proof}
In all cases, the inclusion ``$\subseteq$'' is easy to see. We now prove that ``$\supseteq$'' holds in statement (i). Let $b\in\Az_{K,f}$ be such that $b_\p=0$ if $a_\p=0$, and let $U$ be an open set containing $b$; we may assume $U$ is of the form $U=\prod_{\p\in \cP_K}U_\p$, where $U_\p\subseteq K_\p$ is open and $U_\p=\cO_{K,\p}$ for all but finitely many $\p$. 
Then,
\[
F = \{ \p \in \cP_K \colon a_\p \not\in \cO_{K,\p}\}\cup\{\p \in \cP_K\setminus\cZ(a)\colon  U_\p\neq \O_{K,\p}\}\subseteq\cP_K\setminus\cZ(a)
\]
is finite, and we let $S= \cP_K \setminus \cZ(a)$. 
By Lemma~\ref{lem:approx}, $K^*$ is dense in $\prod'_{\p \in S} (K_\p,\cO_{K,\p})$, so we can take $k \in K^* \cap \prod_F a_\p^{-1} U_\p \times \prod_{S \setminus F} \cO_{K,\p}$. 
Then, we have $k\in a_\p^{-1}U_\p$ for all $\p \in S$. For $\p\in \cZ(a)$, we have $ka_\p=0=a_\p=b_\p\in U_\p$ for all $k\in K^*$. Thus, $ka_\p\in U_\p$ for all $\p\in \cZ(a)$, so that we have $ka\in U$.

Since the quotient maps $\Az_{K,f}\to \Az_{K,f}/\mu_K$ and $\Az_{K,f}\to \Az_{K,f}/\ol{\O}_K^*$ are open and the actions of $\mu_K$ and $\ol{\O}_K^*$ commute with the action of $K^*$, the argument above implies that ``$\supseteq$'' holds in statements (ii) and (iii) also.
\end{proof}

The description of the quasi-orbit spaces is now obtained in a similar fashion to \cite[Proposition~2.4]{LR:00}, which combined with Lemma~\ref{lem:Williams} gives the following. 

\begin{proposition}
\label{prop:quasi-orbits}
The map $\Az_{K,f}\to 2^{\cP_K}$ given by $a\mapsto\cZ(a)$ descends to give homeomorphisms from each of $\cQ(\Az_{K,f}/K^*)$, $\cQ((\Az_{K,f}/\mu_K)/\Gamma_K)$, and $\cQ((\Az_{K,f}/\ol{\O}_K^*)/\Gamma_K)$ onto $2^{\cP_K}$.
Therefore, each of $\fA_K$, $\fB_K$, and $\fB_\val$ is a C*-algebra over $2^{\cP_K}$.
\end{proposition}

Let $\psi_K$ be the composition $\Prim(\fA_K)\to \cQ(\Az_{K,f}/K^*)\simeq 2^{\cP_K}$, where the first map is from Lemma~\ref{lem:Williams} and the second is from Proposition~\ref{prop:quasi-orbits}.
For a locally closed set $Z \subseteq 2^{\cP_K}$, the C*-algebra $\fA_K(Z)$ is a crossed product, and its diagonal is the space of adeles whose zero sets are elements of $Z$ by Proposition~\ref{prop:A(Z)}. Similar results also hold for $\fB_K$ and $\fB_\val$.

\begin{definition}
\label{def:ideals}
For each subset $S\subseteq\cP_K$, let $P_S$ denote the ideal of $\fA_K$ associated via Equation~\eqref{eqn:P[x]} to the quasi-orbits corresponding to $S$ under the homeomorphisms from Proposition~\ref{prop:quasi-orbits}.
\end{definition} 
\vspace{-3mm}
The following result is the generalisation of \cite[Proposition~2.5]{LR:00} to arbitrary number fields, formulated in a slightly different manner. 

\begin{proposition}
\label{prop:primideals}
For each $S \in 2^{\cP_K}$, $\psi_K^{-1}(S)=\{P_S\}$ if $S \neq \cP_K$, and $\psi_K^{-1}(\cP_K)$ is homeomorphic to $\widehat{K^*}$. In paricular, we have a set-theoretic decomposition $\Prim(\fA_K) \cong \left(2^{\cP_K}\setminus\{\cP_K\}\right)\sqcup\widehat{K^*}$. 
Moreover, 
a point in $\Prim(\fA_K)$ is closed if and only if it lies in $\widehat{K^*}$.
\end{proposition}
\begin{proof}
Let $\rho \colon \cQ(\Az_{K,f}/K^*) \to 2^{\cP_K}$ be the homeomorphism from Proposition~\ref{prop:quasi-orbits}. For any $x \in \Az_{K,f}$, the isotropy group of $x$ is trivial if $\rho([x]) \neq \cP_K$, and is $K^*$ if $\rho([x]) = \cP_K$. Hence, the first claim follows by Theorem~\ref{thm:Williams} and the definition of $P_S$. We prove the second claim. It is easy to see that every point in $\widehat{K^*}$ is closed in $\Prim(\fA_K)$. Conversely, let $P \in \Prim(\fA_K)$ be a closed point and suppose $\psi_K(P)=S \neq \cP_K$. Then, we have $P=P_S$ and $\psi_K^{-1}(S)=\{P_S\}$ by the first assertion. Then, $\{S\}^c=\psi_K(\{P_S\}^c)$ is open by assumption and openness of $\psi_K$, $S$ is a closed point in $2^{\cP_K}$. Hence, we have $S=\cP_K$, which is a contradiction.
Hence, the second claim holds.
\end{proof}

\begin{remark}
As in \cite[Proposition~2.5]{LR:00} for the case $K=\Qz$, it is not difficult to describe the topology on the parameter space $\left(2^{\cP_K}\setminus\{\cP_K\}\right)\sqcup\widehat{K^*}$ explicitly using Proposition~\ref{prop:primideals} and the fact that $\psi_K$ is continuous and open.
\end{remark}

We now give a generalisation of the analogue of \cite[Lemma~2.10]{KT} for our situation.

\begin{lemma}
\label{lem:over2Sigma}
Let $K$ and $L$ be number fields, and suppose $\varphi\colon\Prim(\fA_K)\to \Prim(\fA_L)$ is a homeomorphism. Then, there exists a bijection $\theta\colon \cP_K\to \cP_L$ such that the following diagram commutes:
\[  \begin{tikzcd}
  \Prim(\fA_K)\arrow[d,"\psi_K"]\arrow[rr,"\varphi"] & & \Prim(\fA_L)\arrow[d,"\psi_L"]\\
  2^{\cP_K} \arrow[rr,"\tilde{\theta}"] & & 2^{\cP_L}\nospacepunct{,}
  \end{tikzcd}\]
  where $\tilde{\theta}$ denotes the homeomorphism $2^{\cP_K}\simeq 2^{\cP_L}$ induced by $\theta$. In particular, any *-isomorphism $\fA_K\xrightarrow{\cong}\fA_L$ is $2^{\cP_K}$-equivariant after identifying $2^{\cP_K}$ and $2^{\cP_L}$.
\end{lemma}
\begin{proof}

By Proposition~\ref{prop:primideals}, $2^{\cP_K}$ is the quotient space of $\Prim(\fA_K)$ obtained by identifying all closed points.
Under this identification, $\psi_K$ coincides with the quotient map. Therefore, there exists a unique homeomorphism $\tilde{\theta} \colon 2^{\cP_K} \to 2^{\cP_L}$ which makes the diagram commute. Moreover, $\tilde{\theta}$ comes from a bijection $\cP_K \to \cP_L$ by Lemma~\ref{lem:2Sto2T}. 
\end{proof}

\begin{definition} \label{def:phiK}
Let $\phi_K \colon \fB_K \to \fA_K$ be the *-homomorphism from Proposition~\ref{prop:crossedproduct1} associated with the trivial character of $K^*$ (in the notation of the proposition, we take $X=\Az_{K,f}$, $G=K^*$, and $\mu=\mu_K$). 
Let $\fval_K \colon \fB_\val \to \fB_K$ be the *-homomorphism induced from the canonical $\Gamma_K$-equivariant inclusion $C_0(\Az_{K,f}/\ol{\O}_K^*)\to C_0(\Az_{K,f}/\mu_K)$. 
\end{definition}

\begin{proposition}
The *-homomorphisms $\phi_K \colon \fB_K \to \fA_K$ and $\fval_K \colon \fB_\val \to \fB_K$ are $2^{\cP_K}$-equivariant.
\end{proposition}
\begin{proof}
By \cite[Lemma~2.8]{MeyerNest:09}, to show $\varphi_K$ is $2^{\cP_K}$-equivariant, it suffices to show $\varphi_K(\fB_K(U_F))\subseteq \fA_K(U_F)$ for every finite subset $F\subseteq \cP_K$. 
Let $\xi\colon \Az_{K,f}\to 2^{\cP_K}$ and $\dot{\xi}\colon \Az_{K,f}/\mu_K\to 2^{\cP_K}$ be the maps defined by $\xi(a)=\cZ(a)$ and $\dot{\xi}(\dot{a})=\cZ(a)$ for $a \in \Az_{K,f}$. By Proposition~\ref{prop:quasi-orbits}, $\xi$ and $\dot{\xi}$ coincide with the quotient maps $\Az_{K,f} \to \cQ(\Az_{K,f}/K^*)$ and $\Az_{K,f}/\mu_K \to  \cQ((\Az_{K,f}/\mu_K)/\Gamma_K)$, respectively, under the identifications $\cQ(\Az_{K,f}/K^*) \cong \cQ((\Az_{K,f}/\mu_K)/\Gamma_K) \cong 2^{\cP_K}$. 

Let $F\subseteq \cP_K$ be a finite subset.  
By Proposition~\ref{prop:A(Z)}, we have $\fA_K(U_F)=C_0(\xi^{-1}(U_F))\rtimes K^*$ and $\fB_K(U_F)=C_0(\dot{\xi}^{-1}(U_F))\rtimes\Gamma_K$. If $f\in C_0(\dot{\xi}^{-1}(U_F))$ and $\ol{g}\in\Gamma_K$, then $\varphi_K(fu_{\ol{g}})=p(f\circ\pi) u_g$, where $p=\frac{1}{|\mu|}\sum_{h\in\mu_K}u_h$ and $\pi\colon \Az_{K,f}\to\Az_{K,f}/\mu_K$ is the quotient map. Since $p(f\circ\pi) u_g=(f\circ\pi) pu_g$, $f\circ\pi\in C_0(\xi^{-1}(U_F))$, and $pu_g\in\spn\{u_h : h\in K^*\}$, we see that $\varphi_K(\fB_K(U_F))\subseteq \fA_K(U_F)$.

By Proposition~\ref{prop:A(Z)}, it is easy to see that $\fval_K(\fB_\val(U_F))\subseteq \fB_K(U_F)$ for every finite subset $F\subseteq \cP_K$, so, as in the first part of the proof, we see that $\fval_K$ is $2^{\cP_K}$-equivariant.
\end{proof}

\subsection{Subquotients and extensions from finite sets of primes}
\label{sec:subquotients}

As for the case of Bost--Connes C*-algebras from \cite[\S~2.3]{KT}, we explicitly describe subquotients associated with finite sets of primes. 
For every finite subset $F\subseteq \cP_K$ and $\p\in \cP_K \setminus F$, applying Proposition~\ref{prop:A(Z)} to the locally closed sets $\{F^c\}$ and $\{F^c,F_\p^c\}$ gives canonical *-isomorphisms
\vspace{-0.5cm}

\begin{minipage}{.4\linewidth}
\begin{align*}
    & \fA_K^F \cong C_0\left(\textstyle{\prod}_{\q\in F}K_\q^*\right)\rtimes K^*;\\
    &\fB_K^F \cong C_0\left((\textstyle{\prod}_{\q\in F}K_\q^*)/\mu_K\right)\rtimes \Gamma_K;\\
    &\fB_\val^F \cong C_0\left(\textstyle{\prod}_{\q\in F}\q^\Zz\right)\rtimes \Gamma_K;
    \end{align*}
    \end{minipage}
\begin{minipage}{.6\linewidth}
\begin{align*}
    &\fA_K(\{F^c,F_\p^c\})\cong C_0\left(K_\p\times\textstyle{\prod}_{\q\in F}K_\q^*\right)\rtimes K^*;\\
    &\fB_K(\{F^c,F_\p^c\})\cong C_0\left((K_\p\times\textstyle{\prod}_{\q\in F}K_\q^*)/\mu_K\right)\rtimes \Gamma_K;\\
    &\fB_\val(\{F^c,F_\p^c\})\cong C_0\left(\p^{\tilde{\Zz}}\times\textstyle{\prod}_{\q\in F}\q^\Zz\right)\rtimes \Gamma_K.
\end{align*}
\end{minipage}

Here, we use the fact that the map
\[
\Az_{K,f}\to \textstyle{\prod}_{\p\in\cP_K}'(\p^{\tilde{\Zz}},\p^{\tilde{\Nz}}),\quad a\mapsto (\p^{v_\p(a)})_\p
\]
descends to a $\Gamma_K$-equivariant homeomorphism $\Az_{K,f}/\ol{\O}_K^*\simeq \textstyle{\prod}_{\p\in\cP_K}'(\p^{\tilde{\Zz}},\p^{\tilde{\Nz}})$, where $\p^{\tilde{\Zz}}=\{\p^n : n\in\tilde{\Zz}\}\simeq\tilde{\Zz}$. The group $\Gamma_K$ acts on $\p^{\tilde{\Zz}}$ via the homomorphism $v_\p$ and on the restricted product diagonally.

\begin{remark}
\label{rmk:level0}
For $F=\emptyset$, the above gives $\fA_K^\emptyset\cong C^*(K^*)$ and $\fB_K^\emptyset\cong\fB_\val^\emptyset\cong C^*(\Gamma_K)$.
\end{remark}

Let $F\subseteq\cP_K$ be a finite subset. Recall that the C*-algebra $B_K^F$ from Definition~\ref{def:BKF} is the unital part of the composition factor $\fB_K^F$. Note that the notations $\fB_K^F$ and $B_K^F$ are different from those in \cite{KT}. 
Since $(\prod_{\p\in F} \cO_\p^*)/\mu_K$ is a compact open subgroup of $(\prod_{\p\in F} K_\p^*)/\mu_K$, we have a canonical inclusion map $B_K^F\to \fB_K^F$. Similarly, let $B_{\val}^F := C^*(\Gamma_K^F)$. Then, we have a canonical inclusion map $B_\val^F\to \fB_\val^F$. 

Let $F\subseteq\cP_K$ be a non-empty finite subset. Applying Lemma~\ref{lem:approx}, we can see that $\fB_K^F \cong \Kz \otimes B_K^F$ and $\fB_\val^F \cong \Kz \otimes B_\val^F$. In particular, $\fB_K^F$ and $\fB_\val^F$ are stable. This is why we call $B_K^F$ and $B_\val^F$ ``the unital parts''. However, we do not use these *-isomorphisms directly, since they are not canonical. Instead, we rely on Proposition~\ref{prop:tensor}.

\begin{definition}
Let $F$ be a finite subset of $\cP_K$. Define $\bs{\xi}_K^F \in \KK(B_K^F, \fB_K^F)$ and $\bs{\xi}_\val^F \in \KK(B_\val^F, \fB_\val^F)$ to be the elements induced from the inclusion maps $B_K^F\to \fB_K^F$ and $B_\val^F\to \fB_\val^F$, respectively. In addition, define $\xi_K^F \colon \K_*(B_K^F) \to \K_*(\fB_K^F)$ and $\xi_\val^F \colon \K_*(B_\val^F) \to \K_*(\fB_\val^F)$ to be the homomorphisms of \K-groups induced from $\bs{\xi}_K^F$ and $\bs{\xi}_\val^F$, respectively. 
\end{definition}

\begin{lemma}\label{lem:orderisom}
    For every finite subset $F$ of $\cP_K$, the elements $\bs{\xi}_K^F$ and $\bs{\xi}_{\val}^F$ are ordered \KK-equivalences.
\end{lemma}
\begin{proof}
First, $\Gamma_K^F$ is a summand of $\Gamma_K$ by Lemma~\ref{lem:GammaKdecomp}.
Let $\lambda\colon \Gamma_K\to X:=(\prod_{\p\in F} K_\p^*)/\mu_K$ be the canonical homomorphism. By Lemma~\ref{lem:approx}, the image of $\lambda$ is dense. Let $X_0:=(\prod_{\p\in F} \cO_\p^*)/\mu_K$. We have
\[
\ol{\lambda(\Gamma_K^F)}=\ol{X_0\cap\lambda(\Gamma_K)}=X_0\cap\ol{\lambda(\Gamma_K)}=X_0, 
\]
where the second equality uses that $X_0$ is compact open in $X$, and the third equality uses that $\lambda(\Gamma_K)$ is dense in $X$. The claim for $\bs{\xi}_K^F$ now follows from Proposition~\ref{prop:tensor}. Similarly, the claim for $\bs{\xi}_\val^F$ follows from Proposition~\ref{prop:tensor} (in this case, $X_0=\{1\}$).
\end{proof}

\begin{remark}
\label{rmk:varphiKF}
Let $F\subseteq\cP_K$ be a nonempty finite subset. Applying Proposition~\ref{prop:projmu} with $G=K^*$, $X=\prod_{\p\in F}K_\p^*$, and $\mu=\mu_K$ gives a *-homomorphism $\Phi\colon \fB_K^F\to\fA_K^F$. Using the explicit description of our composition factors above, we see that this *-homomorphism coincides with the *-homomorphism $\varphi_K^F\colon \fB_K^F\to\fA_K^F$ induced by $\varphi_K$. In particular, $\varphi_K^F$ induces an ordered \KK-equivalence for every nonempty finite subset $F\subseteq\cP_K$. In fact, $\fA_K^F$ is *-isomorphic to a matrix algebra over $\fB_K^F$ (and is thus also stable).
\end{remark}

We now define three families of extensions that will play a fundamental role in our reconstruction theorem. 
\begin{definition}
For each finite subset $F\subseteq\cP_K$ and $\p\in \cP_K \setminus F$, let $\cF_K^{F,\p}:=\cE_{\fA_K}^{F,\p}$, $\cE_K^{F,\p}:=\cE_{\fB_K}^{F,\p}$, and $\cE_\val^{F,\p}:=\cE_{\fB_\val}^{F,\p}$, where the extensions on the right hand sides are from \eqref{eqn:Efp}.
\end{definition}
\vspace{-3mm}

We close this section with a decomposition of the extension $\cE_\val^F$ in terms of the dilated Toeplitz extension $\cT$ from Definition~\ref{def:Toeplitz}.
Unlike in the case of Bost--Connes C*-algebras in \cite{KT}, the decomposition is not canonical; this is a fundamental technical difference between our work and \cite{KT}.

\begin{lemma}
\label{lem:E_val}
Let $F$ be a finite subset of $\cP_K$, and let $\p \in \cP_K \setminus F$. Fix $\pi_\p \in \Gamma_K^F$ satisfying $v_\p(\pi_\p) = 1$. 
Let $\cE \in \cExt(B_\val^F, \Kz \otimes B_\val^{F_\p})$ be the extension 
    \begin{equation} 
    \cE \colon 0\to (C_0(\p^{\Zz}) \rtimes \pi_\p^\Zz) \otimes C^*(\Gamma_K^{F_\p}) \to (C_0(\p^{\tilde{\Zz}}) \rtimes \pi_\p^\Zz) \otimes C^*(\Gamma_K^{F_\p}) \to C^*(\pi_\p^\Zz) \otimes C^*(\Gamma_K^{F_\p}) \to 0.   
    \end{equation}
Then, $\cE$ is canonically isomorphic to $\cT \otimes C^*(\Gamma_K^{F_\p})$. Moreover, by identifying $(C_0(\p^{\Zz}) \rtimes \pi_\p^\Zz) \otimes C^*(\Gamma_K^{F_\p})$ and $C^*(\Gamma_K^{F_\p})$ in $\KK$, $\bs{\xi} := \bigl( \bs{\xi}_\val^F,\, 
\bs{\xi}_\val^{F_\p}
\bigr) \colon \bigl[ \cE \bigr]_{\KK} \to \bigl[ \cE_\val^{F,\p} \bigr]_{\KK}$ is an isomorphism in $\Arr(\KK)$. 
\end{lemma}
\begin{proof}
First, note that such $\pi_\p$ exists by Lemma~\ref{lem:approx}.
It is clear that $\cE$ is canonically isomorphic to $\cT \otimes C^*(\Gamma_K^{F_\p})$. 
The extension $\cE_\val^F$ is given explicitly by 
\[
\cE_\val^{F,\p}\colon  0\to C_0\left(\p^{\Zz}\times\textstyle{\prod}_{ F}\q^\Zz\right)\rtimes \Gamma_K\to  C_0\left(\p^{\tilde{\Zz}}\times\textstyle{\prod}_{F}\q^\Zz\right)\rtimes \Gamma_K \to C_0\left(\textstyle{\prod}_{F}\q^\Zz\right)\rtimes \Gamma_K \to 0.
\]
In particular, there is a canonical homomorphism $\cE \to \cE_\val^F$ consisting of inclusion maps, which induces the morphism $\bs{\xi}$ in $\Arr(\KK)$. 
Moreover, $\bs{\xi}$ is an isomorphism in $\Arr(\KK)$ by Lemma~\ref{lem:orderisom}.
\end{proof}

\section{Reconstruction of the dynamical systems} \label{sec:recon}

\subsection{The Rieffel correspondence}
\label{sec:Rieffel}
We collect basics on the Rieffel correspondence. Let $A$ and $B$ be Morita equivalent C*-algebras, and let $\cX$ be an $A$--$B$-imprimitivity bimodule.
Then, $\cX$ induces a lattice isomorphism $\fR \colon \Iz(B) \to \Iz(A)$, called the Rieffel correspondence, which restricts to a homeomorphism $\Prim(B) \to \Prim(A)$.
If $I \in \Iz(B)$, then $\fR(I)$ is the closed linear span of $_A\langle xb, y\rangle$, where $x,y \in \cX$ and $b \in I$ (see \cite[Proposition 3.24]{RW:book}). The next lemma easily follows from the definition of the Rieffel correspondence. 
\begin{lemma} \label{lem:Rieffel}
Let $A$ be a C*-algebra, and let $p \in M(A)$ be a full projection. Let $\fR \colon \Iz(A) \to \Iz(pAp)$ be the Rieffel correspondence for the $pAp$--$A$-imprimitivity bimodule $pA$. Then, for each $I \in \Iz(A)$, we have 
\[ \fR(I) = pIp = pAp \cap I.\]
\end{lemma}

Now assume that $A$ is a C*-algebra over a topological space $X$, and let $\psi \colon \Prim(A) \to X$ be the associated surjection. 
Let $p \in M(A)$ be a full projection.  We consider $pAp$ as a C*-algebra over $X$ by equipping it with the surjection $pAp \to X$ which makes the following diagram commute: 
\[  \begin{tikzcd}
  \Prim(A)\arrow[rd,"\psi"]\arrow[rr,"\fR"] & & \Prim(pAp)\arrow[ld]\\
   & X\nospacepunct{.}&
  \end{tikzcd}\]

Let $Z \subseteq X$ be a locally closed subset, and let $U,V \subseteq X$ be open sets such that $V \subseteq U$ and $Z=U \setminus V$. 
We denote by $p^Z \in M(A(Z))$ the image of $p$ under the composition of canonical maps $M(A) \to M(A/A(V)) \to M(A(U)/A(V))$. The proof of the next lemma is straightforward by Lemma~\ref{lem:Rieffel}:
\begin{lemma} \label{lem:Rproj}
For every locally closed subset $Z \subseteq X$, $(pAp)(Z)$ is canonically isomorphic to $p^ZA(Z)p^Z$. 
\end{lemma}

\subsection{Reduction from \texorpdfstring{$\fA_K$}{Lg} to \texorpdfstring{$\fB_K$}{Lg}}
\label{sec:AKtoBK}
Let $p \in M(\fA_K)$ be a full projection. 
Fix a finite subset $F$ of $\cP_K$ and $\p \in \cP_K\setminus F$. 
Since the inclusion map $p\fA_Kp \to \fA_K$ is a $2^{\cP_K}$-equivariant *-homomorphism by Lemma~\ref{lem:Rproj}, we have the following commutative diagram with exact rows: 
 \begin{equation} \begin{tikzcd}
	\Kz \otimes p\cF_K^{F,\p}p \colon 0\arrow{r}&	\Kz \otimes (p\fA_Kp)^{F_\p} \arrow[d]\arrow[r]  & \Kz \otimes (p\fA_Kp)(\{F^c, F_\p^c\}) \arrow[d] \arrow[r] & \Kz \otimes (p\fA_Kp)^F\arrow[d] \arrow[r] & 0\\
	\Kz \otimes \cF_K^{F,\p} \colon 0\arrow{r}&	\Kz \otimes \fA_K^{F_\p} \arrow[r]  & \Kz \otimes \fA_K(\{F^c, F_\p^c\})\arrow[r] & \Kz \otimes \fA_K^F\arrow[r] & 0\nospacepunct{.}
 \end{tikzcd}\end{equation}
Here, the vertical maps are the inclusion maps tensored with $\id_\Kz$. 
Let 
\[\bigl[ I_{K,p}^{F,\p} \bigr]_{\KK} \colon \bigl[ \Kz \otimes p\cF_K^{F,\p}p \bigr]_{\KK} \to \bigl[ \Kz \otimes \cF_K^{F,\p} \bigr]_{\KK} \to \bigl[ \cF_K^{F,\p} \bigr]_{\KK}\]
denote the composition of morphisms in $\Arr(\KK)$, where the first morphism is induced from the above diagram and the second morphism is the inverse of the morphism $\bigl[ \cF_K^{F,\p} \bigr]_{\KK} \to \bigl[ \Kz \otimes \cF_K^{F,\p} \bigr]_{\KK}$ in $\Arr(\KK)$ induced by the canonical inclusion $\cF_K^{F,\p} \to \Kz \otimes \cF_K^{F,\p}$.
By Lemma~\ref{lem:Rproj}, $(p\fA_Kp)^F$ is a full corner in $\fA_K^F$, so that $\bigl[ I_{K,p}^{F,\p} \bigr]_{\KK}$ is an order isomorphism in $\Arr(\KK)$.

\begin{lemma}\label{lem:admissible}
Let $P=P_{\cP_K}$, and let $p \in M(\fA_K)$ be a full projection. Then, the inclusion map $p\fA_Kp \to \fA_K$ descends to a *-isomorphism $p\fA_Kp/pPp \to \fA_K/P$. 
\end{lemma}
\begin{proof}
Since $\fA_K/P$ is unital, the quotient map $\fA_K \to \fA_K/P$ extends to a surjective *-homomorphism 
$\pi \colon M(\fA_K) \to \fA_K/P$. Then, $\pi(p)$ is a full projection. Since $\fA_K/P \cong C^*(K^*)$ is an abelian C*-algebra (see Remark~\ref{rmk:level0}), the unit $1_{C^*(K^*)}$ is the unique full projection of $\fA_K/P$. Hence, $\pi(p)=1$. 
\end{proof}

\begin{proposition}\label{prop:commutativeKL}
Let $K$ and $L$ be number fields, and let $p \in M(\fA_K)$ and $q \in M(\fA_L)$ be full projections. Suppose $p\fA_Kp$ and $q\fA_Lq$ are *-isomorphic. Then, there exists a bijection $\theta \colon \cP_K \to \cP_L$, a *-isomorphism $\alpha^\emptyset \colon \fB_K^\emptyset \to \fB_L^\emptyset$, and a family of ordered \KK-equivalences $\bigl[\alpha^F\bigr]_{\KK} \in \KK\bigl(\fB_K^F,\, \fB_L^{\theta(F)} \bigr)$
for each nonempty finite subset $F \subseteq \cP_K$ such that 
\[ \bs{\alpha}^{F,\p} := (\bigl[\alpha^F\bigr]_{\KK},\, \bigl[\alpha^{F_\p}\bigr]_{\KK}) \colon \bigl[\cE_K^{F,\p} \bigr]_{\KK} \to \bigl[ \cE_L^{\theta(F),\theta(\p)} \bigr]_{\KK}
\]
is an order isomorphism in $\Arr(\KK)$ 
for any (possibly empty) finite subset $F\subseteq \cP_K$ and $\p \in \cP_K \setminus F$. Moreover, we have $|\mu_K|=|\mu_L|$.
\end{proposition}
\begin{proof}
 Let $\tilde{\alpha} \colon p\fA_Kp \to q\fA_Lq$ be a *-isomorphism. By Lemma~\ref{lem:admissible}, we have canonical *-isomorphisms $\fA_K^\emptyset \cong (p\fA_Kp)^\emptyset$ and $\fA_L^\emptyset\cong (q\fA_Lq)^\emptyset$. Let $\delta$ be the composition 
 \[ 
 C^*(K^*)\cong \fA_K^\emptyset \cong (p\fA_Kp)^\emptyset \xrightarrow{\tilde{\alpha}^\emptyset} (q\fA_Lq)^\emptyset \cong \fA_L^\emptyset \cong C^*(L^*),
 \]
 where the first and last *-isomorphisms are from Remark~\ref{rmk:level0}.
 By Lemma \ref{lem:minproj}, there exists $\chi \in \widehat{\mu_L}$ such that $\delta(p_1^K)=p_\chi^L$, where $p_\chi^L \in C^*(L^*)$ denotes the projection corresponding to $\chi$, and $p_1^K \in C^*(K^*)$ is the projection corresponding to the trivial character of $K^*$. Fix an extension of $\chi$ to $L^*$, and let $\tau \colon \widehat{L^*} \acts \fA_L$ be the dual action. Then, we have $\tau^\emptyset_\chi(p_1^L)=p_\chi^L$, where $p_1^L\in C^*(L^*)$ is the projection corresponding to the trivial character of $L^*$. 
 
 Consider the composition $\varphi\colon \Prim(\fA_K)\simeq \Prim(p\fA_Kp) \simeq \Prim(q\fA_Lq)\simeq \Prim(\fA_L)$, where the first and last homeomorphisms are the Rieffel correspondences and the middle homeomorphism is induced by $\tilde{\alpha}$. By applying Lemma~\ref{lem:over2Sigma} to the homeomorphism $\varphi$, there is a unique bijection $\theta\colon \cP_K \to \cP_L$ such that $\tilde{\alpha}$ is $2^{\cP_K}$-equivariant under the identification $2^{\cP_K}$ and $2^{\cP_L}$ via $\theta$.
 
 Let $F$ be a nonempty subset of $\cP_K$ and let $\p \in \cP_K \setminus F$. Let 
 \[\bs{\alpha}^{F,\p} := \bigl[\varphi_L^{\theta(F),\theta(\p)} \bigr]_{\KK}^{-1} \circ \bigl[\tau_\chi^{\theta(F),\theta(\p)} \bigr]_{\KK}^{-1} \circ \bigl[I_{L,q}^{\theta(F),\theta(\p)} \bigr]_{\KK} \circ \bigl[\tilde{\alpha}^{F,\p} \bigr]_{\KK} \circ \bigl[I_{K,p}^{F,\p} \bigr]_{\KK}^{-1} 
 \circ \bigl[\varphi_K^{F,\p} \bigr]_{\KK}. 
 \]
 By Remark~\ref{rmk:varphiKF}, $\bigl[\varphi_K^{F,\p} \bigr]_{\KK}$ and $\bigl[\varphi_L^{\theta(F),\theta(\p)} \bigr]_{\KK}$ are order isomorphisms in \Arr(\KK). Thus, $\bs{\alpha}^{F,\p}$ is an order isomorphism in $\Arr(\KK)$ since it is a composition of order isomorphisms in $\Arr(\KK)$. Let $\bs{\alpha}^{F,\p} := (\boldsymbol{x},\boldsymbol{y})$. We can directly see that $\boldsymbol{x}$ does not depend on $\p$. Let $[\alpha^{F}]_{\KK} :=\boldsymbol{x}$. Then, we can see that $\boldsymbol{y} = [\alpha^{F_\p}]_{\KK}$ by construction. Hence, the claim holds for $F \neq \emptyset$. When $F=\emptyset$, let $\alpha^\emptyset$ be the composition of 
 \[ \fB_K^\emptyset \xrightarrow{\varphi_K^\emptyset} \fA_K^\emptyset = (p\fA_Kp)^\emptyset \xrightarrow{\tilde{\alpha}^\emptyset} (q\fA_Lq)^\emptyset = \fA_L^\emptyset \xrightarrow{(\tau_\chi^\emptyset)^{-1}} \fA_L^\emptyset \xleftarrow{\varphi_L^\emptyset} \fB_L^\emptyset. \]
 Note that using the isomorphism from Proposition~\ref{prop:main6.3}, this composition makes sense since the image of the composition of the first three maps is equal to the image of $\varphi_L^\emptyset$. 
 Then, $\alpha^\emptyset$ is a *-isomorphism, and for every $\p \in \cP_K$, we see that $\bs{\alpha}^{\emptyset,\p} :=\bigl( \bigl[\alpha^{\emptyset}\bigr]_{\KK},\,\bigl[\alpha^{\{\p\}}\bigr]_{\KK} \bigr)$ is an order isomorphism in $\Arr(\KK)$. 

By Lemma~\ref{lem:admissible} and Remark~\ref{rmk:level0}, $(p\fA_Kp)^\emptyset \cong \fA_K^\emptyset \cong C^*(K^*) \cong C(\widehat{\mu_K} \times \Tz^\infty)$, so the last claim follows.
\end{proof}

For a finite subset $F \subseteq \cP_K$ and a prime $\p \in \cP_K \setminus F$, let $\partial_K^{F,\p} \colon \K_*(\fB_K^F) \to \K_*(\fB_K^{F_\p})$ be the boundary map associated with the extension $\cE_K^{F,\p}$.  
For each finite subset $F$ of $\cP_K$, let 
\[ \alpha^F_* := -\ho{} \bigl[\alpha^{F} \bigr]_{\KK} \colon \K_*(\fB_K^F) \to \K_*(\fB_L^{\theta(F)}). \]
Then, each $\alpha^F_* $ is an order isomorphism between \K-groups. For every finite subset $F$ of $\cP_K$ and $\p \in \cP_K \setminus F$, the following diagram commutes:
  \[\begin{tikzcd}
	\K_*(\fB_K^{F}) \arrow[d, "\alpha^F_*"]\arrow[r, "\partial_K^{F,\p}"] & \K_{*+1}(\fB_K^{F_\p})\arrow[d,"\alpha^{F_\p}_*"]\\
	\K_*(\fB_L^{\theta(F)}) \arrow[r, "\partial_L^{F,\p}"] & \K_{*+1}(\fB_L^{\theta(F_\p)})\nospacepunct{.} 
\end{tikzcd}\]

\subsection{The main theorem}
\label{sec:mainresult}
We can now state the detailed version of our main theorem.
\begin{theorem} \label{thm:main2}
Let $K$ and $L$ be number fields with $|\mu_K|=|\mu_K|$, and let $\theta \colon \cP_K \to \cP_L$ be a bijection. Suppose that there exists a *-isomorphism $\alpha^\emptyset \colon \fB_K^\emptyset \to \fB_L^\emptyset$, and a family of isomorphisms between \K-groups $\alpha_*^F \colon \K_*(\fB_K^F) \to \K_*(\fB_L^{\theta(F)})$ for each nonempty finite subset $F \subseteq \cP_K$ with $1 \leq |F| \leq 3$ such that $\alpha_*^F$ is an order isomorphism if $|F|=1$, and the diagram 
  \begin{equation} \label{eqn:main2} \begin{tikzcd}
	\K_*(\fB_K^{F}) \arrow[d, "\alpha^F_*"]\arrow[r, "\partial_K^{F,\p}"] & \K_{*+1}(\fB_K^{F_\p})\arrow[d,"\alpha^{F_\p}_*"]\\
	\K_*(\fB_L^{\theta(F)}) \arrow[r, "\partial_L^{F,\p}"] & \K_{*+1}(\fB_L^{\theta(F_\p)}) 
\end{tikzcd}\end{equation}
commutes for every finite subset $F\subseteq \cP_K$ with $0 \leq |F| \leq 2$ and $\p \in \cP_K \setminus F$. Let $\gamma \colon \Gamma_K \to \Gamma_L$ be the isomorphism characterised by the following equation in $\K_1(\fB_L^\emptyset)$:
\begin{equation} \label{eqn:gammadef}
    \alpha^\emptyset_*([u_x]_1) = [u_{\gamma(x)}]_1 \quad \fa x \in \Gamma_K. 
\end{equation} 
Then, there exists a unique field isomorphism $\sigma \colon K \to L$ such that the following diagram commute: 
  \[\begin{tikzcd}
	K^* \arrow[r, "\sigma"] \arrow[d] & L^* \arrow[d] \\
	\Gamma_K \arrow[r, "\gamma"] & \Gamma_L\nospacepunct{,}
\end{tikzcd}\]
where the vertical maps are the quotient maps. 
\end{theorem}
Existence of $\gamma$ is proved shortly in Lemma~\ref{lem:existgamma}.
Sections \ref{sec:Cn}, \ref{sec:reconsemi-local}, and \ref{sec:NT} are devoted to the proof of theorem~\ref{thm:main2}. 
It will be convenient to work with slightly more general hypotheses than those in Theorem~\ref{thm:main2}. With this in mind, we make the following definition.

\begin{definition}
\phantomsection
\label{def:CN}
Let $K$ and $L$ be number fields, and let $N \in \Zz$ with $N \geq 1$. We denote by \hyperref[def:CN]{$(C_N)$} the condition that 
\begin{enumerate}[\upshape(i)]
    \item We have $|\mu_K|=|\mu_L|$ and a bijection $\theta \colon \cP_K \to \cP_L$. 
    \item For each nonempty finite subset $F \subseteq \cP_K$ with $1 \leq |F| \leq N$, there exists an isomorphism $\alpha_*^F \colon \K_*(\fB_K^F) \to \K_*(\fB_L^{\theta(F)})$ such that the Diagram~\eqref{eqn:main2} commutes for every finite subset $F \subseteq \cP_K$ with $1 \leq |F| \leq N-1$ and $\p \in \cP_K \setminus F$. 
    \item There exists a *-isomorphism $\alpha^\emptyset \colon \fB_K^\emptyset \to \fB_L^\emptyset$ such that the Diagram~\eqref{eqn:main2} commutes for $F=\emptyset$ and every $\p \in \cP_K$.
    \item The isomorphism $\alpha_*^F$ is an order isomorphism if $|F|=1$. \label{item:order}
\end{enumerate}

In addition, we say that $K$ and $L$ satisfy condition $(C_\infty)$ if  condition $(C_N)$ is satisfied for every $N$.
\end{definition}
\vspace{-3mm}
The assumption of Theorem~\ref{thm:main2} is equivalent to the condition \hyperref[def:CN]{$(C_3)$}. 
 For number fields $K$ and $L$, if full corners of $\fA_K$ and $\fA_L$ are isomorphic, then the condition \hyperref[def:CN]{$(C_\infty)$} is satisfied by Proposition~\ref{prop:commutativeKL}.  Although it is enough to have the condition \hyperref[def:CN]{$(C_3)$} for the construction of a field isomorphism between $K$ and $L$, we give a direct reconstruction result of the dynamical system $\Gamma_K \acts \Az_{K,f}/\mu_K$ in Section \ref{ssec:recondyn} from the condition \hyperref[def:CN]{$(C_\infty)$}. Therefore, we work with condition \hyperref[def:CN]{$(C_N)$} for general $N$. 

\subsection{Commutative diagrams from condition \texorpdfstring{\hyperref[def:CN]{$(C_N)$}}{Lg}}
\label{sec:Cn}
Fix $N \in \Zz$ with $N \geq 1$. Let us assume that $K$ and $L$ are number fields, and that condition \hyperref[def:CN]{$(C_N)$} is satisfied.
Throughout this subsection, we fix the bijection $\theta \colon \cP_K \to \cP_K$ and  the *-isomorphism $\alpha^\emptyset$ as in Definition \ref{def:CN}, and for every finite subset $F$ of $\cP_K$ with $1 \leq |F| \leq N$ and $\p \in \cP_K \setminus F$, fix the order isomorphism  $\alpha_*^F$ between \K-groups as in Definition \ref{def:CN}. 
We now prepare for the proof of Theorem~\ref{thm:main2} and for the reconstruction of the dynamical system $\Gamma_K \acts \Az_{K,f}/\mu_K$.

\begin{lemma} \label{lem:existgamma}
There exists a unique isomorphism $\gamma \colon \Gamma_K \to \Gamma_L$ satisfying Equation \eqref{eqn:gammadef}. 
\end{lemma}
\begin{proof}
For a unital C*-algebra $A$, let $\cU(A)$ denote the unitary group of $A$, and let $\cU_0(A)$ denote the connected component of $\cU(A)$ containing $1_A$.
 Since $\Gamma_K$ is a free abelian group, we have an isomorphism 
 \[ \gamma \colon \Gamma_K \cong \cU(\fB_K^\emptyset)/\cU_0(\fB_K^\emptyset) \xrightarrow{\alpha^\emptyset} \cU(\fB_L^\emptyset)/\cU_0(\fB_L^\emptyset) \cong \Gamma_L \]
 satisfying Equation~\eqref{eqn:gammadef} by \cite[Theorem~8.57(ii)]{HofMor}. 
 Moreover, Equation~\eqref{eqn:gammadef} characterises $\gamma$, since the homomorphism $\cU(A)/\cU_0(A) \to \K_1(A)$ is injective for any unital abelian C*-algebra $A$ by \cite[Proposition~8.3.1]{RLL}.
\end{proof}

\begin{remark} \label{rmk:dual}
We have an alternative description of $\gamma \colon \Gamma_K \to \Gamma_L$ in Theorem~\ref{thm:main2}. 
Let $(\alpha^\emptyset)^*$ denote the homeomorphism $\widehat{\Gamma_L} \to \widehat{\Gamma_K}$ induced by $\alpha^\emptyset$. Since $\widehat{\Gamma_K}$ is a path-connected group, there exists a homeomorphism $\widehat{\Gamma_L} \to \widehat{\Gamma_K}$ which fixes the identity and is homotopic to $(\alpha^\emptyset)^*$. Hence, by \cite[Corollary 2]{Scheffer}, $(\alpha^\emptyset)^*$ is homotopic to a group homomorphism $f \colon \widehat{\Gamma_L} \to \widehat{\Gamma_K}$. Then, $\gamma$ coincides with the dual $\widehat{f} \colon \Gamma_K \to \Gamma_L$ of $f$.
\end{remark}

Let $\Gamma$ be a free abelian group. We introduce notation for elements in $\K_*(C^*(\Gamma))$ based on \cite[\S~3]{KT}. 
For a finite rank oriented summand $\Lambda$ of $\Gamma$, let $\beta_\Lambda := [u_{x_1}]_1 \ho{} \cdots \ho{} [u_{x_n}]_1 \in \K_n(C^*(\Lambda)) \subseteq \K_*(C^*(\Gamma))$, where $(x_1,\dots,x_n)$ is an oriented basis of $\Lambda$. Here, $\K_*(C^*(\Lambda))$ is identified with $\K_*(C^*(\Zz x_1)) \ho{} \cdots \ho{} \K_*(C^*(\Zz x_n))$ via the K\"unneth formula. The element $\beta_\Lambda$ does not depend on the choice of an oriented basis of $\Lambda$, which can be verified via the identification $\K_*(C^*(\Lambda)) \cong \lwedge^* \Lambda$ as Hopf algebras by \cite[Theorem~II.2.1]{Hod}. If $\Lambda$ and $\Lambda'$ are the same finite rank summands equipped with opposite orientations, then $\beta_{\Lambda'}=-\beta_\Lambda$. 
Hence, for an unoriented finite rank summand $\Lambda$ with $\rank \Lambda = n$, we give an orientation of $\Lambda$ by specifying a generator of $\lwedge^n \Lambda \subseteq \K_*(C^*(\Lambda))$. 
Elements of the form $\beta_\Lambda$, where $\Lambda$ runs through finite rank summands with a fixed orientation, generate $\K_*(C^*(\Gamma))$. If $\rank \Lambda = 1$ and $\Lambda$ is generated by $x \in \Lambda$, then either $\beta_\Lambda = [u_x]_1$ or $\beta_\Lambda = - [u_x]_1$, depending on the orientation of $\Lambda$. In this case, $\beta_\Lambda$ is simply denoted by $\beta_x$. If $\rank \Lambda=0$ (that is, $\Lambda$ is a trivial summand), then we always choose an orientation of $\Lambda$ by $\beta_\Lambda = [1]_0$.

\begin{lemma} \label{lem:inclusion}
    For every finite subset $F$ of $\cP_K$, the composition 
    \[ (\ol{\fval}_K^F)_* := (\xi_K^F)^{-1} \circ (\fval_K^F)_* \circ \xi_\val^F \colon \K_*(C^*(\Gamma_K^F)) \to \K_*(B_K^F)\]
    coincides with the homomorphism induced from the canonical inclusion $\iota^F \colon C^*(\Gamma_K^F) \to B_K^F$. In particular, $(\ol{\fval}_K^\emptyset)_*$ is the identity map. 
\end{lemma}
\begin{proof}
        We have the following commutative diagram: 
    \begin{equation} \label{eqn:diagramBB} \begin{tikzcd}
      B_\val^F \arrow[d]\arrow[r, "\iota^F"] & B_K^F \arrow[d]\\
      \fB_\val^F \arrow[r, "\fval_K^F"] & \fB_K^F\nospacepunct{.} 
    \end{tikzcd} \end{equation}
    Here, the vertical maps denote the inclusion maps. Taking \K-groups of this diagram, we have $\xi_K^F \circ \iota^F_* = (\fval_K^F)_* \circ \xi_\val^F$, and hence $\iota^F_* = (\ol{\fval}_K^F)_*$. 
\end{proof}

For each finite subset $F\subseteq\cP_K$ with $|F| \leq N-1$ and $\p \in \cP_K \setminus F$, let $\partial_K^{F,\p}$ and $\partial_{\val}^{F,\p}$ denote the boundary maps associated with the extensions $\cE_K^{F,\p}$ and $\cE_\val^{F,\p}$, respectively, and let 
\[ \ol{\partial}_{\val}^{F,\p} := (\xi_\val^{F_\p})^{-1} \circ \partial_{\val}^{F,\p} \circ \xi_\val^F \colon \K_*(C^*(\Gamma_K^F)) \to \K_*(C^*(\Gamma_K^{F_\p})) .\]

\begin{lemma}[{cf.~\cite[Lemma 4.1]{KT}}]
\label{lem:calcboundary}
    Let $F$ be a finite subset of $\cP_K$ with $|F| \leq N-1$, and let $\p \in \cP_K \setminus F$. Fix $\pi_\p \in \Gamma_K^F$ satisfying $v_\p(\pi_\p) = 1$. Fix an orientation of $\pi_\p^\Zz$ by letting $\beta_{\pi_\p}=\bigl[ u_{\pi_\p} \bigr]_1$.  
    Let $\Lambda \subseteq \Gamma_K^{F_\p}$ be a finite rank summand, and fix an orientation of $\Lambda$. Then, 
    \begin{equation}
    \label{eqn:partiall_val}
    \ol{\partial}_\val^{F,\p}(\beta_\Lambda) = 0 \quad\text{and}\quad 
    \ol{\partial}_\val^{F,\p}(\beta_\Lambda \hat{\otimes} \beta_{\pi_{\p}})
    = - \beta_{\Lambda}. 
    \end{equation}
    Moreover, Equation~\eqref{eqn:partiall_val}, as $\Lambda$ runs through the finite rank summands of $\Gamma_K^{F_\p}$, characterises $\ol{\partial}_\val^{F,\p}$.
\end{lemma}
\begin{proof}
    Let $\cE \in \cExt(B_\val^F, \Kz \otimes B_\val^{F_\p})$ be the extension from Lemma~\ref{lem:E_val}. Then, by Lemma~\ref{lem:E_val}, we have 
    $\ol{\partial}_\val^{F,\p} = \partial_\cE = \id \otimes \partial_{\cT}$. 
    Hence, by Lemma~\ref{lem:boundarytoeplitz}, we have 
    \[ \ol{\partial}_\val^{F,\p}(\beta_\Lambda \hat{\otimes} \beta_{\pi_{\p}}) = \beta_\Lambda \ho{} \partial_\cT(\beta_{\pi_\p}) = -\beta_\Lambda. 
    \]
    Since $\beta_\Lambda$ is identified with $\beta_\Lambda \hat{\otimes} \bigl[ 1_{C^*(\pi_\p^\Zz)} \bigr]_0$,  we have, by Lemma~\ref{lem:boundarytoeplitz},
    \[ \ol{\partial}_\val^{F,\p}(\beta_\Lambda) = \ol{\partial}_\val^{F,\p}(\beta_\Lambda \ho{} \bigl[ 1_{C^*(\pi_\p^\Zz)} \bigr]_0) = \beta_\Lambda \ho{} \partial_\cT(\bigl[ 1_{C^*(\pi_\p^\Zz)} \bigr]_0) = 0. 
    \]

By the proof of Lemma~\ref{lem:GammaKdecomp}, the group $\K_*(C^*(\Gamma_K^F))$ is generated by elements of the form $\beta_\Lambda$ and $\beta_\Lambda \ho{} \beta_{\pi_\p}$ for finite rank oriented summands $\Lambda \subseteq \Gamma_K^{F_\p}$. Hence, the second claim follows.
\end{proof}

Fix a total order of $\cP_K$, and equip $\cP_L$ with the total order induced by $\theta$. For a finite set $F=\{\p_1,\dots,\p_l\}$ with $l \leq N$ and $\p_1 < \p_2 < \dots < \p_l$, let 
\[D_K^F := (\xi_K^F)^{-1} \circ \partial_K^{F_{l-1},\p_1} \circ \cdots \circ\partial_K^{F_1,\p_{l-1}} \circ \partial_K^{\emptyset, \p_l} \colon \K_*(C^*(\Gamma_K)) \to \K_{*+l}(B_K^F), \]
where $F_0 = \emptyset$ and $F_i=\{\p_{l-i+1},\dots,\p_l\}$ for $1 \leq i \leq l$. 
We define $D_\val^F \colon \K_*(C^*(\Gamma_K)) \to \K_*(C^*(\Gamma_K^F))$ similarly as follows:
\[D_\val^F := (\xi_\val^F)^{-1} \circ \partial_\val^{F_{l-1},\p_1} \circ \cdots \circ\partial_\val^{F_1,\p_{l-1}} \circ \partial_\val^{\emptyset, \p_l}
= \ol{\partial}_\val^{F_{l-1},\p_1} \circ \cdots \circ \ol{\partial}_\val^{F_1,\p_{l-1}} \circ \ol{\partial}_\val^{\emptyset, \p_l}.\]
Note that $C^*(\Gamma_K)$, $B_K^\emptyset$, and $B_\val^\emptyset$ are all identified. 

Let $F$ be a finite subset of $\cP_K$ with $|F| \leq N$, and let $\p \in \cP_K$. 
Then, the fact that $\fval_K^{F',\p} \colon \cE_\val^{F',\p} \to \cE_K^{F',\p}$ is a homomorphism of extensions for every $|F'| \leq |F|$ implies that the following diagram commutes:
\begin{equation} \label{eqn:diagramval} \begin{tikzcd}
    \K_*(C^*(\Gamma_K)) \arrow[rd, "D_K^F"']\arrow[r, "D_\val^F"] & \K_{*+l}(C^*(\Gamma_K^F))\arrow[d,"(\ol{\fval}_K^F)_*"]\\
     & \K_{*+l}(B_K^F)\nospacepunct{.} 
\end{tikzcd} \end{equation}
Here, we used the fact that $C^*(\Gamma_K)=B_K^\emptyset$ and $(\ol{\fval}_K^\emptyset)_*\colon \K_*(C^*(\Gamma_K))\to \K_*(B_K^\emptyset)$ is the identity map by Lemma~\ref{lem:inclusion}.
Note that $D_\val^F$ is surjective 
by Lemma~\ref{lem:calcboundary}.  
In addition, the diagram 
\begin{equation} \label{eqn:diagramKL} \begin{tikzcd}
    \K_*(C^*(\Gamma_K)) \arrow[d, "\ol{\alpha}^\emptyset_*"]\arrow[r, "D_K^F"] & \K_{*+l}(B_K^{F})\arrow[d,"\ol{\alpha}^{F}_*"]\\
	\K_*(C^*(\Gamma_L)) \arrow[r, "D_L^{\theta(F)}"] & \K_{*+l}(B_L^{\theta(F)}) 
\end{tikzcd} \end{equation}
commutes by the commutativity of \eqref{eqn:main2}, where $\ol{\alpha}^F_* := (\xi_L^{\theta(F)})^{-1} \circ \alpha^F_* \circ \xi_K^F$. Since $\xi_K^\emptyset$ and $\xi_L^\emptyset$ are the identity maps, $\ol{\alpha}^\emptyset_*$ comes from the *-isomorphism  $\alpha^\emptyset$. On the other hand, $\ol{\alpha}_*^F$ is an order isomorphism between \K-groups, since it is a composition of order isomorphisms between  \K-groups.

The next proposition is the key result in this section. It is the only place where  \ref{item:order} in \hyperref[def:CN]{$(C_N)$} is used.
\begin{proposition} \label{prop:propertygamma}
    For each $\p \in \cP_K$ and $x \in \Gamma_K$, we have $v_{\theta(\p)}(\gamma(x))=v_\p(x)$. Consequently, we have $\gamma(\cO_K^\times/\mu_K)=\cO_L^\times/\mu_L$ and $\gamma(\Gamma_K^F) = \Gamma_L^{\theta(F)}$ for every finite subset $F \subseteq \cP_K$. 
\end{proposition}
\begin{proof}
    First, we show that for any $\p \in \cP_K$ and $x \in \Gamma_K$, we have $D_K^{\{\p\}}([u_x]_1)=(-v_{\p}(x))[1]_0$, where $[1]_0 \in \K_0(B_K^{\{\p\}})$ is the $\K_0$-class of the unit of $B_K^{\{\p\}}$. 
    Fix $\p \in \cP_K$ and let $n=v_\p(x)$. By Lemma~\ref{lem:approx}, we can take $\pi_\p \in \Gamma_K$ with $v_\p(\pi_\p)=1$. Let $\beta_{\pi_\p}=[u_{\pi_\p}]_1$. Then, we have $x= \pi_\p^n y$ for some $y \in \Gamma_K^{\{\p\}}$, and hence $[u_x]_1 = n\beta_{\pi_{\p}}+[u_y]_1$.
    By commutativity of Diagram~\eqref{eqn:diagramval}, Lemma~\ref{lem:inclusion}, and Lemma~\ref{lem:calcboundary}, we have 
    \begin{equation}
\label{eqn:DpK}
    D_K^{\{\p\}}([u_x]_1)= (\ol{\fval}_K^{\{\p\}})_* \circ D_\val^{\{\p\}} ([u_x]_1)=(\ol{\fval}_K^{\{\p\}})_* \circ D_\val^{\{\p\}} (n\beta_{\pi_{\p}}+[u_y]_1)
    = (\ol{\fval}_K^{\{\p\}})_*(-n[1]_0)
    =-n[1]_0.
    \end{equation}
    Similarly, we have $D_L^{\{\theta(\p)\}}([u_{\gamma(x)}]_1)=(-v_{\theta(\p)}(\gamma(x)))[1]_0$.
    Hence, we have  
    \[(-v_{\theta(\p)}(\gamma(x)))[1]_0 = D_L^{\{\theta(\p)\}}([u_{\gamma(x)}]_1) = D_L^{\{\theta(\p)\}} \circ \ol{\alpha}^\emptyset_*([u_x]_1)
    = \ol{\alpha}^{\{\p\}}_* \circ D_K^{\{\p\}} ([u_x]_1)
    =(-v_{\p}(x))\ol{\alpha}^{\{\p\}}_*([1]_0)  \]
    by Lemma~\ref{lem:existgamma} and commutativity of Diagram~\eqref{eqn:diagramKL}. In particular, $v_{\p}(x)=0$ if and only if $v_{\theta(\p)}(\gamma(x)))=0$. 
    Since $B_L^{\{\theta(\p)\}}$ admits a faithful tracial state, $[1]_0$ generates a copy of $\Zz$ in $\K_0(B_L^{\{\theta(\p)\}})$. Hence,
    in order to show the claim, it suffices to show $\ol{\alpha}^{\{\p\}}_*([1]_0)=[1]_0$. Since $D_K^{\{\p\}}([u_x]_1)=(-v_\p(x))[1]_0$ for any $x \in \Gamma_K$, we have $D_K^{\{\p\}}(\Gamma_K) = \Zz [1]_0$. Here, $\Gamma_K$ is identified with $\cU(C^*(\Gamma_K))/\cU(C^*(\Gamma_K))_0 \subseteq \K_1(C^*(\Gamma_K))$. Hence, the element  $[1]_0 \in \K_0(B_K^{\{\p\}})$ is characterised by the property that it is the unique generator of $D_K^{\{\p\}}(\Gamma_K)$ which belongs to $\K_0(B_K^{\{\p\}})_+$. Since $\ol{\alpha}^{\{\p\}}_*$ is an order isomorphism, commutativity of Diagram~\eqref{eqn:diagramKL} implies $\ol{\alpha}^{\{\p\}}_*([1]_0)=[1]_0$.
\end{proof}

For each finite subset $F$ of $\cP_K$, using Proposition~\ref{prop:propertygamma}, we let $\gamma^F \colon C^*(\Gamma_K^F) \to C^*(\Gamma_L^{\theta(F)})$ be the isomorphism induced from $\gamma$. By definition, we have $\gamma^\emptyset = \alpha^\emptyset$. 

\begin{lemma}
\label{lem:DFsquare}
    For every finite subset $F\subseteq\cP_K$ with $|F| \leq N$, the following diagram commutes:
    \begin{equation} \label{eqn:diagramvalKL} \begin{tikzcd}
      \K_*(C^*(\Gamma_K)) \arrow[d, "\gamma^\emptyset_*"]\arrow[r, "D_\val^F"] & \K_{*+l}(C^*(\Gamma_K^{F}))\arrow[d,"\gamma^F_*"]\\
	  \K_*(C^*(\Gamma_L)) \arrow[r, "D_\val^{\theta(F)}"] & \K_{*+l}(C^*(\Gamma_L^{\theta(F)}))\nospacepunct{.} 
    \end{tikzcd} \end{equation}
\end{lemma}
\begin{proof}
    Let $F=\{\p_1,\dots,\p_l\}$ with $\p_1 < \dots < \p_l$. 
    Using Lemma~\ref{lem:approx}, we fix a family $\{\pi_{\p_j}\}_{j=1}^l \subseteq \Gamma_K$ such that $v_{\p_j}(\pi_{\p_k})=\delta_{j,k}$ for every $j,k=1,\dots,l$. Choose an orientation of $\pi_{\p_i}^\Zz$ and $\gamma(\pi_{\p_i})^\Zz$ by $\beta_{\pi_{\p_i}}=[u_{\pi_{\p_i}}]_1$ and $\beta_{\gamma(\pi_{\p_i})} = [u_{\gamma(\pi_{\p_i})}]_1$ for $i=1,\dots,l$, respectively. For each finite rank summand $\Lambda \subseteq \Gamma_K^F$, we fix an orientation of $\Lambda$, and choose an orientation of $\gamma(\Lambda) \subseteq \Gamma_L^{\theta(F)}$ by $\beta_{\gamma(\Lambda)}=\gamma_*^F(\beta_{\Lambda})$.  
    By Lemma~\ref{lem:GammaKdecomp}, the \K-group $\K_*(C^*(\Gamma_K))$ is generated by elements of the form $x=\beta_\Lambda \hat{\otimes} \beta_{\pi_{\p_{j_1}}} \hat{\otimes} \cdots \hat{\otimes} \beta_{\pi_{\p_{j_k}}}$, where $\Lambda \subseteq \Gamma_K^F$ is a finite rank summand, and $1 \leq j_1 < \dots < j_k \leq l$ is an increasing sequence of natural numbers. Hence, it suffices to show that 
    \begin{equation} \label{eqn:proofcommutative}
        \gamma^F_* \circ D_\val^F (x) = D_\val^{\theta(F)} \circ \gamma^\emptyset_*(x) 
    \end{equation}
    for such $x$.  Fix such a summand and increasing sequence. Then, Lemma~\ref{lem:calcboundary} implies that $D_\val^F(x)=0$ unless $k=l$ (which is equivalent to the condition that all primes in $F$ appear in $x$). We have 
    \[ \gamma_*^\emptyset(x) = \beta_{\gamma(\Lambda)} \ho{} \beta_{\gamma(\pi_{\p_{j_1}})} \ho{} \cdots \ho{} \beta_{\gamma(\pi_{\p_{j_k}})}. \]
    By Proposition~\ref{prop:propertygamma}, we have $v_{\theta(\p_j)}(\gamma(\pi_{\p_k}))=\delta_{j,k}$, so that $D_\val^{\theta(F)}(\gamma_*(x))=0$ unless $k = l$ by Lemma~\ref{lem:calcboundary}. Hence, Equation~\eqref{eqn:proofcommutative} holds when $k \neq l$. Suppose $k=l$. Then, $D_\val^F(x) = (-1)^l \beta_\Lambda$, and $D_\val^{\theta(F)}(\gamma_*^\emptyset(x))=(-1)^l \beta_{\gamma(\Lambda)}$. Hence, Equation ~\eqref{eqn:proofcommutative} holds when $k=l$. 
\end{proof}

\subsection{Reconstruction at semi-local levels}
\label{sec:reconsemi-local}
Based on the arguments in the last subsection, we now give a reconstruction result for the dynamical system $\Gamma_K \acts \Az_{K,f}/\mu_K$ at semi-local levels. 

\begin{lemma}  \label{lem:conjugacy1}
Let $K$ and $L$ be number fields, and let $N \in \Zz$ with $N \geq 1$. Assume that condition \hyperref[def:CN]{$(C_N)$} is satisfied. Then, for every finite subset $F\subseteq\cP_K$ with $|F| \leq N$, there exists an isomorphism of compact groups 
\[ \tilde{\eta}^F\colon \left(\prod_{\p \in F} \cO_{K,\p}^*\right)/\mu_K\to \left(\prod_{\q \in \theta(F)} \cO_{L,\q}^*\right)/\mu_L\] 
such that the restriction of $\tilde{\eta}^F$ to $\Gamma_K^F$ coincides with $\gamma^F$.
\end{lemma}
\begin{proof}
We use a similar diagram chase as in \cite[Proof of Theorem 1.1]{KT}. By commutativity of the diagrams \eqref{eqn:diagramval} \eqref{eqn:diagramKL} \eqref{eqn:diagramvalKL}, the small squares on the left and right and the triangles on the top and bottom in the following diagram commute: 
\begin{equation} \label{eqn:bigdiagram} \begin{tikzcd}
  \K_{*+l}(C^*(\Gamma_K^F)) \arrow[ddd, "\gamma^F_*"] \arrow[rr, "(\ol{\fval}_K^F)_*"] & & \K_{*+l}(B_K^F) \arrow[ddd, "\ol{\alpha}^F_*"]\\ 
  & \K_*(C^*(\Gamma_K)) \arrow[lu, "D_\val^F"'] \arrow[ru, "D_K^F"] \arrow[d, "\gamma_*^\emptyset"]& \\
  & \K_{*}(C^*(\Gamma_L)) \arrow[ld, "D_\val^{\theta(F)}"'] \arrow[rd, "D_L^{\theta(F)}"]& \\
  \K_{*+l}(C^*(\Gamma_L^{\theta(F)})) \arrow[rr, "(\ol{\fval}_L^{\theta(F)})_*"] & & \K_{*+l}(B_L^{\theta(F)})\nospacepunct{.}
\end{tikzcd} \end{equation}

Hence, by the surjectivity of $D_\val^F$, the large outer square in Diagram~\eqref{eqn:bigdiagram}
commutes. 
By Lemma~\ref{lem:localstuff}, $\cO_{K,\p}^*\cong \Zz/(N(\p)-1)\Zz\times \cO_{K,\p}^{(1)}$, and $\cO_{K,\p}^{(1)}$ has finite $\Zz_p$-rank, where $p$ is the rational prime lying under $\p$.
Thus, $\left(\prod_F \cO_{K,\p}^*\right)/\mu_K$ and $\left(\prod_{\theta(F)} \cO_{L,\q}^*\right)/\mu_L$ finitely generated pro-$\cN$ completions (as defined in \cite[\S~3]{KT}) of $\Gamma_K^F$ and $\Gamma_L^{\theta(F)}$, respectively, for an appropriate set of rational primes $\cN$. Now the result follows by \cite[Corollary~3.18]{KT}, since $(\ol{\fval}_K^F)_*$ and  $(\ol{\fval}_K^{\theta(F)})_*$ coincide with the maps induced by the inclusion maps by Lemma~\ref{lem:inclusion}.
\end{proof}

For each finite subset $F\subseteq\cP_K$, let $V^F \colon \Zz^F \to \Zz^{\theta(F)}$ be the isomorphism induced by $\theta \colon \cP_K \to \cP_L$.

\begin{proposition}
\label{prop:semi-localdata}
    Let $K$ and $L$ be number fields, and let $N \in \Zz$ with $N \geq 1$. Assume that condition \hyperref[def:CN]{$(C_N)$} is satisfied. Then, for every finite subset $F\subseteq\cP_K$ with $|F| \leq N$, 
    there exists a unique isomorphism 
    \[ \eta^F \colon \left(\prod_{\p \in F} K_\p^* \right)/\mu_K \to \left(\prod_{\q \in \theta(F)} L_\q^*\right)/\mu_L \] 
    which satisfies $\eta^F|_{\Gamma_K}=\gamma$ and makes the following diagram commute: 
\begin{equation} \label{eqn:diagramextension} \begin{tikzcd}
    0 \arrow[r] & (\textstyle{\prod}_F \cO_{K,\p}^*)/\mu_K \arrow[r] \arrow[d, "\tilde{\eta}^F"] & (\textstyle{\prod}_F K_\p^*)/\mu_K \arrow[r, "\prod_F v_\p"] \arrow[d, "\eta^F"] & \Zz^F \arrow[r] \arrow[d, "V^F"] & 0\\
    0 \arrow[r] & (\textstyle{\prod}_{\theta(F)} \cO_{L,\q}^*)/\mu_L \arrow[r] & (\textstyle{\prod}_{\theta(F)} L_\q^*)/\mu_L \arrow[r, "\prod_{\theta(F)} v_\q"] & \Zz^{\theta(F)} \arrow[r] & 0\nospacepunct{.}
\end{tikzcd} \end{equation}
\end{proposition}
\begin{proof}
    Uniqueness follows from the fact that $\Gamma_K$ is dense in $(\prod_F K_\p^*)/\mu_K$ by Lemma~\ref{lem:approx}. We now prove existence. 
    Let $F=\{\p_1,\dots,\p_l\}$, and by Lemma~\ref{lem:approx}, fix a family $\{\pi_{\p_j}\}_{j=1}^l \subseteq K^*$ such that $v_{\p_j}(\pi_{\p_k})=\delta_{j,k}$ for all $j,k=1,\dots,l$. Let $\Lambda = \prod_{j=1}^l \pi_{\p_j}^\Zz$. 
    Let $\tilde{\eta}^F \colon (\prod_F \cO_{K,\p}^*)/\mu_K \to (\prod_{\theta(F)} \cO_{L,\q}^*)/\mu_L$ be the isomorphism of compact groups from Lemma~\ref{lem:conjugacy1}. By Lemma~\ref{lem:GammaKdecomp}, we have $\Gamma_K = \Lambda \times \Gamma_K^F$. Thus, by Proposition~\ref{prop:propertygamma}, we have $\Gamma_L=\gamma(\Lambda) \times \Gamma_L^{\theta(F)}$. Moreover, the surjective group homomorphism 
    \[
    (\textstyle{\prod}_F K_\p^*)/\mu_K\to \Lambda,\quad (x_{\p_j})_{\p_j}\mapsto (\pi_{\p_j}^{v_{\p_j}(x_{\p_j})})_j
    \]
    splits and has kernel $(\prod_F \cO_{K,\p}^*)/\mu_K$, so that $(\prod_F K_\p^*)/\mu_K = \Lambda \times (\prod_F \cO_{K,\p}^*)/\mu_K$. Similarly, $(\prod_{\theta(F)} L_\p^*)/\mu_L\\=\gamma(\Lambda) \times (\prod_{\theta(F)} \cO_{L,\q}^*)/\mu_L$.
    Since $\tilde{\eta}^F|_{\Gamma_K^F}=\gamma^F$, the isomorphism $\eta^F := \gamma|_\Lambda \times \tilde{\eta}^F$ has the desired property. 
\end{proof}
 
\subsection{Reconstruction of the dynamical system} \label{ssec:recondyn}
Throughout this subsection, we let $K$ and $L$ be number fields, and we assume that condition \hyperref[def:CN]{$(C_\infty)$} is satisfied. Following the strategy from \cite[\S~2]{KT}, we give a direct reconstruction of the dynamical system $\Gamma_K \acts \Az_{K,f}/\mu_K$. Results in this subsection are not used in the proof of Theorem~\ref{thm:main2}.

For each nonempty finite subset $F \subseteq \cP_K$, we let $X_K^F:=(\prod_F K_\p^*)/\mu_K$, $\ol{X_K^F}:=(\prod_F K_\p)/\mu_K$, and $\ol{Y_K^F}:=(\prod_F \cO_{K,\p})/\mu_K$. Put $X_K^\emptyset = \ol{X_K^\emptyset} = \ol{Y_K^\emptyset} = \{0\}$.
Then, for a finite subset $F \subseteq \cP_K$, we have $\ol{X_K^F} = \bigsqcup_{E \subseteq F} X_K^E$, under the identification 
\[ X_K^E = \left( \prod_{\p \in F \setminus E} \{0\} \times \prod_{\p \in F} K_\p^* \right)/\mu_K \subseteq \ol{X_K^F} . \]

For each finite set $F \subseteq \cP_K$, let $\eta^F \colon X_K^F \to X_L^{\theta(F)}$ be the valuation-preserving isomorphism from Proposition~\ref{prop:semi-localdata}.
We define $\ol{\eta^F} \colon \ol{X_K^F} \to \ol{X_L^{\theta(F)}}$ by $\ol{\eta^F} := \bigsqcup_{E\subseteq F}\eta^E$. 
For all finite subsets $F,F' \subseteq \cP_K$ with $F \subseteq F'$, the diagram 
\begin{equation} \label{eqn:commutativeol}  \begin{tikzcd}
    \ol{X_K^{F'}} \arrow[r, "\ol{\eta^{F'}}"] \arrow[d] & \ol{X_L^{\theta(F')}} \arrow[d] \\
    \ol{X_K^{F}} \arrow[r, "\ol{\eta^F}"] & \ol{X_L^{\theta(F)}}
\end{tikzcd} \end{equation}
commutes by definition, where the vertical maps are canonical projections, and we have
\begin{equation} \label{eqn:commvalol}
(\textstyle{\prod}_F v_{\theta(\p)}) \circ \ol{\eta^F} = \textstyle{\prod}_F v_\p. 
\end{equation}

Commutativity of Diagram \eqref{eqn:commutativeol} and Equation~\eqref{eqn:commvalol} essentially characterise the topology of $\ol{X_K^F}$:
\begin{lemma} \label{lem:homeo}
    For every finite subset $F \subseteq \cP_K$, the map $\ol{\eta^F}$ is a homeomorphism, and the dynamical systems $\Gamma_K \acts \ol{X_K^F}$ and $\Gamma_L \acts \ol{X_L^{\theta(F)}}$ are conjugate via $(\ol{\eta^F}, \gamma)$.
\end{lemma}
\begin{proof}
    We first show continuity of $\ol{\eta^F}$ by induction on $|F|$. The claim is obvious if $|F|=0$. Let $x \in \ol{X_K^F}$, and let $\{x^{(n)}\}_{n=1}^\infty$ be a sequence in $\ol{X_K^F}$ converging to $x$. We first suppose $x \in X_K^F$. Then, $x^{(n)}$ eventually belongs to $X_K^F$ since $X_K^F$ is open in $\ol{X_K^F}$. Hence, $\ol{\eta^F}(x^{(n)})=\eta^F(x^{(n)})$ converges to $\ol{\eta^F}(x)=\eta^F(x)$ by continuity of $\eta^F$. 
    Next, we suppose $x \in X_K^E$ for some $E \subsetneq F$. Then, for $\p \in F \setminus E$, the sequence $\ol{\eta^F}(x^{(n)})_{\theta(\p)}$ converges to $0$ in $L_{\theta(\p)}$ since $v_{\theta(\p)}(\ol{\eta^F}(x^{(n)})_{\theta(\p)})=v_\p(x^{(n)}_\p) \to \infty$ as $n \to \infty$ by Equation~\eqref{eqn:commvalol}. Furthermore, the induction hypothesis and commutativity of the Diagram~\eqref{eqn:commutativeol} for $E$ and $F$ imply that $\ol{\eta^F}(x^{(n)})_{\theta(\p)}$ converges to $\ol{\eta^F}(x)_{\theta(\p)} = \eta^E(x)_{\theta(\p)}$ in $L_{\theta(\p)}$ for $\p \in E$. Therefore, $\ol{\eta^F}(x^{(n)})$ converges to $\ol{\eta^F}(x)$. 
    
    We see that the inverse of $\ol{\eta^F}$ is continuous by the same argument. Hence, $\ol{\eta^F}$ is a homeomorphism. 
    The second assertion follows from the fact that $\Gamma_K \acts X_K^F$ and $\Gamma_L \acts X_L^{\theta(F)}$ are conjugate via $(\eta^F,\gamma)$. 
\end{proof}

\begin{lemma} \label{lem:YKYL}
    For every finite subset $F \subseteq \cP_K$, we have $\ol{\eta^F}(\ol{Y_K^F}) = \ol{Y_L^{\theta(F)}}$. In particular, semigroup dynamical systems $\cO_K^\times/\mu_K \acts \ol{Y_K^F}$ and $\cO_L^\times/\mu_L \acts \ol{Y_L^{\theta(F)}}$ are conjugate via $(\ol{\eta^F},\gamma)$. 
\end{lemma}
\begin{proof}
    By construction of $\ol{\eta^F}$, the restriction of $\ol{\eta^F}$ to $X_K^F$ is the group isomorphism $\eta^F \colon X_K^F \to X_L^{\theta(F)}$. In particular, $\ol{\eta^F}$ sends $(1,\dots,1)$ to $(1,\dots,1)$. Hence, we have 
    \[\ol{\eta^F}(\cO_K/\mu_K) = \ol{\eta^F}(\cO_K^\times/\mu_K(1,\dots,1) \sqcup \{0\})
    = \cO_L^\times/\mu_L(1,\dots,1) \sqcup \{0\} = \cO_L/\mu_L. \]
    Now the assertion follows because $\ol{\eta^F}$ is continuous by Lemma~\ref{lem:homeo}.
\end{proof}

\begin{theorem}
Let $K$ and $L$ be number fields, and suppose that condition \hyperref[def:CN]{$(C_\infty)$} is satisfied. Then,
    the dynamical systems $K^*/\mu_K \acts \Az_{K,f}/\mu_K$ and $L^*/\mu_L \acts \Az_{L,f}/\mu_L$ are conjugate. 
\end{theorem}
\begin{proof}
    By commutativity of Diagram~\eqref{eqn:commutativeol} and Lemma~\ref{lem:YKYL}, the diagram
    \begin{equation}  \begin{tikzcd}
    \ol{Y_K^{F'}} \arrow[r, "\ol{\eta^{F'}}"] \arrow[d] & \ol{Y_L^{\theta(F')}} \arrow[d] \\
    \ol{Y_K^{F}} \arrow[r, "\ol{\eta^F}"] & \ol{Y_L^{\theta(F)}}
    \end{tikzcd} \end{equation}
    commutes for all finite subsets $F,F' \subseteq \cP_K$ with $F \subseteq F'$. Note that all maps in this diagram are equivariant with respect to $\cO_K^\times/\mu$ and $\cO_L^\times/\mu$ via $\gamma$. Taking the projective limit, we obtain an equivariant homeomorphism $\ol{\eta} \colon \ol{\cO}_K/\mu_K \to \ol{\cO}_L/\mu_L$. 
    
    To see the assertion, it suffices to show that $\ol{\eta}$ (uniquely) extends to a homeomorphism $\eta' \colon \Az_{K,f}/\mu_K \to \Az_{L,f}/\mu_L$ such that  \begin{equation} \label{eqn:equiv}
        \eta'(ax)=\gamma(a)\eta'(x)
    \end{equation}
    for every $a \in \Gamma_K$ and $x \in \Az_{K,f}$. For each $z \in \Az_{K,f}/\mu_K$, there exist $a \in \cO_K^\times/\mu_K$ and $x \in \ol{\cO}_K/\mu_K$ such that $z=a^{-1}x$ by Lemma~\ref{lem:approx}. We define $\eta' \colon \Az_{K,f}/\mu_K \to \Az_{L,f}/\mu_L$ by $\eta'(z)=\gamma(a)^{-1}\ol{\eta}(x)$ for such $z=a^{-1}x \in \Az_{K,f}/\mu_K$. Now it is straightforward to see that $\eta'$ is well-defined, homeomorphic, and satisfies Equation~\eqref{eqn:equiv}. 
\end{proof}

We can also reconstruct the dynamical system $K^* \acts \Az_{K,f}$ from the semi-local data, see Remark~\ref{rmk:recon}.

\section{Reconstruction of a number field from semi-local data}\label{sec:NT}

\subsection{Hoshi's theorem}
\label{sec:Hoshi}
Given a number field $K$ and $\p\in\cP_K$, we follow the notation in \cite{CdSLMS}, and let $\O_{K,[\p]}$ denote the localisation of $\O_K$ at $\p$. We shall need the following result by Hoshi.

\begin{theorem}[{\cite[Corollary~3.3]{Hoshi}}]\label{thm:Hoshi}
   Let $K$ and $L$ be number fields, and suppose $\kappa\colon K^*\to L^*$ is a surjective group homomorphism. Then, $\kappa$ is the restriction of an isomorphism of fields $K\cong L$ if and only if there exists a map $\theta\colon\cP_K\to\cP_L$ such that
   \begin{enumerate}[\upshape(1)]
       \item there exists $\p\in\cP_K$ and $n\in\Zz_{>0}$ with $n\cdot v_\p=v_{\theta(\p)}\circ\kappa$;
       \item we have $1+\p\O_{K,[\p]}=\kappa^{-1}(1+\theta(\p)\O_{L,[\theta(\p)]})$ for all but finitely many $\p\in\cP_K$.
   \end{enumerate}
\end{theorem}

We now collect two results around Hoshi's theorem which are essentially contained in \cite{CdSLMS}. The following observation is used in \cite{CdSLMS} without proof.

\begin{lemma}
\label{lem:one-units}
For all but finitely many $\p\in\cP_K$, we have $1+\p\O_{K,\p}=(\O_{K,\p}^*)^{N(\p)-1}$.
\end{lemma}
\begin{proof}
By Lemma~\ref{lem:localstuff}, $\O_{K,\p}^*=\Zz/(N(\p)-1)\Zz\times \O_{K,\p}^{(1)}$ and $\O_{K,\p}^{(1)}\cong \Zz_p^{[K_\p:\Qz_p]}$ for all but finitely many $\p$, where $p$ is the rational prime lying under $\p$.
Fix such a prime $\p$, and let $f \colon \O_{K,\p}^*\to \O_{K,\p}^*$ be the map defined by $x \mapsto x^{N(\p)-1}$. Then, the image of $f$ is clearly contained in $\O_{K,\p}^{(1)}$. Since $N(\p)-1$ is a unit in $\Zz_p$, we see that the restriction of $f$ defines an automorphism of $\O_{K,\p}^{(1)}$. Hence, $\Im f = \O_{K,\p}^{(1)}$. 
\end{proof}

We state a consequence of Hoshi's theorem that can be extracted from the proof of \cite[Theorem~7.4]{CdSLMS}:

\begin{proposition}
\label{prop:CdSLMS}
Let $K$ and $L$ be number fields. Suppose we have a bijection $\theta \colon \cP_K \to \cP_L$, a group isomorphism $\kappa\colon K^* \to L^*$, and a family of topological group isomorphisms $\varphi_\p \colon K_\p^* \to L_{\theta(\p)}^*$ for all but finitely many $\p \in \cP_K$. If $v_{\theta(\p)}\circ\varphi_\p=v_\p$ and $\varphi_\p|_{K^*}=\kappa$ for all but finitely many $\p$, then $\kappa$ is the restriction of a field isomorphism.
\end{proposition}
\begin{proof}
Let $F\subseteq\cP_K$ be a finite subset such that $v_{\theta(\p)}\circ\varphi_\p=v_\p$ and $\varphi_\p|_{K^*}=\kappa$ for all $\p\in\cP_K\setminus F$.
Clearly, condition (1) from Theorem~\ref{thm:Hoshi} is satisfied. Fix $\p\in\cP_K\setminus F$. We see that $\varphi_\p(\O_{K,\p}^*)=\O_{L,\theta(\p)}^*$. By Lemma~\ref{lem:localstuff}, we have, after possibly enlarging $F$, that $\tors(\O_{K,\p}^*)=\Zz/(N(\p)-1)\Zz$. Thus, $N(\p)=N(\theta(\p))$. 
By Lemma~\ref{lem:one-units}, we have
\[
\kappa(1+\p\O_{K,[\p]})=\kappa(K^*\cap (1+\p\O_{K,\p}))=\varphi_\p(K^*)\cap \varphi_\p((\O_{K,\p}^*)^{N(\p)-1})=L^*\cap (\O_{L,\theta(\p)}^*)^{N(\theta(\p))-1}=1+\theta(\p)\O_{L,[\theta(\p)]}.
\]
Thus, $\kappa$ also satisfies condition (2) from Theorem~\ref{thm:Hoshi}, so $\kappa$ is the restriction of a field isomorphism $K\cong L$.
\end{proof}

\subsection{A semi-local characterisation of number fields}
\label{sec:semi-localchar}
Using Hoshi's theorem, we now give a characterisation of number fields in terms of semi-local dynamical systems modulo roots of unity.

\begin{proposition}
\label{prop:nfchar}
Let $K$ and $L$ be number fields with $|\mu_K|=|\mu_L|$. Suppose we have a bijection $\theta\colon\cP_K\to \cP_L$, a group isomorphism $\gamma\colon\Gamma_K\xrightarrow{\cong}\Gamma_L$, and a family of topological group isomorphisms
\begin{equation}
    \label{eqn:etaF}
    \eta^F\colon \left(\textstyle{\prod}_{\p\in F} K_\p^*\right)/\mu_K\xrightarrow{\cong}\left(\textstyle{\prod}_{\p\in F} L_{\theta(\p)}^*\right)/\mu_L,
\end{equation}
where $F\subseteq\cP_K$ with $1\leq |F|\leq 3$, satisfying
\begin{enumerate}[\upshape(a)]
    \item $\eta^F\vert_{\Gamma_K}=\gamma$;
    \item $\left(\prod_{\p\in F}v_{\theta(\p)}\right)\circ \eta^F=\prod_{\p\in F}v_\p$.
\end{enumerate}
Then, there exists a unique field isomorphism $\sigma\colon K\xrightarrow{\cong} L$ such that $\sigma\vert_{K^*}\colon K^* \xrightarrow{\cong} L^*$ is a lift of $\gamma$.
\end{proposition}

Note that condition (b) for $|F|=2$ and $|F|=3$ follows from condition (a) and condition (b) for $|F|=1$. 

For the remainder of this section, we use the notation and assumptions from the statement of Proposition~\ref{prop:nfchar}. We also identify $\mu_K$ and $\mu_L$, and put $\mu:=\mu_K=\mu_L$. We first prove the uniqueness claim.
\begin{proposition}
There is at most one field isomorphism $\sigma\colon K\xrightarrow{\cong}L$ such that $\sigma\vert_{K^*}\colon K^*\xrightarrow{\cong} L^*$ is a lift of $\gamma$.
\end{proposition}
\begin{proof}
Suppose $\sigma,\tau\colon K\to L$ are field isomorphisms such that $\sigma\vert_{K^*}$ and $\tau\vert_{K^*}$ are lifts of $\gamma$. Put $\rho:=\sigma^{-1}\circ\tau\colon K\to K$. Since $\gamma$ is valuation-preserving by condition (b) and all valuations vanish on $\mu$, it follows that $\sigma(\p)=\tau(\p)=\theta(\p)$ for all $\p\in\cP_K$. Thus, $\rho(\p)=\p$ for all $\p\in\cP_K$. 
Let $N$ be the Galois closure of $K$, and let $\ol{\rho} \in \Gal(N/\Qz)$ be any extension of $\rho$ (such an extension exists by \cite[Lemma, p.227]{Jac}). Let $p$ be a rational prime which splits completely in $N$ (such a prime exists by \cite[Corollary~VII.13.6]{Neu}). Let $\q \in \cP_N$ be a prime above $p$, $\q_1 = \ol{\rho}(\q)$, and $\p = \q \cap K$. Since $\rho(\p)=\p$, we have $\q_1 \cap K=\p$. By \cite[Proposition~II.9.1]{Neu}, there exists $\delta \in \Gal(N/K)$ such that $\delta(\q_1)=\q$. Let $\rho_1=\delta \circ \ol{\rho}$. Then, we have $\rho_1(K)=K$, $\rho_1|_K=\rho$, and $\rho_1 \in D_\q$, where $D_\q$ is the decomposition group of $\q$. Since the inertia degree of $\q$ over $p$ is equal to $1$, $D_\q$ is the trivial group, which implies that $\rho_1=\id_N$. Hence, $\rho=\id_K$. 
\end{proof}

For all but finitely many $\p\in\cP_K$, we will prove that $\eta^{\{\p\}} \colon K^*_\p/\mu \to L_{\theta(\p)}^*/\mu$ lifts to a topological group isomorphisms $\varphi_\p\colon K_\p^*\to L_{\theta(\p)}^*$ satisfying the conditions in Proposition~\ref{prop:CdSLMS}.

\begin{lemma}\label{lem:FF'}
For all $F',F\subseteq \cP_K$ with $F'\subseteq F$ and $1\leq |F'|,|F|\leq 3$, the diagram
\begin{equation}
\label{eqn:etaFc}\tag{c}
\begin{tikzcd}
\left(\textstyle{\prod}_F K_\p^*\right)/\mu \arrow[d]\arrow[r,"\eta^F"] &\left(\textstyle{\prod}_F L_{\theta(\p)}^*\right)/\mu\arrow[d]\\
 \left(\textstyle{\prod}_{F'}K_\p^*\right)/\mu  \arrow[r,"\eta^{F'}"] &
 \left(\textstyle{\prod}_{F'} L_{\theta(\p)}^*\right)/\mu
\end{tikzcd}
\end{equation}
commutes, where the vertical maps are the canonical projections.
\end{lemma}
\begin{proof}
This follows from condition (a) and the fact that $\Gamma_K$ is dense in $\left(\prod_F K_\p^*\right)/\mu$ by Lemma~\ref{lem:approx}.
\end{proof}

Fix a prime $\p \in \cP_K$, and for each $\q\in\cP_K$ with $\q\neq \p$, let $\varphi_\q\colon K_\q^*\to L_{\theta(\q)}^*$ be the map characterised by 
\[
\eta^{\{\p,\q\}}((1,x)\mu)=(1,\varphi_\q(x))\mu \in (L_{\theta(\p)}^* \times L_{\theta(\q)}^*)/\mu
\]for every $x \in K_\q^*$.

\begin{lemma}
\label{lem:varphip}
For all $\q\in\cP_K\setminus\{\p\}$, the map $\varphi_\q$ is a topological group isomorphism. Moreover, $\varphi_\q$ restricts to a group isomorphism $\varphi_\q\vert_{K^*}\colon K^*\xrightarrow{\cong} L^*$ that is a lift of $\gamma$. 
\end{lemma}
\begin{proof}
It is straightforward to see that $\varphi_\q$ is a group homomorphism.
To show continuity of $\varphi_\q$, it suffices to show that 
\[ \iota \colon K_\q^* \to (K_\p^* \times K_\q^*)/\mu,\ x \mapsto (1,x)\mu \]
is a homeomorphism onto its range. Continuity of $\iota$ is clear, so it suffices to show that $K_\q^* \to \iota(K_\q^*)$ is an open map. Let $U$ be an open subset of $K_\q^*$. Since $\mu$ is discrete in $K_\p^*$, we can take an open neighbourhood $V \subseteq K_\p^*$ of $1$ such that $V \cap \mu = \{1\}$. Then, we see that $\iota(U) = ((V \times U)/\mu) \cap \iota(K_\q^*)$. 
Hence, the claim holds since $(V \times U)/\mu$ is the image of an open set by the open surjection $(K_\p^* \times K_\q^*) \to (K_\p^* \times K_\q^*)/\mu$. 

By Lemma~\ref{lem:FF'}, the diagram 
    \begin{equation}  
    \begin{tikzcd}
    (K_{\p}^* \times K_{\q}^*)/\mu \arrow[r, "\eta^{\{\p, \q\}}"] \arrow[d] & (L_{\theta(\p)}^* \times L_{\theta(\q)}^*)/\mu \arrow[d] \\
    K_{\q}^*/\mu \arrow[r, "\eta^{\{\q\}}"] & L_{\theta(\q)}^*/\mu
    \end{tikzcd} 
    \end{equation}
commutes. By condition (a), for every $x \in K^*$, we have $\varphi_\q(x)\mu=\gamma(x\mu) \in \Gamma_L$, so that $\varphi_\q(x) \in L^*$. Hence, $\varphi_\q$ restricts to a group homomorphism $K^* \to L^*$ that is a lift of $\gamma$. 
Similarly, we have a continuous group homomorphism $\psi_\q \colon L^*_{\theta(\q)} \to K^*_\q$ which is characterised by $(\eta^{\{\p, \q\}})^{-1}((1,y)\mu)=(1,\psi_\q(y))\mu$ for $y \in L^*_{\theta(\q)}$. Then, we can see that $\psi_\q$ is the inverse of $\varphi_\q$, and $\psi_\q|_{L^*}$ is the inverse of $\varphi_\q|_{K^*}$.
\end{proof}

\begin{lemma}
\label{lem:varphip2}
For all $\q_1,\q_2\in\cP_K\setminus\{\p\}$, we have $\varphi_{\q_1}\vert_{K^*}=\varphi_{\q_2}\vert_{K^*}$, that is, $\varphi_\q\vert_{K^*}$ does not depend on $\q$.
\end{lemma}
\begin{proof}
Let $\q_1,\q_2 \in \cP_K \setminus \{\p\}$. We show $\varphi_{\q_1}(x)=\varphi_{\q_2}(x)$ for $x \in K^*$. Let $F=\{\p,\q_1,\q_2\}$, $x \in K^*$ and let 
\[(y,z,w)\mu = \eta^F ((1,x,x)\mu) \in (L_{\theta(\p)}^* \times L_{\theta(\q_1)}^* \times L_{\theta(\q_2)}^*)/\mu.\] 
By Lemma~\ref{lem:FF'}, the diagram
    \begin{equation}  \begin{tikzcd}
    (K_{\p}^* \times K_{\q_1}^* \times K_{\q_2}^*)/\mu \arrow[r, "\eta^F"] \arrow[d] & (L_{\theta(\p)}^* \times L_{\theta(\q_1)}^* \times L_{\theta(\q_2)}^*)/\mu \arrow[d] \\
    K_{\p}^*/\mu \arrow[r, "\eta^{\{\p\}}"] & L_{\theta(\p)}^*/\mu
    \end{tikzcd} \end{equation}
commutes, so that $y \in \mu$. By replacing the representative $(y,z,w)$ appropriately, we may assume that $y=1$. Note that $z$ and $w$ are then uniquely determined. By Lemma~\ref{lem:FF'}, the diagram
    \begin{equation}  \begin{tikzcd}
    (K_{\p}^* \times K_{\q_1}^* \times K_{\q_2}^*)/\mu \arrow[r, "\eta^F"] \arrow[d] & (L_{\theta(\p)}^* \times L_{\theta(\q_1)}^* \times L_{\theta(\q_2)}^*)/\mu \arrow[d] \\
    (K_{\p}^* \times K_{\q_i}^*)/\mu \arrow[r, "\eta^{\{\p,\q_i\}}"] & (L_{\theta(\p)}^*\times L_{\theta(\q_i)}^*)/\mu
    \end{tikzcd} \end{equation}
commutes for $i=1,2$, so that $z=\varphi_{\q_1}(x)$ and $w=\varphi_{\q_2}(x)$. Consequently, we have 
\[ \eta^F((1,x,x)\mu)=(1,\varphi_{\q_1}(x),\varphi_{\q_2}(x))\mu. \]
Finally, by Lemma~\ref{lem:FF'}, the diagram
    \begin{equation}  \begin{tikzcd}
    (K_{\p}^* \times K_{\q_1}^* \times K_{\q_2}^*)/\mu \arrow[r, "\eta^F"] \arrow[d] & (L_{\theta(\p)}^* \times L_{\theta(\q_1)}^* \times L_{\theta(\q_2)}^*)/\mu \arrow[d] \\
    (K_{\q_1}^* \times K_{\q_2}^*)/\mu \arrow[r, "\eta^{\{\q_1,\q_2\}}"] & (L_{\theta(\q_1)}^* \times L_{\theta(\q_2)}^*)/\mu
    \end{tikzcd} \end{equation}
commutes, so that $(\varphi_{\q_1}(x),\varphi_{\q_2}(x))\mu \in \Gamma_L$. Thus, $(\varphi_{\q_1}(x),\varphi_{\q_2}(x)) =\zeta(y,y)$ for some $\zeta \in \mu$ and $y \in L^*$, so that $\varphi_{\q_1}(x)=\varphi_{\q_2}(x)$.
\end{proof}

\begin{proof}[Proof of Proposition~\ref{prop:nfchar}]
We have shown that for every $\q \in \cP_K \setminus \{\p\}$, there is a topological group isomorphism $\varphi_\q\colon K_\q^*\to L_{\theta(\q)}^*$ such that $\kappa:=\varphi_\q|_{K^*}\colon K^*\to L^*$ is a group isomorphism that does not depend on $\q$ and is a lift of $\gamma$.
By condition (b), we have $(v_{\theta(\p)}\times v_{\theta(\q)})\circ \eta^{\{\p,\q\}}=v_{\p}\times v_\q$. Thus, for $x\in K_\q^*$, we have 
\[
(v_{\theta(\p)}\times v_{\theta(\q)})((1,\varphi_\q(x))\mu)=(v_{\theta(\p)}\times v_{\theta(\q)})\circ \eta^{\{\p,\q\}}((1,x)\mu)=v_{\p}\times v_\q((1,x)\mu),
\]
so that $v_{\theta(\q)}\circ \varphi_\q(x)=v_\q(x)$. The result now follows from Proposition~\ref{prop:CdSLMS}.
\end{proof}

\begin{remark} \label{rmk:recon}
An argument similar to Lemma~\ref{lem:varphip2} shows that $\varphi_\q$ does not depend on the choice of $\p \in \cP_K$. Therefore, by considering another fixed prime, we obtain  valuation-preserving isomorphisms $\varphi_\q \colon K^*_\q \to L_{\theta(\q)}^*$ for all $\q \in \cP_K$ that restrict to the same group isomorphism $\gamma \colon K^* \to L^*$. Hence, by following the argument in \S~\ref{ssec:recondyn}, we obtain a conjugacy of $K^* \acts \Az_{K,f}^*$ and $L^* \acts \Az_{L,f}^*$. 
\end{remark}

\subsection{Proofs of the main theorems}
\label{sec:proofs}
We now conclude the proofs of our main theorems.
\begin{proof}[Proof of Theorem~\ref{thm:main2}]
Let $K$ and $L$ be number fields that satisfy condition \hyperref[def:CN]{$(C_3)$}. Then, the assumptions in Proposition~\ref{prop:nfchar} are satisfied by Lemma~\ref{lem:existgamma} and Proposition~\ref{prop:semi-localdata}. Hence, the unique field isomorphism $\sigma \colon K \to L$ that satisfies the conditions in Theorem ~\ref{thm:main2} exists. 
\end{proof}

\begin{proof}[Proof of Theorem~\ref{thm:mainA}]
Let $K$ and $L$ be number fields, and suppose that full corners of $\fA_K$ and $\fA_L$ are isomorphic. Then, by Proposition~\ref{prop:commutativeKL}, condition \hyperref[def:CN]{$(C_\infty)$} is satisfied. Hence, by Theorem~\ref{thm:main2}, $K$ and $L$ are isomorphic. The converse is obvious. 
\end{proof}

\section{Application: Topological full groups}
\label{sec:fullgroups}
We next prove Corollary~\ref{cor:fullgroup}. Our terminology for groupoids follows \cite{NO,SSW}, and all \'{e}tale groupoids are assumed to be locally compact and Hausdorff. An \'etale groupoid $\cG$ is said to be effective if $\Interior(\cG')=\cG^{(0)}$, where $\cG'$ denotes the isotropy bundle of $\cG$, and $\cG$ is said to be non-wandering if for every clopen subset $U \subseteq \cG^{(0)}$, there exists $\gamma \in \cG$ such that $s(\gamma) \neq r(\gamma)$ and $s(\gamma), r(\gamma) \in U$ (see \cite[Definition 7.8]{NO}). Let $\cR$ denote the full equivalence relation on a countably infinite set. Given an \'{e}tale groupoid $\cG$ whose unit space is a locally compact Cantor set, we let $\fg{\cG}$ denote the topological full group of $\cG$, as defined in \cite[Definition~3.2]{NO}.

\begin{lemma} \label{lem:stabilisation}
If $\cG$ is an effective non-wandering \'etale groupoid whose unit space is a locally compact Cantor set, then the \'{e}tale groupoid $\cG \times \cR$ also has these properties.  
\end{lemma}
\begin{proof}
The unit space $(\cG \times \cR)^{(0)} = \cG^{(0)} \times \cR^{(0)}$ is a locally compact Cantor set since $\cR^{(0)}$ is a countably infinite set. Since $\cG$ is effective, we have 
\[ \Interior((\cG \times \cR)') = \Interior(\cG' \times \cR^{(0)}) = \Interior(\cG' ) \times \cR^{(0)} = \cG^{(0)} \times \cR^{(0)}, \] 
so that $\cG \times \cR$ is effective. We show that $\cG \times \cR$ is non-wandering. 
Let $U$ be a non-empty clopen subset of $\cG^{(0)} \times \cR^{(0)}$. Then, there exists a non-empty clopen subset $V$ of $\cG^{(0)}$ and $y \in \cR^{(0)}$ such that $V \times \{y\} \subseteq U$. Since $\cG$ is non-wandering, there exists $\gamma \in \cG$ such that $s(\gamma) \neq r(\gamma)$ and $s(\gamma), r(\gamma) \in V$. Let $\tilde{\gamma} := (\gamma,y) \in \cG \times \cR$. Then, we have $s(\tilde{\gamma}) \neq r(\tilde{\gamma})$ and $s(\tilde{\gamma}), r(\tilde{\gamma}) \in V \times \{y\}$. Hence, $\cG \times \cR$ is non-wandering. 
\end{proof}

Given a number field $K$, we let $\cG_K:=K^*\ltimes \Az_{K,f}$ be the transformation groupoid associated with the action $K^*\acts\Az_{K,f}$. Then, $\cG_K$ is \'etale, and there is a canonical isomorphism $\fA_K\cong C^*(\cG_K)$. 

\begin{lemma} \label{lem:nonwand}
For any number field $K$, the \'etale groupoid $\cG_K$ is effective and non-wandering. In addition, the unit space $\cG_K^{(0)}$ is a locally compact Cantor set. 
\end{lemma}
\begin{proof}
The unit space $\cG_K^{(0)} = \Az_{K,f}$ is clearly a locally compact Cantor set. We see that $\cG_K$ is effective, since $0 \in \Az_{K,f}$ is the unique point whose isotropy group is nontrivial with respect to the action $K^* \acts \Az_{K,f}$.  
We show that $\cG_K$ is non-wandering. 
Let $U \subseteq \Az_{K,f}$ be a compact open set and we show that there exists a $K^*$-orbit which meets $U$ at least twice. We may assume that $U$ is of the form 
\[ U = \prod_{\p \in F} (x_\p + \p^{k_\p} \cO_{K,\p}) \times \prod_{\p \not\in F} \cO_{K,\p}, \]
where $F \subseteq \cP_K$ is a finite subset, $k_\p \geq 1$, and $x_\p \in K_\p$ for $\p \in F$. By Lemma~\ref{lem:approx}, there exists $g_0 \in \cO_K \setminus\{0\}$ such that $v_\p(g_0) \geq k_\p-v_\p(x_\p)$ for $\p \in F$. Let $g=1+g_0 \in \cO_K$. Then, we have $g \neq 0,1$, and
\[ gx_\p = x_\p + g_0x_\p \in x_\p + \p^{k_\p}\cO_{K,\p}. \]
Let $x \in \Az_{K,f}$ be the adele whose $\p$-th coordinate is equal to $x_\p$ for $\p \in F$ and is equal to $1$ for $\p \not\in F$. Then, we have $x,gx \in U$ and $x\neq gx$, which complete the proof. 
\end{proof}

\begin{proof}[Proof of Corollary~\ref{cor:fullgroup}]
Suppose $\fg{\cG_K \times \cR}$ and $\fg{\cG_L \times \cR}$ are isomorphic. Since all orbits of  $\cG_K \times \cR$ and $\cG_L \times \cR$ are infinite, and the other assumptions in \cite[Theorem~7.10]{NO} are satisfied by Lemma~\ref{lem:stabilisation} and Lemma~\ref{lem:nonwand}, 
we have $\cG_K \times \cR \cong \cG_L \times \cR$. By \cite[Theorem~3.2]{CRS}, $\cG_K$ and $\cG_L$ are Kakutani equivalent, so that full corners of $C^*(\cG_K)$ and $C^*(\cG_L)$ are isomorphic. Now the claim follows from Theorem~\ref{thm:mainA}. 
\end{proof}

\begin{remark} \label{rmk:BC}
For a number field $K$, let $\cH_K$ denote the Bost--Connes groupoid of $K$. 
Based on a variant of the strong approximation \cite[Lemma~3.5]{Bruce}, the same proof as Lemma~\ref{lem:nonwand} shows that $\cH_K$ is non-wandering. Hence, combining arguments in this section and \cite[Theorem~1.1]{KT}, we conclude that the topological full group $\fg{\cH_K \times \cR}$ is also a complete invariant of the number field $K$. 
\end{remark}

\section{Explicit C*-algebraic descriptions of number-theoretic invariants}
\label{sec:explicit}
Throughout this section, let $K$ denote a number field.
We now show that several classical invariants of $K$ can be explicitly expressed in terms of C*-algebras.

\subsection{The ideal class group}\label{sec:classgroup}

Proposition~\ref{prop:propertygamma} provides a description of the ideal class group of $K$ in terms of C*-algebras. Let 
\[ D \colon \cU(B_K^\emptyset)/\cU_0(B_K^\emptyset) \to \Bigl( \Zz[1]_0 \Bigr)^{\oplus \cP_K},\ x \mapsto \Bigl( D_K^{\{\p\}}(x) \Bigr)_{\p \in \cP_K}.\]
\begin{proposition}
The ideal class group $\Cl_K$ is canonically isomorphic to $\Coker D$. 
\end{proposition}
\begin{proof}
We identify $\Gamma_K$ with $\cU(B_K^\emptyset)/\cU_0(B_K^\emptyset)$. As shown in the proof of Proposition~\ref{prop:propertygamma}, we have $D_K^{\{\p\}}(x)=-v_\p(x)[1]_0$ for all $x\in \Gamma_K$. Therefore, the image of $D$ coincides with $P:=\{(v_\p(x))_\p : x\in K^*\}$. Since we have the prime decomposition $(a)=\prod_\p \p^{v_\p(a)}$ for every $a \in K^*$, 
the subgroup $P$ is precisely the subgroup of principal (fractional) ideals of $K$ under the identification of $\Bigl( \Zz[1]_0 \Bigr)^{\oplus \cP_K}$ with the ideal group $\bigoplus_{\p\in\cP_K}\p^\Zz$ of $K$. Hence, $\Coker D$ is isomorphic to $\Cl_K$. 
\end{proof}

We point out that this is the first explicit C*-algebraic description of the ideal class group. In \cite{Li:Kt2}, the ideal class group is recovered from the $ax+b$-semigroup C*-algebra together with its canonical Cartan subalgebra. However, in both the semigroup C*-algebra and Bost--Connes C*-algebra cases, there is no known explicit description of the ideal class group in terms of the C*-algebra alone.

\subsection{The automorphism group}
\label{sec:AutK}
We prove that the automorphism group $\Aut(K)$ of $K$ is characterised as a quotient group of $\Aut(\fB_K)$. If $K$ is Galois over $\Qz$, then $\Aut(K)$ is the Galois group $\Gal(K/\Qz)$ by definition. Based on Lemma~\ref{lem:over2Sigma}, let 
\begin{align*}
    G_0 &= \{ \tilde{\alpha} \in \Aut(\fA_K) \colon \tilde{\alpha}^\emptyset \in \Aut(\fA_K^\emptyset) \mbox{ is homotopic to } \id\}, \\
    H_0 &= \{ \alpha \in \Aut(\fB_K) \colon \alpha^\emptyset \in \Aut(\fB_K^\emptyset) \mbox{ is homotopic to } \id\}. 
\end{align*}
Note that all approximately inner *-automorphisms of $\fA_K$ and $\fB_K$ belong to $G_0$ and $H_0$, respectively, since they induce the identity map on the commutative C*-algebras $\fA_K^\emptyset$ and $\fB_K^\emptyset$. 
\begin{proposition}
\label{prop:Aut(K)}
The group $\Aut(\fB_K)/H_0$ is canonically isomorphic to $\Aut(K)$. 
\end{proposition}
\begin{proof}
 We have a natural homomorphism $\Aut(K) \to \Aut(\fB_K)/H_0$, and Theorem~\ref{thm:main2} provides a map $\psi_1 \colon \Aut(\fB_K) \to \Aut(K)$. By construction and Remark~\ref{rmk:dual}, $\psi_1$ is a homomorphism and $H_0=\Ker \psi_1$. We can see that the composition $\Aut(K) \to \Aut(\fB_K)/H_0 \to \Aut(K)$ is the identity. Hence, these maps are inverse to each other.
 \end{proof}
This is also the first explicit C*-algebraic description of a Galois group in the investigation of C*-algebras from number theory. 

The group $\Aut(K)/G_0$ also has an interesting description in terms of number-theoretic invariants. 
Let $\lambda \colon \Aut(K) \acts \widehat{\mu_K}$ be the canonical action, that is, for every $\sigma \in \Aut(K)$, $a \in K^*$, and $\chi \in \widehat{\mu_K}$, we have
\[ (\lambda_\sigma(\chi))(a) = \chi(\sigma^{-1}(a)).\]

\begin{proposition}
The group $\Aut(\fA_K)/G_0$ is canonically isomorphic to the semidirect product $\widehat{\mu_K} \rtimes_\lambda \Aut(K)$. 
\end{proposition}
\begin{proof}
 Let $p \in \fA_K^\emptyset$ be the projection corresponding to the trivial character of $\mu_K$, and let $G$ be the subgroup of $\Aut(\fA_K)$ consisting of *-automorphisms $\tilde{\alpha}$ such that $\tilde{\alpha}^\emptyset$ fixes $p$. 
 Let $\Homeo(\widehat{K^*})$ be the group of self-homeomorphisms of $\widehat{K^*}$, let $\pi_0(\widehat{K^*})$ be the set of connected components of $\widehat{K^*}$, and let $\fS(\pi_0(\widehat{K^*}))$ be the permutation group of $\pi_0(\widehat{K^*})$. Let 
 \[ \tilde{\psi}_2 \colon \Aut(\fA_K) \to \Homeo(\widehat{K^*}) \to \fS(\pi_0(\widehat{K^*}))\]
 be the composition of the canonical homomorphisms, where the first map is obtained by sending $\tilde{\alpha} \in \Aut(\fA_K)$ to the Gelfand--Naimark dual of $\tilde{\alpha}^\emptyset \in \Aut(\fA_K^\emptyset)$, and the second map is obtained by sending a homeomorphism to the induced permutation of connected components. Since both $\Aut(\fA_K) \to \Homeo(\widehat{K^*})$ and $\Homeo(\widehat{K^*}) \to \pi_0(\widehat{K^*})$ are continuous, $\tilde{\psi}_2$ descends to a  homomorphism 
 \[ \psi_2 \colon \Aut(\fA_K)/G_0 \to \fS(\pi_0(\widehat{K^*})). \] 
 In addition, if $\{f_t\}_{t \in [0,1]} \subseteq \Homeo(\widehat{K^*})$ is a continuous path with $f_0=\id$, then $f_t$ acts trivially on $\pi_0(\widehat{K^*})$ for every $t \in [0,1]$. Hence, $G$ contains $G_0$. 
 
 By the the argument showing $\Aut(K) \cong \Aut(\fB_K)/H_0$ in the proof of Proposition~\ref{prop:Aut(K)}, combined with Proposition~\ref{prop:commutativeKL}, the natural homomorphism $\Aut(K) \to G/G_0$ is an isomorphism. For $\sigma \in \Aut(K)$, let $\tilde{\sigma} \in G$ denote the corresponding *-automorphism. Then, $\tilde{\sigma}$ is characterised by the property that $\tilde{\sigma}^\emptyset \in \Aut(C^*(K^*))$ satisfies $\tilde{\sigma}^\emptyset(u_a)=u_{\sigma(a)}$ for $a \in K^*$.

By Lemma~\ref{lem:over2Sigma}, if $\tilde{\alpha}_1, \tilde{\alpha}_2 \in \Aut(\fA_K)$ are homotopic, then $\tilde{\alpha}_1^\emptyset, \tilde{\alpha}_2^\emptyset \in \Aut(\fA_K)$ are homotopic. Hence, we have a canonical surjective homomorphism $\Aut(\fA_K)/\sim \to \Aut(\fA_K)/G_0$, where $\sim$ denotes the homotopy equivalence relation.
By Proposition~\ref{prop:factorKK2}, the dual action $\tilde{\tau} \colon \widehat{K^*} \to \Aut(\fA_K)$ factors through $\tau \colon \widehat{\mu_K} \to \Aut(\fA_K)/G_0$. The composition $\psi_2 \circ \tau \colon \widehat{\mu_K} \to \fS(\pi_0(\widehat{K^*}))$ coincides with the action by multiplication $\widehat{\mu_K} \acts \pi_0(\widehat{K^*})=\widehat{\mu_K}$, and thus it is faithful. Hence, the homomorphism $\tau$ is injective. We identify $\widehat{\mu_K}$ with a subgroup of $\Aut(\fA_K)/G_0$ via $\tau$.

Since the action $\psi_2 \circ \tau$ of $\widehat{\mu_K}$ is transitive, the group $\Aut(\fA_K)/G_0$ is generated by $\widehat{\mu_K}$ and $\Aut(K)$. In order to see the claim, it suffice to show that for every $\chi \in \widehat{K^*}$ and $\sigma \in \Aut(K)$,  $\tilde{\sigma}^\emptyset \tilde{\tau}_\chi^\emptyset (\tilde{\sigma}^\emptyset)^{-1}$ is homotopic to $\tilde{\tau}_{\chi'}^\emptyset$ in $\Aut(C^*(K^*))$, where $\chi' \in \widehat{K^*}$ is an extension of $\lambda_\sigma(\chi) \in \widehat{\mu_K}$. In fact, a direct computation shows $\tilde{\sigma}^\emptyset \tilde{\tau}_\chi^\emptyset (\tilde{\sigma}^\emptyset)^{-1}(u_a) = \chi(\sigma^{-1}(a))u_a$ for every $a \in K^*$, which completes the proof by Proposition~\ref{prop:factorKK2}. 
\end{proof}

\section{Concluding remarks}
\label{sec:remark}

Throughout this section, let $K$ denote a number field.

\subsection{Splitting numbers of rational primes and the Dedekind zeta function}\label{sec:zeta}
Following the strategy of \cite{Tak2}, we obtain the Dedekind zeta function as an invariant of the C*-algebra $\fA_K$ as follows:
Let $\p \in \cP_K$. We have seen that $\fA_K^{\{\p\}}$ is Morita equivalent to $B_K^{\{\p\}}$ (see Remark~\ref{rmk:varphiKF}). Using this, 
it follows that $\fA_K^{\{\p\}}$ has a unique (unbounded) trace $T_\p$, up to scaling. Then, we see that $(T_\p)_*(\K_0(\fA_K^{\{\p\}}))\cong \Zz[1/p]$, where $p$ is the rational prime lying under $\p$. Recall that the splitting number of a rational prime $p$ is the number of prime ideals in $\O_K$ lying above $p$. Therefore, for each rational prime $p$, the splitting number of $p$ is equal to the number of first level composition factors of $\fA_K$ for which the image of their $\K_0$-group under the unique map induced by the trace is isomorphic to $\Zz[1/p]$. Now the claim follows by \cite[Main Theorem]{SP}.

The splitting numbers of rational primes in \S~\ref{sec:zeta} and ideal class group in \S~\ref{sec:classgroup} are given entirely in terms of the zeroth and first levels, that is, they only require condition \hyperref[def:CN]{$(C_1)$}. On the other hand, the description of $\Aut(K)$ in \S~\ref{sec:AutK} involves the zeroth level, but is given in terms of the whole C*-algebra $\fA_K$, which requires \hyperref[def:CN]{$(C_3)$}. 
None of these invariants are complete invariants in general. If $K$ is Galois over $\Qz$, then $K$ is characterised by its Dedekind zeta function by \cite{P}, so in this case $K$ is (indirectly) characterized by the information in \S~\ref{sec:zeta}. However, for non-Galois extensions, the situation is much more complicated, see \cite[Theorems~3]{Pra} and the references therein. 
By \cite{P2}, the Dedekind zeta function and the ideal class group of a number field do not determine each other.

\subsection{Rigidity for group C*-algebras}
\label{sec:C*-rigidity}
We briefly explain an interpretation of our results as rigidity results for group C*-algebras of certain locally compact groups. Let $\Az_{K,f}\rtimes K^*$ be the semi-direct product with respect to the canonical action $K^*\acts\Az_{K,f}$. Then, using \cite[Example~3.16]{Will:book} and a choice of self-duality $\Az_{K,f}\cong \widehat{\Az_{K,f}}$ (as in \cite[Theorem~2.4]{CL:11}, but for the finite adele ring), one obtains a non-canonical isomorphism $C^*(\Az_{K,f}\rtimes K^*)\cong \fA_K$. Now, Theorem~\ref{thm:mainA} implies that the class of group C*-algebras $C^*(\Az_{K,f}\rtimes K^*)$ is rigid in the sense that $C^*(\Az_{K,f}\rtimes K^*)\cong C^*(\Az_{L,f}\rtimes L^*)$ if and only if $\Az_{K,f}\rtimes K^*\cong \Az_{L,f}\rtimes L^*$.

\subsection{Continuous orbit equivalence rigidity}
Our results imply that the dynamical systems $K^*\acts \Az_{K,f}$ are rigid in the sense that any two such systems $K^*\acts \Az_{K,f}$ and $L^*\acts \Az_{L,f}$ are continuously orbit equivalent if and only if they are conjugate. See \cite{Li:coe} for background on continuous orbit equivalence.

\subsection{Complete \K-theoretic invariants}
\label{sec:completeinvariants}
Similarly to the Bost--Connes C*-algebras \cite[Theorem~1.1 (6)]{KT}, we can construct a complete \K-theoretic invariant of number fields. Consider the invariant consisting of \K-groups $\K_*(\fB_K^F)$ for every $F \subseteq \cP_K$ with $0 \leq |F| \leq 3$, the boundary maps $\partial_K^{F,\p} \colon \K_*(\fB_K^F) \to \K_*(\fB_K^{F_\p})$ for every $F \subseteq \cP_K$ with $0 \leq |F| \leq 2$ and $\p \in \cP_K \setminus F$, the positive cones $\K_0(\fB_K^{\{\p\}})_+ \subseteq \K_0(\fB_K^{\{\p\}})$ for every $\p \in \cP_K$, and the unitary group modulo its connected component $\cU(\fB_K^\emptyset)/\cU_0(\fB_K^\emptyset) \subseteq \K_1(\fB_K^\emptyset)$. By the proof of Theorem~\ref{thm:main2}, this invariant is a complete invariant of the number field $K$. In \cite{KT}, we need \K-groups of composition factors for all finite subsets $F \subseteq \cP_K$. In our case, we only need \K-groups for finite subsets $F \subseteq \cP_K$ with cardinality less than $4$. Hence, 
one can say that the complexity of the invariant here is reduced. However, we need to specify the position of the unitary group in $\K_1(\fB_K^\emptyset)$. 

\subsection{C*-dynamical systems from totally positive elements}
\label{sec:totallypos}
The arguments until \S~\ref{sec:recon} can be also applied to various finite index subgroups of $K^*$. For example, let $\fC_K := C_0(\Az_{K,f})\rtimes K^*_+$, where $K_+^*$ is the subgroup of $K^*$ consisting of totally positive elements. 
If $K=\Qz$, then $\fC_\Qz$ is essentially the same as the original Bost--Connes C*-algebra for $\Qz$---the Bost--Connes C*-algebra for $\Qz$ is a full corner of $\fC_K$. In addition, if $K$ is totally imaginary, then we have $\fC_K=\fA_K$. 

Let $K$ and $L$ be number fields, and suppose full corners of $\fC_K$ and $\fC_L$ are isomorphic. By Lemma~\ref{lem:minproj}, either $K$ and $L$ are both totally imaginary, or both non-totally imaginary, since the torsion group of $K^*_+$ is trivial if and only if $K$ admits a real place. If we have the former case, then $K$ and $L$ are isomorphic by Theorem~\ref{thm:mainA}. If we have the latter case, the same argument until \S~\ref{sec:recon} shows that the dynamical systems $K^*_+ \acts \Az_{K,f}$ and $L^*_+ \acts \Az_{K,f}$ are conjugate. However, it remains open whether this conjugacy implies an isomorphism of $K$ and $L$, since the arguments in \S~\ref{sec:NT} do not work in this case. 

\subsection{Classifiability of the unital part of the composition factors}
\label{sec:classification}
Fix a nonempty finite subset $F \subseteq \cP_K$. 
 Since $B_K^F$ is simple and has a unique tracial state, modifying the proof of \cite[Theorem~4.1]{BLPR}, one sees that $B_K^F$ is an AH-algebra of slow dimension growth. By \cite[Theorem~2]{BDR}, $B_K^F$ has real rank zero. Therefore, $B_K^F$ is classified by its Elliott invariant by \cite[Theorem~9.4]{DG}. See \cite[Chapter~3,~\S~3.3]{RSt} for more on the classification of AH-algebras. We leave it as an open problem to calculate the K-groups of $B_K^F$.

\end{document}